\documentclass{article}

\usepackage[english]{babel}
\usepackage[T1]{fontenc}

\usepackage[letterpaper,top=3cm,bottom=3cm,left=2cm,right=2cm,marginparwidth=2cm]{geometry}

\usepackage{amsmath}
\usepackage{graphicx}
\usepackage{amsfonts}
\usepackage{amsthm}
\usepackage{comment}
\usepackage{slashed}
\usepackage[colorinlistoftodos]{todonotes}
\usepackage[colorlinks=true, allcolors=blue]{hyperref}
\usepackage{amssymb}
\usepackage{float}
\usepackage{tikzfill}
\usepackage{lipsum}
\allowdisplaybreaks
\newtheorem{lemma}{Lemma}[section]
\newtheorem{theorem}{Theorem}[section]
\newtheorem*{theorem*}{Theorem}
\newtheorem{definition}{Definition}[section]
\newtheorem{remark}{Remark}[section]
\newtheorem{corollary}{Corollary}[section]
\newtheorem{proposition}{Proposition}[section]
\newtheorem*{proposition*}{Proposition}
\newcommand{\tr}{\text{tr}}
\numberwithin{equation}{section}

\title{Systems of Wave Equations on Asymptotically de Sitter \\ Vacuum Spacetimes in All Even Spatial Dimensions}
\author{Serban Cicortas\footnote{Princeton University, Department of Mathematics, Fine Hall, Washington Road, Princeton, NJ 08544, USA}}

\begin{document}

\maketitle

\begin{abstract}
    This is the second paper of a two part work that establishes a definitive quantitative nonlinear scattering theory for asymptotically de Sitter vacuum solutions $\big(\mathcal{M},g\big)$ in $(n+1)$ dimensions with $n\geq4$ even, which are determined by small scattering data at $\mathcal{I}^{\pm}.$ In this paper we prove quantitative estimates for systems of wave equations on the $\big(\mathcal{M},g\big)$ backgrounds. The systems considered include the Einstein vacuum equations commuted with suitable time-dependent vector fields, where we treat the nonlinear terms as general inhomogeneous factors. The estimates obtained are essential in establishing sharp top order estimates in \cite{Cmain} for the scattering map of the Einstein vacuum equations, taking asymptotic data at $\mathcal{I}^-$ to asymptotic data at $\mathcal{I}^+$.
\end{abstract}

\section{Introduction}

In this work, we study systems of wave equations on asymptotically de Sitter vacuum backgrounds and we establish quantitative estimates. The simplest setting considered is given by the scalar wave equation on the background of de Sitter space:
\[\square_{g_{dS}}\phi=0.\]
More generally, we consider time-dependent background metrics in the class of asymptotically de Sitter vacuum spacetimes, and we allow for inhomogeneous terms on the right hand side of the equations. 

We apply the results of this paper in \cite{Cmain} for the Einstein vacuum equations with a positive cosmological constant, which can be written as a system of wave equations with the nonlinear terms treated as general inhomogeneous factors. We note that the current work also generalizes our previous results in \cite{linearwave}, where we established a scattering theory for the wave equation on exact de Sitter space.

We define below asymptotically de Sitter vacuum solutions, which represent the background spacetimes for the systems of wave equations considered in this paper. We state the main results in Theorems~\ref{main theorem first system} and \ref{main theorem second system}. We then explain the relation of the main theorems to the nonlinear scattering theory proved in \cite{Cmain}.

\paragraph{Asymptotically de Sitter vacuum solutions.} We consider the $(n+1)$-dimensional Einstein vacuum equations with positive cosmological constant $\Lambda=\frac{n(n-1)}{2},$ for any $n\geq 4$ even:
\begin{equation}\label{Einstein vacuum eq}
    Ric_{\mu\nu}-\frac{1}{2}Rg_{\mu\nu}+\Lambda g_{\mu\nu}=0.
\end{equation}
We consider asymptotically de Sitter solutions of \eqref{Einstein vacuum eq} arising from scattering data at past infinity $\mathcal{I}^-$. These are vacuum spacetimes $\big(\mathcal{M},g\big)$ of the form $\mathcal{M}=(0,\tau_0)\times S^n$ for some $\tau_0>0,$ which can be written with respect to coordinates $\{\theta^A\}$ associated to an arbitrary chart on $S^n$ as follows:
\begin{equation}\label{general form for metric g}
        g=-\frac{d\tau^2}{\tau^2}+\frac{1}{4\tau^2}\slashed{g}_{AB}(\tau,\theta^1,\ldots,\theta^n)d\theta^Ad\theta^B.
\end{equation}
Moreover, we also require that the spacetime $\big(\mathcal{M},g\big)$ is determined by scattering data at $\mathcal{I}^-=\{\tau=0\}$ consisting of a Riemannian metric $\big(S^n,\slashed{g}_0\big)$ and a symmetric traceless 2-tensor $h$ on $S^n$, which satisfies additional constraints.

The standard example of an asymptotically de Sitter vacuum solution is given by the $(n+1)$-dimensional de Sitter space $\big((0,\infty)\times S^n,g_{dS}\big)$. Denoting by $\slashed{g}_{S^n}$ the round metric on $S^n,$ we have:
\begin{equation}\label{de Sitter metric}
    g_{dS}=-\frac{d\tau^2}{\tau^2}+\frac{1}{4\tau^2}\slashed{g}_{dS}=-\frac{d\tau^2}{\tau^2}+\frac{1}{4\tau^2}\bigg(\frac{1}{2}+2\tau^2\bigg)^2\slashed{g}_{S^n}.
\end{equation}
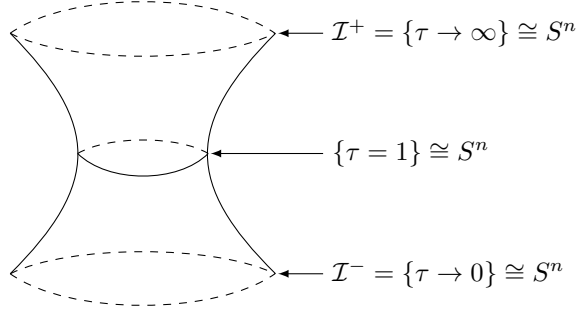
\begin{figure}
    \centering
    \begin{tikzpicture}[scale=0.8]
    \draw (5,2)  node[anchor=west] {$\{\tau=1\}\cong S^n$};
    \draw (4.21,0) .. controls (2.7,1.49) and (2.7,2.49) .. (4.21,4);
    \draw[dashed] (0.921,2) .. controls (1.43,2.3) and (2.59,2.3) .. (3.09,2);
    \draw[dashed] (-0.2,4) .. controls (0.79,3.49) and (3.2,3.49) .. (4.2,4);
    \draw (-0.2,0) .. controls (1.3,1.49) and (1.3,2.49) .. (-0.2,4);
    \draw[latex-] (3.13,2) -- (5,2);
    \draw[dashed] (-0.2,0) .. controls (0.79,-0.7) and (3.2,-0.7) .. (4.2,0);
    \draw[latex-] (4.25,0) -- (5,0);
    \draw (5,0)  node[anchor=west] {$\mathcal{I}^-=\{\tau\rightarrow0\}\cong S^n$};
    \draw[dashed] (-0.2,4) .. controls (0.79,4.7) and (3.2,4.7) .. (4.2,4);
    \draw[latex-] (4.25,4) -- (5,4);
    \draw (5,4)  node[anchor=west] {$\mathcal{I}^+=\{\tau\rightarrow\infty\}\cong S^n$};
    \draw (0.921,2) .. controls (1.43,1.5) and (2.59,1.5) .. (3.09,2);
    \draw[dashed] (-0.2,0) .. controls (0.79,0.49) and (3.2,0.49) .. (4.2,0);
    \end{tikzpicture}
    \caption{Diagram of the the $(n+1)$-dimensional de Sitter space $\big((0,\infty)\times S^n,g_{dS}\big)$.}
    \label{fig:main-thm}
\end{figure}

According to the results in \cite{Fefferman-Graham, ambient-metric}, \cite{selfsimilarvacuum}, and \cite{hintz}, we have that for any $n\geq4$ even and any smooth scattering data $\big(\slashed{g}_0,h\big)$ there exists a unique asymptotically de Sitter vacuum solution satisfying the expansion at $\tau=0$:
\begin{equation}\label{expansion in terms of  tau}
    \slashed{g}=\slashed{g}_0+\tau^2\slashed{g}_1+\ldots+\frac{\tau^{n-2}}{(n/2-1)!}\slashed{g}_{n/2-1}+\frac{2\tau^{n}\log(2\tau)}{(n/2)!}\mathcal{O}+\frac{\tau^{n}}{(n/2)!}k+O\big(\tau^{n+2}|\log\tau|\big),
\end{equation}
where the tensors $\slashed{g}_1,\ldots,\slashed{g}_{n/2-1},\mathcal{O},\tr k$ are determined by $\slashed{g}_0$ through certain compatibility relations, $h$ is the trace free part of $k,$ and the higher order terms in the expansion are determined by $\slashed{g}_0$ and $h$.

An essential feature that creates significant challenges in our work is the lack of smoothness in $\tau$ of the expansion \eqref{expansion in terms of  tau} due to the term $\mathcal{O},$ called the obstruction tensor. On the other hand, in the case of odd spatial dimension $n$ the corresponding expansion is smooth in $\tau,$ which results in considerable simplifications.

\paragraph{Assumptions on the background spacetime.} Before introducing the systems of wave equations and stating the main results, we fix the assumptions made on the background spacetime $\big(\mathcal{M},g\big)$ for the rest of the paper. 

We fix $\epsilon>0$ to be small enough and $N>0$ to be a large integer depending on $n$. We consider $\big(\mathcal{M},g\big)$ to be a smooth asymptotically de Sitter vacuum solution of the form \eqref{general form for metric g} determined by scattering data $\big(\slashed{g}_0,h\big)$, with $\mathcal{M}=(0,1]\times S^n.$ \footnote{In the context of \cite{Cmain}, we will always work with globally defined spacetimes with $\mathcal{M}=(0,\infty)\times S^n$ which are quantitatively close to de Sitter space. However, in view of the time inversion transformation $\tau\rightarrow1/\tau$, it suffices to restrict to $\mathcal{M}=(0,1]\times S^n$ for the purpose of this paper.} We assume that $\big(\mathcal{M},g\big)$ is close to de Sitter space, in the sense that:
\begin{equation}\label{preliminary estimates for background metric}
    \sum_{i=0}^N\sup_{\mathcal{M}}\big|\mathcal{L}_{\theta}^i\big(\slashed{g}-\slashed{g}_{dS}\big)\big|\leq\epsilon,
\end{equation}
where $\slashed{g}_{dS}$ is defined in \eqref{de Sitter metric}, we denote by $\mathcal{L}_{\theta}^i$ all the combinations of $i$ Lie angular derivatives in a coordinate chart, and we sum over a family of coordinate charts that covers $S^n.$

We use the convention that the absolute value $|\cdot|$, the integration volume form, and the covariant angular derivatives are taken with respect to the conformal metric $\slashed{g}$ induced on $S_{\tau}=\{\tau\}\times S^n$. We also denote by $H^k(S_{\tau})$ the $L^2$-based Sobolev space of order $k$ defined with respect $\slashed{g}$, which is defined in detail in Section \ref{LP section}.

We denote by $\chi=\mathcal{L}_{e_4}\slashed{g}$ the renormalized second fundamental form of $S_{\tau}$ with $e_4=\frac{1}{2\tau}\partial_{\tau},$ and we denote the schematic quantities $\psi\in\{1,\slashed{g},\chi,\tau^2\chi\}$. We assume that for some $C_0>0$ we have the bound:
\begin{equation}\label{preliminary estimates for background psi}
    \sup_{\tau\in(0,1]}\big\|\chi\big\|_{H^N(S_{\tau})}+\sup_{\tau\in(0,1]}\big\|\nabla_{\tau}\chi\big\|_{H^N(S_{\tau})}\leq C_0.
\end{equation}

Finally, we say that a tensor is horizontal on $\mathcal{M}$ if its entries are in $TS_{\tau}$.

\paragraph{The model systems.} We introduce two model systems of wave equations on the $\big(\mathcal{M},g\big)$ background. These are motivated by \cite[Section~5]{Cmain} and they include the Einstein vacuum equations \eqref{Einstein vacuum eq} commuted $n/2$ times with the vector field $e_4=\frac{1}{2\tau}\partial_{\tau}.$ According to the expansion \eqref{expansion in terms of  tau}, commuting at this order captures the singular behavior of the solution caused by the obstruction tensor $\mathcal{O}.$
\begin{definition}[Model Systems]\label{model systems definition}
    We fix the integers $I>0$ and $\sigma\in\{1,2\}.$ Let $\Phi_0,\ldots,\Phi_I$ be smooth horizontal tensors on $\mathcal{M}$, which satisfy the expansions at $\mathcal{I}^-$ for all $1\leq i\leq I$:
    \begin{align*}
        &\Phi_0=2\mathcal{O}\log\tau+h+O\big(\tau^2|\log\tau|^2\big),\ \nabla_{\tau}\Phi_0=\frac{2\mathcal{O}}{\tau}+O\big(\tau|\log\tau|^2\big)\text{ in }C^{\infty}(S^n)\\
        &\Phi_i=\Phi_i^0+O\big(\tau^2|\log\tau|^2\big),\ \nabla_{\tau}\Phi_i=O\big(\tau|\log\tau|^2\big)\text{ in }C^{\infty}(S^n).
    \end{align*}
    Based on these expansions, we say that $\Phi_0$ is a singular quantity and that $\Phi_1,\ldots,\Phi_I$ are regular quantities at $\mathcal{I}^-$. 
    
    We assume that $\Phi_0,\ldots,\Phi_I$ satisfy the model system of wave equations for any $0\leq m\leq M$:
    \begin{align}
        &\nabla_{\tau}\big(\nabla_{\tau}\nabla^m\Phi_0\big)+\frac{1}{\tau}\nabla_{\tau}\nabla^m\Phi_0-4\Delta_{\slashed{g}}\nabla^m\Phi_0=\psi\nabla^{m+1}\Phi+F_{m}^{0}\label{equation for Phi 0} \\
        &\nabla_{\tau}\big(\nabla_{\tau}\nabla^m\Phi_i\big)+(-1)^{\sigma-1}\cdot\frac{1}{\tau}\nabla_{\tau}\nabla^m\Phi_i-4\Delta_{\slashed{g}}\nabla^m\Phi_i=\sum_{j=\sigma-1}^I\psi\nabla^{m+1}\Phi_j+F_{m}^{i},\label{equation for Phi i}
    \end{align}    
    where $\psi\nabla^{m+1}\Phi$ is a schematic notation for sums of general terms with any $\psi\in\{1,\slashed{g},\chi,\tau^2\chi\}$, $\Phi\in\{\Phi_0,\ldots,\Phi_I\}$, the inhomogeneous terms satisfy $F_{m}^{0},F_{m}^{i}\in L^1_\tau\big([0,1]\big)C^{\infty}(S^n),$ and the covariant angular derivatives $\nabla$ are taken with respect to the conformal metric $\slashed{g}$ induced on $S_{\tau}$ according to \eqref{general form for metric g}. 
    
    We refer to the above system with the choice of sign $\sigma=1$ as the \textbf{first model system}, and similarly in the case $\sigma=2$ as the \textbf{second model system}.
\end{definition}

An additional motivation for the model systems which does not rely on the computations for the Einstein equations is provided by the wave equation on asymptotically de Sitter spacetimes. Let $\phi:\mathcal{M}\rightarrow\mathbb{R}$ be a smooth solution of: 
\begin{equation}\label{the wave eq}
    \square_{g}\phi=0.
\end{equation}
According to \cite{microlocal}, $\phi$ satisfies a similar expansion to \eqref{expansion in terms of  tau}. Commuting the wave equation up to $n/2$ times with the vector field $e_4$, we set $\Phi_0=(e_4)^{n/2}\phi,$ $\Phi_i=(e_4)^{n/2-i}\phi,$ for $1\leq i\leq I=n/2,$ and we obtain that $\big(\Phi_0,\ldots,\Phi_I\big)$ satisfy both model systems. In particular, this contains the case of the wave equation on exact de Sitter space considered in our previous work \cite{linearwave}.

As in the case of the above wave equation example, we prove in \cite[Section~5]{Cmain} that the commuted Einstein vacuum equations also satisfy both model systems with suitable inhomogeneous terms. We consider two model systems since in the case of the regular quantities $\nabla^m\Phi_i$ we have flexibility on which equations to use when proving estimates. On the other hand, we notice that for the singular quantities $\nabla^m\Phi_0$ we always work with \eqref{equation for Phi 0}.

\paragraph{The main results.} We prove quantitative estimates for the two model systems on the background spacetime $\big(\mathcal{M},g\big)$. In order to obtain sharp estimates, we consider a renormalization of the tensor $h$ in the expansion of $\Phi_0:$
\begin{equation}\label{renormalization of h}
\mathfrak{h}=h-2(\log\nabla)\mathcal{O},
\end{equation}
where the operator $\log\nabla$ is defined in Section \ref{LP section} using the geometric Littlewood-Paley theory of \cite{geometricLP}. We note that the need to renormalize $h$ is already present in the simplified model problem considered in \cite{linearwave} for the scattering for the linear wave equation on exact de Sitter space.

For solutions of the first model system, we prove estimates at $\tau\in(0,1]$ in terms of the asymptotic data at $\mathcal{I}^-$:
\begin{theorem}[First model system]\label{main theorem first system}
    Let $M>N$ be large enough. We assume that $\Phi_0,\ldots,\Phi_I$ satisfy the first model system on the background $\big(\mathcal{M},g\big),$ with $\sigma=1$ in \eqref{equation for Phi i}. We also assume that the sphere $\big(S^n,\slashed{g}_0\big)$ satisfies for some $C_1>0:$
    \[\big\|\slashed{Riem}(\slashed{g}_0)\big\|_{H^{M}(S^n)}\leq C_1.\]
    For all $\tau\in(0,1],$ we define the energy of the solution on $S_{\tau}$ to be:
    \[\mathcal{E}_I(\tau)=\tau^2\big\|\nabla_{\tau}\nabla^M\Phi_0\big\|^2_{H^{1/2}}+\tau^2\big\|\nabla^M\Phi_0\big\|^2_{H^{3/2}}+\sum_{i=1}^I\Big(\tau\big\|\nabla_{\tau}\nabla^M\Phi_i\big\|^2_{H^{1/2}}+ \tau\big\|\nabla^{M+1}\Phi_i\big\|^2_{H^{1/2}}+\big\|\Phi_i\big\|^2_{H^{M+1}}\Big).\]
    We define the asymptotic data norm and the inhomogeneous norm as:
    \[\mathcal{D}_I=\big\|\mathcal{O}\big\|^2_{H^{M+1}}+\big\|\mathfrak{h}\big\|^2_{H^{M+1}}+\sum_{i=1}^I\big\|\Phi_i^0\big\|^2_{H^{M+1}},\ \mathcal{F}_I(\tau)=\sum_{m=0}^{M}\sum_{i=0}^I\int_0^{\tau}\big\|F_m^i\big\|_{L^2}^2d\tau'+\sum_{i=0}^I\int_0^{\tau}\tau'\big\|F_M^i\big\|_{H^{1/2}}^2d\tau'.\]
    The solution of the first model system satisfies the estimates for some constant $C_I=C_I\big(M,C_0,C_1\big)>0$:
    \begin{equation}\label{main estimate first model system}
        \mathcal{E}_I(\tau)\leq C_I\mathcal{D}_I+C_I\mathcal{F}_I(\tau),
    \end{equation}
    \begin{equation}\label{practical estimate first model system}
        \big\|\Phi_0\big\|^2_{H^{M+1}}\leq C_I\mathcal{D}_I+C_I\mathcal{F}_I(\tau)+C_I|\log\tau|^2\big\|\mathcal{O}\big\|^2_{H^{M+1}}.
    \end{equation}
\end{theorem}

For solutions of the second model system, we prove estimates at $\tau\in(0,1)$ in terms of the solution at $\tau=1$. We also prove estimates for the asymptotic data at $\mathcal{I}^-$.
\begin{theorem}[Second model system]\label{main theorem second system}    
    Let $M>N$ be large enough. We assume that $\Phi_0,\ldots,\Phi_I$ satisfy the second model system on the background $\big(\mathcal{M},g\big),$ with $\sigma=2$ in \eqref{equation for Phi i}. We also assume that on $S_{1}=\{1\}\times S^n$ we have for some $C_2>0:$
    \[\big\|\psi\big\|_{H^{M+1}(S_{1})}\leq C_2.\]
    For all $\tau\in(0,1],$ we denote by $\mathcal{E}_{II}(\tau)$ the energy of the solution on $S_{\tau}$, which is defined in detail in~\eqref{definition of E II}. We have schematically at top order with respect to angular derivatives:
    \[\mathcal{E}_{II}(\tau)=\tau^2\big\|\Phi_0\big\|_{H^{M+3/2}}^2+\tau^2\big\|\nabla_{\tau}\nabla^M\Phi_0\big\|_{H^{1/2}}^2+\sum_{i=1}^I\big\|\Phi_i\big\|_{H^{M+3/2}}^2+\sum_{i=1}^I\big\|\nabla_{\tau}\nabla^M\Phi_i\big\|_{H^{1/2}}^2+\ldots\]
    We define the asymptotic data norm and the inhomogeneous norm as:
    \[\mathcal{D}_{II}=\sum_{i=0}^I\big\|\Phi_i\big\|_{H^{M+3/2}(S_{1})}^2+\sum_{i=0}^I\big\|\nabla_{\tau}\Phi_i\big\|_{H^{M+1/2}(S_{1})}^2,\ \mathcal{F}_{II}(\tau)=\sum_{i=0}^I\sum_{m=0}^{M}\int_{\tau}^1\tau'\big\|F_m^i\big\|_{H^{1/2}}^2d\tau'.\]
    The solution of the second model system satisfies the estimate for some constant $C_{II}=C_{II}\big(M,C_0,C_2\big)>0$:
    \begin{equation}\label{main estimate second model system}
        \mathcal{E}_{II}(\tau)\leq C_{II}\mathcal{D}_{II}+C_{II}\mathcal{F}_{II}(\tau).
    \end{equation}
    Moreover, the asymptotic data at $\mathcal{I}^-$ given by $\Phi_i^0,\ \mathcal{O},$ and $\mathfrak{h}$ satisfies the estimates:
    \begin{equation}\label{main estimate asymptotic data}
        \big\|\mathcal{O}\big\|_{H^{M+1}}^2+\sum_{i=1}^I\big\|\Phi_i^0\big\|^2_{H^{M+1}}\leq C_{II}\mathcal{D}_{II}+C_{II}\mathcal{F}_{II}(0),
    \end{equation}
    \begin{equation}\label{main estimate asymptotic data h}
        \big\|\mathfrak{h}\big\|_{H^{M+1}}^2\leq C\Big(C_{II},\big\|\slashed{Riem}(\slashed{g}_0)\big\|_{H^{M}}\Big)\cdot\Big(\mathcal{D}_{II}+\mathcal{F}_{II}(0)+\big\|h\big\|_{L^2}^2\Big).
    \end{equation}
\end{theorem}

\begin{remark}
    Solutions of the wave equation \eqref{the wave eq} on the background $\big(\mathcal{M},g\big)$ satisfy both model systems, where we set $\Phi_0=(e_4)^{n/2}\phi,$ $\Phi_i=(e_4)^{n/2-i}\phi,$ for $1\leq i\leq I=n/2.$ In particular, Theorems~\ref{main theorem first system} and \ref{main theorem second system} generalize our previous results in \cite{linearwave}, where we proved a scattering theory for the wave equation on exact de Sitter space. We use \cite{linearwave} as a guideline for the current paper.
\end{remark}

\begin{remark}\label{1/2 gain remark}
    For both model systems, the solution at $\tau=1$ and the asymptotic data at $\mathcal{I}^-$ naturally lie in different Sobolev spaces. In particular, solutions of the first model system "gain" $1/2$ derivatives at $\tau=1$ compared to the asymptotic data. Similarly, the asymptotic data of solutions of the second model system "lose" $1/2$ derivatives compared to the solution at $\tau=1$. This behavior is already present at the level of the linear wave equation on exact de Sitter space in \cite{linearwave}, and is explained below using the asymptotics of the Bessel functions in Section~\ref{toy problem section}. We also point out that capturing this structure of the solutions, together with the need to renormalize $h$ to $\mathfrak{h}$, requires decomposing the solutions using the geometric LP theory of \cite{geometricLP}.
\end{remark} 
\begin{remark}
    The first model system has a favorable structure for "forward" estimates, which consist of proving bounds for the solution at times $\tau\in(0,1]$ in terms of the asymptotic data at $\mathcal{I}^-$. On the other hand, the second model system is more suitable for "backward" estimates, which consist of proving bounds for the solution at times $\tau\in(0,1)$ and for the asymptotic data at $\mathcal{I}^-$ in terms of the solution at $\tau=1$. While the equation \eqref{equation for Phi 0} for the singular quantity $\Phi_0$ is the same for both model systems, the flexibility in \eqref{equation for Phi i} simplifies our analysis of the quantities that are regular at $\mathcal{I}^-$. Our approach of using the model systems to prove sharp estimates for the scattering map in \cite{Cmain} is justified since the commuted Einstein vacuum equations satisfy both model systems. Moreover, we defined the norms in order to have $\mathcal{E}_I(1)\sim\mathcal{D}_{II},$ which allows us to combine the estimates for the two model systems in \cite{Cmain}.
\end{remark}

\paragraph{The nonlinear scattering theory.} In \cite{Cmain}, we established a definitive quantitative nonlinear scattering theory for asymptotically de Sitter solutions of the Einstein vacuum equations \eqref{Einstein vacuum eq}, which are determined by small scattering data at $\mathcal{I}^{\pm}$. We present this result and explain how it relies on the Theorems \ref{main theorem first system} and \ref{main theorem second system} proved in the present paper.

We introduce briefly some notation, and refer the reader to \cite{Cmain} for precise definitions. Given smooth scattering data $\big(\slashed{g}_0,h\big),$ we denote by $\Sigma\big(\slashed{g}_0,h\big)$ an asymptotic initial data set consisting of certain tensors that can be computed in terms of $\slashed{g}_0,h,$ and their derivatives. We note that two essential components of $\Sigma\big(\slashed{g}_0,h\big)$ are the obstruction tensor $\mathcal{O}$ and the renormalized tensor $\mathfrak{h}=h-2\big(\log\nabla\big)\mathcal{O}.$ In particular, we denote by $\Sigma_{dS}$ the asymptotic initial data set corresponding to de Sitter space. For any $M>0$, we define the asymptotic initial data norm $\big\|\Sigma\big(\slashed{g}_0,h\big)\big\|_M$ to be an $L^2$-based Sobolev type norm of order $M$ measuring closeness to the de Sitter data. For any $\varepsilon>0$, we define $B_{\varepsilon}^M\big(\Sigma_{dS}\big)$ to be the set of smooth $\varepsilon$-small asymptotic data of order $M$, given by the open ball of size $\varepsilon$ around $\Sigma_{dS}$ with respect to the asymptotic initial data norm.

We state the main result of \cite{Cmain}:
\begin{theorem}[{\cite[Theorem~1.1]{Cmain}}]\label{main theorem of main}
    For any even integer $n\geq4,$ we have a scattering theory for asymptotically de Sitter vacuum solutions determined by small data. For any $M>0$ large enough there exists $\varepsilon_0>0$ small enough, such that for any $0<\varepsilon\leq\varepsilon_0$ we have:
    \begin{enumerate}
    \item\textit{Existence and uniqueness of scattering states:} for any $\varepsilon$-small asymptotic data of order $M$ at $\mathcal{I}^-$ or $\mathcal{I}^+$ given by $\Sigma\big(\slashed{g}_0,h\big)\in B_{\varepsilon}^M\big(\Sigma_{dS}\big)$, there exists a unique smooth global solution $\big(\mathcal{M},g\big)$ of the form (\ref{general form for metric g}) to the Einstein vacuum equations (\ref{Einstein vacuum eq}) which remains quantitatively close to the de Sitter metric; 
    \item\textit{Asymptotic completeness:} any smooth solution of the Einstein vacuum equations (\ref{Einstein vacuum eq}) of the form (\ref{general form for metric g}), which is quantitatively close to the de Sitter metric at a finite time $\tau,$ exists globally and induces scattering data $\big(\slashed{g}_0,h\big)$ at $\mathcal{I}^-$ and $\big(\underline{\slashed{g}_0},\underline{h}\big)$ at $\mathcal{I}^+;$
    \item\textit{Existence of a scattering map with quantitative estimates:} there exists a constant $C_M>0$ independent of $\varepsilon,$ such that we have a well-defined scattering map taking the asymptotic data at $\mathcal{I}^-$ to asymptotic data at $\mathcal{I}^+:$
    \begin{equation}\label{def of scattering map intro}
        \mathcal{S}:B_{\varepsilon}^M\big(\Sigma_{dS}\big)\rightarrow B_{C_M\varepsilon}^M\big(\Sigma_{dS}\big),\ \mathcal{S}\Big(\Sigma\big(\slashed{g}_0,h\big)\Big)=\Sigma\big(\underline{\slashed{g}_0},\underline{h}\big).
    \end{equation}
    The scattering map is locally invertible and locally Lipschitz at $\Sigma_{dS},$ satisfying the quantitative estimate:
    \begin{equation}\label{main estimate for S map intro}
        \Big\|\mathcal{S}\big(\Sigma(\slashed{g}_0,h)\big)\Big\|_M\leq C_M\Big\|\Sigma\big(\slashed{g}_0,h\big)\Big\|_M.
    \end{equation}
    This estimate is sharp and avoids any "derivative loss", in the sense that we use the same Sobolev-type norm of order $M$ to measure the smallness of asymptotic data at $\mathcal{I}^{\pm}.$
\end{enumerate}
\end{theorem}

The first two statements of Theorem \ref{main theorem of main} are proved entirely in \cite{Cmain}. In particular, given small scattering data we obtain a global smooth solution $\big(\mathcal{M},g\big)$ with $\mathcal{M}=(0,\infty)\times S^n$, which is quantitatively close to the de Sitter metric. We use these bounds as preliminary estimates in order to justify our assumptions \eqref{preliminary estimates for background metric} and \eqref{preliminary estimates for background psi}.

The construction of the scattering map given by \eqref{def of scattering map intro} and the sharp estimate \eqref{main estimate for S map intro} rely on \cite[Theorem~7.1, Theorem 9.1, Theorem 9.2]{Cmain}, which follow from Theorems~\ref{main theorem first system} and \ref{main theorem second system} proved in the present paper. We treat the Einstein equations \eqref{Einstein vacuum eq} commuted with $e_4$ at top order using the model systems on the $\big(\mathcal{M},g\big)$ background, with the nonlinear terms contained in the inhomogeneous terms. Theorems \ref{main theorem first system} and \ref{main theorem second system} give sharp estimates for the top order quantities, which are then used in \cite{Cmain} to deal with the nonlinear terms and ultimately prove \eqref{main estimate for S map intro}. We note that the main estimates proved in the present paper are also illustrated in \cite[Section~7, Section~9]{Cmain} for a simplified toy problem.

We outline the structure of the remainder of the introduction. In Section \ref{outline proof intro}, we present the main steps of the proof. We illustrate the argument of \cite{linearwave} for a toy problem in Section \ref{toy problem section}. We then introduce the geometric Littlewood-Paley theory of \cite{geometricLP} in Section \ref{LP theory intro}. In Section \ref{first model system intro section}, we outline the proof of Theorem \ref{main theorem first system}. In Section \ref{second model system intro section}, we explain the main steps in the proof of Theorem \ref{main theorem second system}. Finally, we outline the structure of the rest of the paper in Section \ref{paper outline}.

\subsection{Outline of the proof}\label{outline proof intro}

\subsubsection{A toy problem on the background of de Sitter}\label{toy problem section}
We consider a toy problem for the model systems on the background of exact de Sitter space. This contains some important features of the general problem, and it allows us to illustrate the argument of \cite{linearwave}.

We consider $\alpha:(0,1]\times S^n\rightarrow\mathbb{R}$ to be a solution of:
\begin{equation}\label{model eq dS}
    \partial_{\tau}^2\alpha+\frac{1}{\tau}\partial_{\tau}\alpha-\Delta_{\slashed{g}_{S^n}}\alpha=f_1(\tau)\tau\partial_{\tau}\alpha+f_2(\tau)\alpha
\end{equation}
\[\alpha=\alpha_{\mathcal{O}}\log\tau+\alpha_{h}+O\big(\tau^2|\log\tau|^2\big),\]
where $|f_1|+|f_2|=O(1)$. Moreover, we assume that $\alpha$ is supported on spherical harmonics with corresponding eigenvalues in $[2^{2l},2^{2l+2}).$ We introduce the new time variable $t=2^l\tau$ and decompose the solution as $\alpha=\alpha_J+\alpha_Y,$ where $\alpha_J,\ \alpha_Y$ solve \eqref{model eq dS}, and satisfy the expansions:
\[\alpha_J=\alpha_{\mathfrak{h}}+O\big(t^2|\log t|^2\big),\ \alpha_Y=\alpha_{\mathcal{O}}\log t+O\big(t^2|\log t|^2\big),\]
with $\alpha_{\mathfrak{h}}=\alpha_{h}-\alpha_{\mathcal{O}}\log2^l.$ We refer to $\alpha_J,\ \alpha_Y$ as the regular and singular components of $\alpha.$ 

As suggested by our notation, $\alpha_J$ and $\alpha_Y$ satisfy similar bounds in terms of $t$ to the first and second Bessel functions $J_0,\ Y_0$. Thus, at $t=0$ we have $|\alpha_J|\lesssim1,\ |\alpha_Y|\lesssim|\log t|,$ while at $t=2^l$ we have $|\alpha_J|,|\alpha_Y|\lesssim t^{-1/2}.$ The asymptotic behavior at $t=0$ implies the need to renormalize $\alpha_{h}$ to $\alpha_{\mathfrak{h}},$ similarly to \eqref{renormalization of h}. Moreover, the asymptotic behavior of the Bessel functions for large $t$ implies that at $\tau=1$ we have $\alpha_J,\alpha_Y\sim2^{-l/2}.$ This determines the "gain" of 1/2 derivatives for the solution at $\tau=1$ compared to the asymptotic data at $\tau=0,$ similarly to Remark~\ref{1/2 gain remark}.

While the above discussion provides heuristics on the behavior of the solution, we note that the results of \cite{linearwave} do not rely on the theory of Bessel functions. Instead, we capture the quantitative properties of the solution by constructing frequency dependent multipliers and proving energy estimates separately in the low frequency regime $t\in(0,1]$ and in the high frequency regime $t\in[1,2^l].$ We explain this in detail below for the model systems.

\subsubsection{Geometric Littlewood-Paley projections}\label{LP theory intro}

According to Theorems \ref{main theorem first system} and \ref{main theorem second system}, a remarkable property of solutions of the model systems is that the asymptotic data at $\mathcal{I}^-$ and the solution at $\{\tau=1\}$ lie in Sobolev spaces which differ by 1/2 in terms of regularity. The detailed analysis needed to capture this property requires the use of Littlewood-Paley projections. These are used to define fractional Sobolev spaces, the $\log\nabla$ operator present in the renormalization of $h,$ and to construct the frequency dependent multipliers needed in our estimates. 

In the case of the toy problem considered above, we projected using the spherical harmonics decomposition. However, in the general setting the conformal metric $\slashed{g}$ induced by the background on the spheres $S_{\tau}$ has a nontrivial time dependence. As a result, we use the geometric Littlewood-Paley theory of Klainerman-Rodnianski from \cite{geometricLP}, which we introduce in detail in Section \ref{LP section}. We also use the methods of \cite{geometricLP, duhamelLP} to prove additional new results that are needed in our situation, which we outline below.

The geometric LP projections satisfy a series of standard properties, which in particular allow us to define fractional derivatives. For example, the operator $\log\nabla$ used in the renormalization \eqref{renormalization of h} is defined as:
\[(\log\nabla)F=\sum_{k\geq0}P_k^2F\cdot\log2^k.\]

Unlike the case of the toy problem, we encounter additional difficulties since the LP projections only satisfy $L^2$-almost orthogonality. For any two families of LP projections $P_k$ and $P'_l$ we have according to \cite{geometricLP}:
\[\big\| P_kP'_lF\big\|_{L^2}\lesssim2^{-4|k-l|}\cdot\big\|F\big\|_{L^2}.\]

The LP projections satisfy the finite band property $\|P_kF\|_{L^2}\lesssim2^{-k}\|\nabla\widetilde{P}_kF\|_{L^2},$ where $\widetilde{P}_k^2=P_k.$ Because of the $L^2$-almost orthogonality, bounding the RHS of this inequality using the $P_l$ projection operators creates dangerous terms with frequency higher than $k.$ Instead, we prove a novel refined Poincaré inequality for any $k\geq0$ and $\delta>0:$
\begin{equation}\label{introo refined Poincare inequality}
    \big\|P_kF\big\|_{L^2}^2\lesssim\frac{1}{\delta}2^{-2k}\big\|\nabla P_kF\big\|_{L^2}^2+\delta\sum_{0\leq l<k}2^{-9k+7l}\big\|\nabla P_lF\big\|_{L^2}^2+\delta^{-1}2^{-4k}\big\|F\big\|_{L^2}^2.
\end{equation}
The important feature of this inequality is that all the LP projection operators have the same symbol. Moreover, only the last term contains frequencies higher than $k$, but this is lower order due to the good $2^{-4k}$ weight.

Finally, we also prove bounds for the commutation error terms obtained due to the time dependence of the metric $\slashed{g}.$ For example, for a certain projection operator $\underline{\widetilde{P}}_k$ defined in Section \ref{LP section} in terms of $P_k$, we have:
\begin{equation}\label{introo LP bound}
        \big\|[\nabla_4,P_k]F\big\|_{L^2}\lesssim\big\|\underline{\widetilde{P}}_kF\big\|_{L^2}+2^{-k}\big\|F\big\|_{L^2}.
\end{equation}
While the presence of projection operators with different symbols in this inequality cannot be avoided, we point out that both sides of the inequality can be summed for $k\geq0,$ which will prove to be essential later.

\subsubsection{The proof of Theorem \ref{main theorem first system}}\label{first model system intro section}
The goal of Theorem \ref{main theorem first system} is to obtain estimates for solutions of the first model system at $\tau\in(0,1]$ in terms of the asymptotic data at $\mathcal{I}^-.$  We provide a detailed outline of the proof of Theorem \ref{main theorem first system}, and note the reader should use this section for assistance when reading the proof in Section~\ref{model system for direction section}. Our strategy is to adapt the approach in \cite{linearwave} to the current setting by using the geometric LP theory instead. We first decompose the solution into its singular and regular parts, which we treat separately in all the estimates. We then explain how to prove lower order estimates and top order estimates, which complete the proof of Theorem \ref{main theorem first system}.

\paragraph{Decomposition of $\Phi_0$.} The first model system has a favorable structure for proving estimates for the quantities $\Phi_1,\ldots,\Phi_I$ in terms of the asymptotic data at $\mathcal{I}^-,$ due to the choice of sign $\sigma=1$ in Definition~\ref{model systems definition}. Moreover, these quantities are regular at $\tau=0,$ according to their asymptotic expansions. While the equation for $\Phi_0$ has the same favorable structure, we encounter difficulties since $\Phi_0$ blows-up at $\mathcal{I}^-$ as $\log\tau.$ We isolate this singular behavior by defining the singular component of $\Phi_0$, which decouples from the rest of the system and can be treated separately. The remaining component of $\Phi_0$ is regular at $\mathcal{I}^-$ and can be treated similarly to the other regular quantities.

For each $m\leq M,$ we define the singular component $\big(\nabla^m\Phi_0\big)_Y$ to be the horizontal tensor solving:
\begin{align}\label{equation for Phi Y}
    &\nabla_{\tau}\big(\nabla_{\tau}\big(\nabla^m\Phi_0\big)_Y\big)+\frac{1}{\tau}\nabla_{\tau}\big(\nabla^m\Phi_0\big)_Y-4\Delta\big(\nabla^m\Phi_0\big)_Y=\psi\nabla\big(\nabla^m\Phi_0\big)_Y\\
    &\big(\nabla^m\Phi_0\big)_Y({\tau})=2\nabla^m\mathcal{O}\log({\tau})+2(\log\nabla)\nabla^m\mathcal{O}+ O\big({\tau}^2|\log({\tau})|^2\big),\ \nabla_{\tau}\big(\nabla^m\Phi_0\big)_Y({\tau})=\frac{2\nabla^m\mathcal{O}}{\tau}+ O\big({\tau}|\log({\tau})|^2\big).\notag
\end{align}
We prove the existence and uniqueness of the singular component in Section~\ref{existence of singular component section}. We define the regular component: 

\[\big(\nabla^m\Phi_0\big)_J=\nabla^m\Phi_0-\big(\nabla^m\Phi_0\big)_Y.\]
Making the renormalization $\mathfrak{h}_m=\nabla^mh-2(\log\nabla)\nabla^m\mathcal{O},$ we get that the regular component satisfies the equation:
\begin{align}\label{equation for Phi J}
    \nabla_{\tau}\big(\nabla_{\tau}\big(\nabla^m\Phi_0\big)_J\big)+\frac{1}{\tau}\nabla_{\tau}\big(\nabla^m\Phi_0\big)_J-4\Delta\big(\nabla^m\Phi_0\big)_J=\psi\nabla\big(\nabla^m\Phi_0\big)_J+\sum_{j=1}^I\psi\nabla^{m+1}\Phi_{j}+F_{m}^{0}\\
    \big(\nabla^m\Phi_0\big)_J({\tau})=\mathfrak{h}_m+ O\big({\tau}^2|\log({\tau})|^2\big),\ \nabla_{\tau}\big(\nabla^m\Phi_0\big)_J({\tau})=O\big({\tau}|\log({\tau})|^2\big)\text{ in }C^{\infty}(S^n).\notag
\end{align}

\paragraph{Lower order estimates.} In order to prove the needed refined estimates for the solution at top order, it is essential to first prove in Section~\ref{lower order estimates section} energy estimates that are lower order in terms of angular derivatives. We treat separately the singular and regular quantities. For the regular quantities we note that the favorable choice of sign $\sigma=1$ and the regularity at $\mathcal{I}^-$ allow us to use $\nabla_{\tau}$ as a multiplier to prove in Proposition~\ref{standard estimates regular components propositionn}:
\begin{equation}\label{lower order regular estimate intro}
    \sum_{m=0}^{M-1}\big\|\big(\nabla^m\Phi_0\big)_J\big\|^2_{H^{1}}+\sum_{i=1}^I\big\|\Phi_i\big\|^2_{H^{M}}\leq C_I\mathcal{D}_I+C_I\mathcal{F}_I(\tau).
\end{equation}
For the singular component, the lower order version of \eqref{practical estimate first model system} consists of proving in Proposition~\ref{practical estimate for Phi Y proposition}:
\begin{equation}\label{lower order singular estimate intro}
    \sum_{m=0}^{M-1}\big\|\big(\nabla^m\Phi_0\big)_Y\big\|^2_{H^{1}}\leq C_I\mathcal{D}_I+C_I|\log\tau|^2\big\|\mathcal{O}\big\|^2_{H^{M+1}}.
\end{equation}
To prove this, we further decompose for every $m<M:$ $\big(\nabla^m\Phi_0\big)_Y=\big(\nabla^m\Phi_0\big)_Y^1+\big(\nabla^m\Phi_0\big)_Y^2,$ where $\big(\nabla^m\Phi_0\big)_Y^1$ and $\big(\nabla^m\Phi_0\big)_Y^2$ are the solutions of (\ref{equation for Phi Y}) that satisfy the expansions:
\[\big(\nabla^m\Phi_0\big)_Y^1({\tau})=2\nabla^m\mathcal{O}\log({\tau})+ O\big({\tau}^2|\log({\tau})|^2\big),\ \big(\nabla^m\Phi_0\big)_Y^2({\tau})=2(\log\nabla)\nabla^m\mathcal{O}+ O\big({\tau}^2|\log({\tau})|^2\big).\]
The desired lower order estimate for the singular component follows using the $\nabla_{\tau}$ multiplier in the equations for $\big(\nabla^m\Phi_0\big)_Y^1/\log\tau$ and $\big(\nabla^m\Phi_0\big)_Y^2,$ similarly to the approach in \cite{linearwave}.

Additionally, we prove estimates for the commutator term $\mathcal{C}=\big(\nabla^M\Phi_0\big)_Y-\nabla\big(\nabla^{M-1}\Phi_0\big)_Y$ and $H^{1/2}$ estimates for $\nabla(\nabla^{M-1}\Phi_0)_Y$ and $\nabla_{\tau}(\nabla^{M-1}\Phi_0)_Y$ in Section~\ref{fractional estimates section}, which simplify the proof of the top order estimates.

\paragraph{Top order estimates for the regular components.} In order to capture the specific behavior of the regular components at top order, we consider the equations satisfied by the projections $P_k\big(\nabla^M\Phi_0\big)_J$ and $P_k\nabla^M\Phi_i$, for all $1\leq i\leq I$ and $k\geq0.$ The goal is to prove similar estimates to the ones for the toy problem in Section~\ref{toy problem section}, dictated by the asymptotics of the Bessel function $J_0.$

For each $P_k$ projection we consider the low frequency regime $\tau\in(0,2^{-k-1}]$ in Section~\ref{first system low fre estimates section}. We propagate the $L^2$ bounds satisfied by the asymptotic data at $\mathcal{I}^-$ using $\nabla_{\tau}$ as a multiplier to obtain schematically for $\tau\in(0,2^{-k-1}]$:
\[\big\|P_k\big(\nabla^M\Phi_0\big)_J\big\|^2_{H^1}+\sum_{i=1}^I\big\|P_k\nabla^M\Phi_i\big\|^2_{H^1}+\ldots\lesssim\big\|P_k\mathfrak{h}_M\big\|^2_{H^1}+\sum_{i=1}^I\big\|P_k\nabla^M\Phi_i^0\big\|^2_{H^1}+\sum_{i=0}^I\int_0^{\tau}\frac{1}{2^{k}}\big\|P_kF_M^i\big\|_{L^2}^2d\tau'+\ldots\]
Using the commutation estimates for LP projections, the error terms that contain $\Phi$ have good $2^{-k}$ weights.

For each $P_k$ projection we also consider the high frequency regime $\tau\in[2^{-k-1},1]$ in Section~\ref{first system high fre estimates section}. We prove boundedness for the energy obtained by using the multiplier $2^k\tau\nabla_{\tau}$ to get schematically for $\tau\in[2^{-k-1},1]:$
\begin{align*}
    2^k\tau\big\|P_k\big(\nabla^M\Phi_{0}\big)_J\big\|^2_{H^1}+\sum_{i=1}^I2^k\tau\big\|P_k\nabla^M\Phi_i\big\|^2_{H^1}+\ldots\lesssim\big\|P_k\big(\nabla^M\Phi_{0}\big)_J\big\|^2_{H^1}\big|_{\tau=2^{-k-1}}+\sum_{i=1}^I\big\|P_k\nabla^M\Phi_{i}\big\|^2_{H^1}\big|_{\tau=2^{-k-1}}+&\\
    +\sum_{i=0}^I\int_{2^{-k-1}}^{\tau}2^k\tau'\big\|P_kF_M^i\big\|_{L^2}^2d\tau'+\int_{2^{-k-1}}^{\tau}(\tau')^3\big\|\underline{\widetilde{P}}_k\nabla\big(\nabla^M\Phi_{0}\big)_J\big\|_{L^2}^2d\tau'+\ldots&
\end{align*}
It is essential that in this estimate we obtain a top order bulk term with favorable sign, which can be dropped from the estimate\footnote{On the other hand, we point out that in the context of the second model system, the bulk term obtained in the high frequency regime estimate for the singular component has an unfavorable sign, creating major difficulties.}. We also notice that to deal with the error terms we use refined LP commutation estimates such as \eqref{introo LP bound}. In particular, this creates error terms with different projection operators, such as the last term written above, which cannot be controlled directly at the level of the high frequency regime estimates.

To conclude the proof of the top order estimates for the regular components, we combine the high frequency regime and low frequency regime estimates in Section~\ref{first system combined estimates section}. We point out that this argument uses the top order estimates for the singular component, which are proved separately as we explain below. To bound the error terms containing different LP projection operators, we first need to sum the estimates obtained for all $k\geq0,$ and then use Gronwall. For the negative frequencies we use the lower order estimates from Section~\ref{lower order estimates section}. As a result, we obtain:
\[\tau\big\|(\nabla^M\Phi_{0}\big)_J\big\|^2_{H^{3/2}}+\sum_{i=1}^I\tau\big\|\nabla^{M+1}\Phi_i\big\|^2_{H^{1/2}}+\ldots\leq C_I\mathcal{D}_I+C_I\mathcal{F}_I(\tau).\]

\paragraph{Top order estimates for the singular component.} The singular component decouples from the rest of the system, so it can be treated independently of the regular quantities. We proved top order estimates for the singular component in \cite{Cmain}, in order to illustrate the methods of the present paper. According to \cite{Cmain}, we have:
\begin{align*}
    &\tau^2\big\|\nabla_{\tau}\big(\nabla^M\Phi_0\big)_Y\big\|^2_{H^{1/2}}+ \tau^2\big\|\nabla\big(\nabla^M\Phi_0\big)_Y\big\|^2_{H^{1/2}}\leq C_I\big\|\mathcal{O}\big\|^2_{H^{M+1}},\\
    &\big\|\big(\nabla^M\Phi_0\big)_Y\big\|^2_{H^{1}}\leq C_I\big(1+|\log\tau|^2\big) \big\|\mathcal{O}\big\|^2_{H^{M+1}}.
\end{align*}
In Section~\ref{singular component outline section} we explain the proof in \cite[Section~7]{Cmain}, which follows similar steps to the above argument for the regular components and it relies on the lower order estimates for the singular component in Section~\ref{lower order estimates section}. In the high frequency regime there is the additional simplification of not having inhomogeneous terms on the RHS. In the low frequency regime we need to account for the singular behavior of the solution at $\tau=0,$ so we prove estimates for $P_k\big(\nabla^M\Phi_0\big)_Y/\log(2^k\tau)$ instead. The error terms can be simplified significantly using the additional lower order estimates in Section~\ref{lower order estimates section}. Finally, we combine the low frequency and high frequency regime estimates in a similar way to the regular components.

\subsubsection{The proof of Theorem \ref{main theorem second system}}\label{second model system intro section}
In Theorem~\ref{main theorem second system}, we obtain estimates for solutions of the second model system at $\tau\in(0,1)$ in terms of the solution at $\tau=1.$ Additionally, we also prove estimates for the asymptotic data at $\mathcal{I}^-.$ In this section, we provide a detailed outline of the proof of Theorem~\ref{main theorem second system}, and refer the reader to this section for assistance when reading the proof in Section~\ref{model system back direction section}. Similarly to Section~\ref{first model system intro section}, our strategy is to use the geometric LP theory and adapt the approach in \cite{linearwave} to the current setting. We first explain how to prove estimates for the regular quantities. We then outline the lower order estimates and top order estimates for the singular quantities, and we also explain the estimates for the asymptotic data at $\mathcal{I}^-,$ completing the proof of Theorem \ref{main theorem second system}.

\paragraph{Estimates for the regular quantities.} The second model system has a favorable structure for proving estimates for the regular quantities $\Phi_1,\ldots,\Phi_I$ in terms of the initial data at $\tau=1,$ due to the choice of sign $\sigma=2$ in Definition~\ref{model systems definition}. Another key feature of the second model system is that the singular terms $\psi\nabla^{m+1}\Phi_0$ are absent from the RHS of \eqref{equation for Phi i}. Thus, the equations for the regular quantities decouple from the singular quantities, and we can estimate them separately.

The favorable choice of sign $\sigma=2$ allows us to use $\nabla_{\tau}$ as a multiplier in \eqref{equation for Phi i}. We obtain in Section~\ref{good quantities estimates subsection}:
\begin{equation}\label{good quantities estimate intro}
    \sum_{i=1}^I\big\|\Phi_i\big\|_{H^{M+3/2}}^2+\sum_{i=1}^I\big\|\nabla_{\tau}\nabla^M\Phi_i\big\|_{H^{1/2}}^2+\ldots\leq C_{II}\mathcal{D}_{II}+C_{II}\mathcal{F}_{II}(\tau).
\end{equation}
Using the expansions satisfied by the regular quantities at $\tau=0,$ we obtain for the asymptotic data at $\mathcal{I}^-$:
\begin{equation}\label{good quantities asymptotic data intro}
    \sum_{i=1}^I\big\|\Phi_i^0\big\|^2_{H^{M+1}}\leq C_{II}\mathcal{D}_{II}+C_{II}\mathcal{F}_{II}(0).
\end{equation}

\paragraph{Estimates for the singular quantities} The structure of the equation \eqref{equation for Phi 0} satisfied by the singular quantity $\Phi_0$ creates significant new challenges compared to the case of the regular quantities. Additionally, we must prove estimates consistent with the expansion of $\Phi_0$ at $\tau=0,$ and obtain suitable bounds for $\mathcal{O}$ and $\mathfrak{h}.$

We start with a lower order estimate proved in Section~\ref{back singular preliminary section}, which provides useful preliminary bounds for $\Phi_0$:
\begin{equation}\label{back singular preliminary intro}
    \sum_{m=0}^M\tau^2\big\|\nabla_{\tau}\nabla^m\Phi_0\big\|_{L^2}^2+\tau^2\big\|\Phi_0\big\|_{H^{M+1}}^2+\tau\big\|\Phi_0\big\|_{H^M}^2+\int_{\tau}^1\tau'\big\|\Phi_0\big\|_{H^{M+1}}^2d\tau'\leq C_{II}\mathcal{D}_{II}+C_{II}\mathcal{F}_{II}(\tau).
\end{equation}

The main part of the argument consists of proving sharp estimates for the top order quantity $\xi=\nabla^M\Phi_0.$ We use as a guideline the toy problem considered in \cite[Section~9]{Cmain}, where we studied the equation satisfied by $\xi=\nabla^M\Phi_0,$ but we dropped the terms $F'_M=F_M^0+\sum_{i=1}^I\psi\nabla^{M+1}\Phi_i$ in order to illustrate the main ideas.

Similarly to Section~\ref{first model system intro section}, we prove estimates for each projection $P_k\nabla^M\Phi_0,$ with $k\geq0.$ For this purpose, we consider separately the low frequency regime $k\geq x$ with $\tau\in[0,X2^{-k-1}],$ and the high frequency regime $k\geq x$ with $\tau\in[X2^{-k-1},1]$, where we introduced a large constant $X=2^{x+1}$ for technical reasons. We note already here that we can deal with the terms with $k<x$ using the preliminary estimate \eqref{back singular preliminary intro}.

In Section~\ref{low frequency backward section singular}, we consider the low frequency regime $k\geq x,$ $\tau\in[0,X2^{-k-1}],$ and we prove a similar estimate to \eqref{back singular preliminary intro} for $P_k\nabla^M\Phi_0:$
\[\tau^2\big\|P_k\nabla_{\tau}\xi\big\|_{L^2}^2+\tau^2\big\|\nabla P_k\xi\big\|_{L^2}^2\lesssim X^22^{-2k}\Big(\big\|P_k\nabla_{\tau}\xi\big\|^2_{L^2}+\big\|\nabla P_k\xi\big\|^2_{L^2}\Big)\Big|_{\tau=X2^{-k-1}}+C_X2^{-3k}\mathcal{F}_{II}(\tau)+\ldots\]
We note that the data terms at $\tau=X2^{-k-1}$ will be bounded later using the high frequency regime estimates.

In Section~\ref{prelim high frequency estimate singular section}, we consider the high frequency regime $\tau\in[X2^{-k-1},1]$, and we prove a similar estimate to the one in Section~\ref{first model system intro section}. Using the multiplier $2^k\tau\nabla_{\tau}$, we get:
\[2^k\tau\big\|P_k\xi\big\|^2_{H^1}+\ldots\lesssim2^k\big\|P_k\xi\big\|^2_{H^1}\big|_{\tau=1}
+\int_{\tau}^1\frac{2^k}{(\tau')^2}\big\|P_k\xi\big\|_{L^2}^2d\tau'+\int_{\tau}^{1}2^k\tau'\big\|P_kF_M'\big\|_{L^2}^2d\tau'+\ldots\]

The main difference compared to Section~\ref{first model system intro section} is the presence of the second term on the RHS above, which is a top order bulk term with an unfavorable sign, creating significant difficulties. We improve the above high frequency regime estimate and deal with this bad term in Section~\ref{improved high frequency estimates section}, which represents the main technical part of the paper. We use the refined Poincaré inequality \eqref{introo refined Poincare inequality} to bound the error term:
\[\int_{\tau}^1\frac{2^k}{(\tau')^2}\big\|P_k\xi\big\|_{L^2}^2d\tau'.\]
As a result, we obtain a sum of error terms on the RHS in the low frequency regime and high frequency regime. We bound the high frequency regime error terms using the novel Gronwall-like inequality in Lemma~\ref{grwonwall type lemma}. For the low frequency regime error terms we use the estimates of Section~\ref{low frequency backward section singular}, which in turn create a sum of error terms that we bound using the discrete Gronwall inequality.

Moreover, in the above estimates we also have error terms with different projection operators, obtained by using LP commutation estimates such as \eqref{introo LP bound}, similarly to Section~\ref{first model system intro section}. As before, we deal with these terms towards the end of our argument, when summing the estimates obtained for all $k\geq x.$

Finally, in Section~\ref{main estimate back direction proof section} we combine the improved high frequency regime estimates of Section~\ref{improved high frequency estimates section} with the low frequency regime estimates of Section~\ref{low frequency backward section singular} and the preliminary estimate \eqref{back singular preliminary intro}, to obtain the main estimate for $\Phi_0:$
\[\tau^2\big\|\Phi_0\big\|_{H^{M+3/2}}^2+\tau^2\big\|\nabla_{\tau}\nabla^M\Phi_0\big\|_{H^{1/2}}^2+\ldots\leq C_{II}\mathcal{D}_{II}+C_{II}\mathcal{F}_{II}(\tau).\]
Using the estimate for the regular quantities \eqref{good quantities estimate intro} as well, we conclude the proof of the main estimate \eqref{main estimate second model system} in Theorem~\ref{main theorem second system}.

\paragraph{Estimates for the asymptotic quantities} The final step in the proof of Theorem~\ref{main theorem second system} is showing the estimates \eqref{main estimate asymptotic data} and \eqref{main estimate asymptotic data h} for the asymptotic data at $\mathcal{I}^-$ in Section~\ref{asympt data estimates section}. In view of \eqref{good quantities asymptotic data intro}, to establish \eqref{main estimate asymptotic data} we prove:
\begin{equation}\label{estimate for O intro}
    \|\mathcal{O}\|_{H^{M+1}}^2\leq C_{II}\mathcal{D}_{II}+C_{II}\mathcal{F}_{II}(0).
\end{equation}
This estimate for the obstruction tensor follows by taking the limit $\tau\rightarrow0$ in the low frequency regime estimate above for each $k\geq x$, and using the expansion of $\Phi_0$ at $\mathcal{I}^-.$

The estimate \eqref{main estimate asymptotic data h} for $\mathfrak{h}$ is more involved. We first notice that it suffices to show:
\begin{equation}\label{estimate for h intro}
    \sum_{k\geq x}2^{2k}\big\|P_k\mathfrak{h}_M\big\|_{L^2}^2\leq C_{II}\mathcal{D}_{II}+C_{II}\mathcal{F}_{II}(0),
\end{equation}
where $\mathfrak{h}_M=\nabla^Mh-2(\log\nabla)\nabla^M\mathcal{O}.$ According to the expansion of $\nabla^M\Phi_0$ at $\mathcal{I}^-,$ we can prove the above bound using energy estimates for the equations satisfied by the quantities $\overline{\xi}_k=P_k\nabla^M\Phi_0-\tau\log(2^k\tau)P_k\nabla_{\tau}\nabla^M\Phi_0.$ To deal with the error terms obtained in this case, we use the previously established bounds for the low frequency regime and the high frequency regime.

\subsection{Outline of the paper}\label{paper outline}
We outline the structure of the paper. In Section~\ref{LP section} we introduce the geometric Littlewood-Paley theory of \cite{geometricLP}, and prove the additional new results needed in our analysis. In Section~\ref{model system for direction section}, we prove Theorem~\ref{main theorem first system} following the outline in Section~\ref{first model system intro section}. In Section~\ref{model system back direction section}, we prove Theorem~\ref{main theorem second system} following the strategy presented in Section~\ref{second model system intro section}.

\paragraph{Acknowledgements.} The author would like to acknowledge Igor Rodnianski and Mihalis Dafermos for their valuable advice in the process of writing this paper.

\section{Geometric Littlewood-Paley Theory}\label{LP section}

We introduce the geometric Littlewood-Paley theory of Klainerman-Rodnianski \cite{geometricLP}. We also use the methods of \cite{geometricLP, duhamelLP} to prove additional new results that are needed in the proofs of the sharp estimates in Theorems~\ref{main theorem first system} and \ref{main theorem second system}, as explained in Section~\ref{LP theory intro} of the Introduction.

\subsection{Bounds for the heat equation}\label{heat eq section}
The geometric Littlewood-Paley projections are defined using the heat flow. In this section, we introduce some standard properties of the heat equation based on \cite{geometricLP}, and we prove additional commutation estimates. We notice that the estimates depend on the bounds assumed on the background spacetime \eqref{preliminary estimates for background metric} and \eqref{preliminary estimates for background psi}. Unless otherwise noted, all the implicit constants present in the estimates depend only on the constant $C_0$.

For any tensor field $F$ on $S_{\tau}$, we denote by $U(z)F$ the solution on $[0,\infty)\times S_{\tau}$ to the heat equation:
\begin{equation}\label{heat equation}
    \partial_zU(z)F-\Delta U(z)F=0,\ U(0)F=F,
\end{equation}
where $\Delta$ is the Laplace-Beltrami operator on $\big(S_{\tau},\slashed{g}\big),$ for some $\tau\in[0,1].$ We notice that the operators $U$ are self-adjoint and form a semigroup.

We use the following estimates for the heat kernel of \cite{geometricLP}:

\begin{proposition*}[{\cite[Proposition~4.1]{geometricLP}}]
    We have the estimates for the operator $U(z)$:
    \begin{align*}
        \big\|U(z)F\big\|_{L^2}&\leq\|F\|_{L^2}\\
        \big\|\nabla U(z)F\big\|_{L^2}&\leq\|\nabla F\|_{L^2}\\
        \big\|\nabla U(z)F\big\|_{L^2}+\big\|U(z)\nabla F\big\|_{L^2}&\leq\sqrt{2}z^{-\frac{1}{2}}\|F\|_{L^2}\\
        \big\|\Delta U(z)F\big\|_{L^2}&\leq\frac{\sqrt{2}}{2}z^{-1}\|F\|_{L^2}.
    \end{align*}
\end{proposition*}

We prove additional estimates using the methods of \cite{geometricLP, duhamelLP}. In particular, we notice the importance of proving estimates for the commutation of the operator $U(z)$ and the $e_4$ vector field.

\textbf{Notation.} Unless otherwise noted, in this section we write $A\lesssim B$ for some quantities $A,B>0$ if there exists a constant $C>0$ depending only on the constant $N,C_0$ defined in the Introduction, such that $A\leq CB.$ 
\begin{lemma}\label{heat lemma} We have the estimates for the operator $U(z)$:
\begin{align}
    \big\|\nabla^mU(z)F\big\|_{L^2}&\lesssim C\big(\|\slashed{Riem}\|_{H^{m-2}}\big)\cdot\|F\|_{H^m}\label{heat equation 1}\\
    \big\|[U(z),\nabla]F\big\|_{L^2}&\lesssim(\sqrt{z}+z)\cdot\big\|F\big\|_{L^2}\label{heat equation 2}\\
    \big\|\nabla[U(z),\nabla]F\big\|_{L^2}&\lesssim(1+\sqrt{z})\cdot\big\|F\big\|_{L^2}\label{heat equation 4}\\
    \big\|[U(z),\nabla^m]F\big\|_{L^2}&\lesssim(\sqrt{z}+z) C\big(\|\slashed{Riem}\|_{H^{m-1}}\big)\cdot\big\|F\big\|_{H^{m-1}}\label{heat equation 3}\\
    \big\|\nabla[U(z),\nabla^m]F\big\|_{L^2}&\lesssim(\sqrt{z}+z) C\big(\|\slashed{Riem}\|_{H^m}\big)\cdot\big\|F\big\|_{H^{m}}\label{heat equation 3.5}\\
    \big\|[U(z),G]F\big\|_{L^2}&\lesssim(z+\sqrt{z})\big\|G\big\|_{W^{2,\infty}}\big\|F\big\|_{L^2}\label{heat equation 5}\\
    \big\|\nabla[U(z),G]F\big\|_{L^2}&\lesssim(1+\sqrt{z})\big\|G\big\|_{W^{2,\infty}}\big\|F\big\|_{L^2}\label{heat equation 6}\\
    \big\|[U(z),\nabla_4]F\big\|_{L^2}&\lesssim(1+z)\cdot\big\|F\big\|_{L^2}\label{heat equation 7}\\
    \big\|\nabla[U(z),\nabla_4]F\big\|_{L^2}&\lesssim(1+z)\cdot\big\|F\big\|_{H^1}\label{heat equation 8}\\
    \big\|\nabla^2U(z)F\big\|_{L^2}&\lesssim \big(1+z^{-3/4}\big)\cdot\big\|F\big\|_{H^{1/2}}.\label{heat equation 9}
\end{align}
\end{lemma}
\begin{proof}
    We prove (\ref{heat equation 1}) by induction, while considering separately the cases $m\leq2n$ and $m>2n.$ The case $m=1$ is proved in \cite{geometricLP}. We assume at first that (\ref{heat equation 1}) holds up to $m<2n,$ and prove it for $m+1.$ We notice that for any tensor $\phi$ we have:
    \[\|\nabla^2\phi\|_{L^2}\lesssim\|\Delta\phi\|_{L^2}+\|\slashed{Riem}\|_{L^{\infty}}\|\phi\|_{H^1}.\]
    As a result, we get:
    \[\|\nabla^{m+1}U(z)F\|_{L^2}\lesssim\|\Delta\nabla^{m-1}U(z)F\|_{L^2}+\|\slashed{Riem}\|_{L^{\infty}}\|\nabla^{m-1}U(z)F\|_{H^1}.\]
    Since $\|\slashed{Riem}\|_{L^{\infty}}\lesssim1$ by the bound on the background spacetime in \eqref{preliminary estimates for background metric}, the second term is controlled using the induction hypothesis.
    For the first term, we write using $[\Delta,U(z)]F=0$:
    \[\|\Delta\nabla^{m-1}U(z)F\|_{L^2}\lesssim\|\nabla^{m-1}\Delta U(z)F\|_{L^2}+\big\|[\Delta,\nabla^{m-1}]U(z)F\big\|_{L^2}\]
    \[\lesssim\|\nabla^{m-1}U(z)\Delta F\|_{L^2}+\sum_{i+2j=m-1}\big\|\nabla^i(\slashed{Riem}^{j+1}U(z)F)\big\|_{L^2}\lesssim\|F\|_{H^{m+1}},\]
    by using \eqref{preliminary estimates for background metric} again. This completes the first induction argument, and establishes (\ref{heat equation 1}) for all $m\leq2n.$ Next, for some $m\geq2n$ we assume that (\ref{heat equation 1}) holds up to $m,$ and prove it for $m+1.$ As before, we have:
    \[\|\nabla^{m+1}U(z)F\|_{L^2}\lesssim\|\nabla^{m-1}U(z)\Delta F\|_{L^2}+\big\|[\Delta,\nabla^{m-1}]U(z)F\big\|_{L^2}+C\big(\|\slashed{Riem}\|_{H^{m-2}}\big)\cdot\|F\|_{H^m}\]
    \[\lesssim C\big(\|\slashed{Riem}\|_{H^{m-2}}\big)\cdot\|F\|_{H^{m+1}}+\sum_{i+2j=m-1}\big\|\nabla^i(\slashed{Riem}^{j+1}U(z)F)\big\|_{L^2}.\]
    The second term in the above can be written as:
    \[\sum_{i+2j=m-1}\sum_{i_0+\ldots+i_{j+1}=i}\bigg\|\big(\nabla^{i_0}U(z)F\big)\prod_{l=1}^{j+1}\big(\nabla^{i_l}\slashed{Riem}\big)\bigg\|_{L^2}.\]
    For each terms of the sum, we bound the factor with the most number of angular derivatives in $L^2.$ The other factors have at most $(m-1)/2$ derivatives, so we bound them in $L^{\infty}$ and apply the Sobolev inequality to get:
    \[\|\nabla^{m+1}U(z)F\|_{L^2}\lesssim C\big(\|\slashed{Riem}\|_{H^{m-1}}\big)\cdot\|F\|_{H^{m+1}}.\]
    
    This completes the proof of (\ref{heat equation 1}). Next, we prove (\ref{heat equation 2}). By Duhamel's formula as in \cite{duhamelLP} we have:
    \[[U(z),\nabla]F=\int_0^zU(z-z')[\Delta,\nabla]U(z')Fdz'.\]
    Using the standard estimates for the heat kernel in \cite{geometricLP}, we get (\ref{heat equation 2}). The proof of (\ref{heat equation 4}) is similar. We also have using $[\Delta,\nabla^m]=\nabla[\Delta,\nabla^{m-1}]+[\Delta,\nabla]\nabla^{m-1}$:
    \[[U(z),\nabla^m]F=\int_0^zU(z-z')[\Delta,\nabla^m]U(z')Fdz'=\int_0^zU(z-z')\nabla\big(\slashed{Riem}\nabla^{m-1}+[\Delta,\nabla^{m-1}]\big)U(z')Fdz'+\]\[+\int_0^zU(z-z')\big(\nabla\slashed{Riem}\cdot\nabla^{m-1}U(z')F\big)dz'.\]
    
    This implies (\ref{heat equation 3}) using the bounds in (\ref{heat equation 1}), and by estimating the commutator term as in the proof of (\ref{heat equation 1}). The proof of (\ref{heat equation 3.5}) is similar. The proofs of (\ref{heat equation 5}) and (\ref{heat equation 6}) follow, since according to \cite{geometricLP}:
    \[[U(z),G]F=\int_0^zU(z-z')\big(\Delta G\cdot U(z')F+2\nabla G\cdot\nabla U(z')F\big)dz'.\]
    By Duhamel's formula as in \cite{duhamelLP} we have:
    \[[U(z),\nabla_4]F=\int_0^zU(z-z')[\Delta,\nabla_4]U(z')Fdz'.\]
    Using \eqref{commutation formula 2} and the standard estimates for the heat kernel in \cite{geometricLP}, we get (\ref{heat equation 7}). The proof of (\ref{heat equation 8}) is similar. Finally, to prove (\ref{heat equation 9}):
    \[\big\|\nabla^2U(z)F\big\|_{L^2}^2\lesssim\sum_{k\in\mathbb{Z}}\big\|P_k\nabla^2U(z)F\big\|_{L^2}^2\lesssim (1+z^{-1})\big\|F\big\|_{L^2}^2+\sum_{k\geq0}\big\|\nabla^2P_kU(z)F\big\|_{L^2}^2\lesssim\]\[\lesssim (1+z^{-1})\big\|F\big\|_{L^2}^2+\sum_{k\geq0}\big\|\nabla^2U(z)P_kF\big\|_{L^2}^2\lesssim (1+z^{-1})\big\|F\big\|_{L^2}^2+\sum_{k\geq0} z^{-1}\big\|P_kF\big\|_{L^2}\cdot z^{-1/2}\big\|\nabla P_kF\big\|_{L^2}\lesssim\]\[\lesssim (1+z^{-1})\big\|F\big\|_{L^2}^2+\sum_{k\geq0} z^{-3/2}2^k\big\|P_kF\big\|_{L^2}\big\|\underline{P}_kF\big\|_{L^2}\lesssim (1+z^{-1})\big\|F\big\|_{L^2}^2+z^{-3/2}\big\|F\big\|_{H^{1/2}}^2.\]
    We note that the proof of the last statement uses parts of Lemma \ref{Litt Paley lemma}. We included its statement in the heat equation bounds section for future convenience.
\end{proof}

In our proofs we will repeatedly use the commutation formulas in the following lemma. Their derivation is standard and is contained in Section 2 of \cite{Cmain}.
\begin{lemma} We have the commutation formulas:
    \begin{align}
        [\nabla,\nabla_4]\phi&=\nabla\chi\cdot\phi+\chi\cdot\nabla\phi,\label{commutation formula 1}\\
        [\Delta,\nabla_4]\phi&=\nabla(\chi\nabla\phi)+O\big(\|\chi\|_{W^{2,\infty}}(|\nabla\phi|+|\phi|)\big).\label{commutation formula 2}
    \end{align}
\end{lemma}

\subsection{Bounds for the LP projections}\label{lp bounds section}

In this section, we follow \cite{geometricLP} to define the LP projections using the heat flow. We then prove a series of additional bounds for the LP projections that are needed in our analysis.

The class $\mathcal{M}$ of smooth symbols defined in \cite{geometricLP} consists of smooth functions $m:[0,\infty)\rightarrow\mathbb{R},$ which decay at infinity and satisfy the vanishing moments property:
\[\int_0^{\infty}z^{k_1}\partial_z^{k_2}m(z)dz=0,\ |k_1|+|k_2|\leq K,\]
for large enough order $K>0.$ For any $m\in\mathcal{M}$, we set $m_k(z)=2^{2k}m(2^{2k}z).$ For any tensor field $F$ on $S_{\tau}$, we define the LP projections for $k\in\mathbb{Z}$:
\[P_kF=\int_0^{\infty}m_k(z)U(z)Fdz.\]
We refer the reader to Theorem 5.5 in \cite{geometricLP} for the fundamental properties of these operators, similar to the standard LP projections. We use the following estimates for the  LP projections of \cite{geometricLP}:
\begin{proposition*}[{\cite[Theorem~5.5, Remark~5.6]{geometricLP}}]
    For an arbitrary LP projection, and any smooth tensor $F$ we have:
    \begin{enumerate}
        \item Bessel inequality.
        \[\sum_{k\in\mathbb{Z}}\big\|P_kF\big\|_{L^2}^2\lesssim\big\|F\big\|_{L^2}^2\]
        \item Finite band property.
        \[\big\|\nabla P_kF\big\|_{L^2}\lesssim2^k\big\|F\big\|_{L^2},\ \big\|P_kF\big\|_{L^2}\lesssim2^{-k}\big\|\nabla F\big\|_{L^2}\]
        \[\big\|\Delta P_kF\big\|_{L^2}^2\lesssim2^{2k}\big\|F\big\|_{L^2},\ \big\|P_kF\big\|_{L^2}\lesssim2^{-2k}\big\|\Delta F\big\|_{L^2}\]
        \item $L^2$-almost orthogonality. For any two families of LP projections $P_k, \widetilde{P}_k$ we have:
        \[\big\| P_k\widetilde{P}_{k'}F\big\|_{L^2}\lesssim2^{-4|k-k'|}\cdot\big\|F\big\|_{L^2}\]
    \end{enumerate}
\end{proposition*}
Using the geometric LP projections, we can define fractional Sobolev spaces. We first state the following result of \cite{geometricLP}:
\begin{proposition*}[{\cite[Corollary~7.12]{geometricLP}}]
    For an arbitrary LP projection, $a\geq0$ and any smooth tensor $F,$ we have:
   \[\sum_{k\geq0}2^{2ak}\big\|P_kF\big\|^2_{L^2}\lesssim\big\|F\big\|_{H^{a}}^2.\]
   Moreover, if $\sum_kP_k^2=I$ and $a<4,$ then:
   \[\big\|F\big\|_{H^{a}}^2\lesssim\sum_{k\geq0}2^{2ak}\big\|P_kF\big\|^2_{L^2}+\big\|F\big\|_{L^2}^2.\]
\end{proposition*}
We use this result to give an equivalent definition of fractional Sobolev spaces:
\begin{definition}
    Let $P_k$ be a family of projections with $\sum_kP_k^2=I$. We write any $a\geq0$ as $a=[a]+\{a\},$ with $\{a\}\in[0,1).$ For any smooth tensor $F$ we define its Sobolev norm of order $a$ as:
    \[\big\|F\big\|_{H^{a}}^2:=\big\|F\big\|_{H^{[a]}}^2+\big\|\nabla^{[a]} F\big\|_{H^{\{a\}}}^2,\]
    where we define the $H^{\{a\}}$ Sobolev norm by:
    \[\big\|F\big\|_{H^{\{a\}}}^2:=\sum_{k\geq0}2^{2\{a\}k}\big\|P_kF\big\|^2_{L^2}+\big\|F\big\|_{L^2}^2.\]
\end{definition}

Following the ideas of \cite{geometricLP}, we prove additional commutation bounds. As before, we highlight the importance of proving estimates for the commutation with $e_4,$ which allow us to control the change of the projection operators in terms of the time $\tau.$ Moreover, we notice that the following bounds rely on our previous estimates for the heat flow in Lemma \ref{heat lemma}.
\begin{lemma}\label{Litt Paley lemma}
The LP projection operators satisfy the following bounds for $k\geq 0$:
\begin{align}
    [\nabla_t,P_k]&=\frac{t}{2^{2k-1}}[\nabla_4,P_k],\text{ where } t=2^{k}\tau\label{LP est 1}\\
    \big\|[\nabla^m,P_k]F\big\|_{L^2}&\lesssim2^{-k} C\big(\|\slashed{Riem}\|_{H^{m-1}}\big)\cdot\big\|F\big\|_{H^{m-1}}\label{LP est 2}\\
    \big\|\nabla[\nabla^m,P_k]F\big\|_{L^2}&\lesssim2^{-k} C\big(\|\slashed{Riem}\|_{H^{m}}\big)\cdot\big\|F\big\|_{H^{m}}\label{LP est 3}\\
    \big\|[\nabla_4,P_k]F\big\|_{L^2}&\lesssim\big\|F\big\|_{L^2}\label{LP est 4}\\
    \big\|\nabla[\nabla_4,P_k]F\big\|_{L^2}&\lesssim\big\|F\big\|_{H^1}\label{LP est 5}\\
    \big\|[P_k,G]F\big\|_{L^2}&\lesssim2^{-k}\big\|G\big\|_{W^{2,\infty}}\big\|F\big\|_{L^2}\label{LP est 6}\\
    \big\|\nabla[P_k,G]F\big\|_{L^2}&\lesssim\big\|G\big\|_{W^{2,\infty}}\big\|F\big\|_{L^2}.\label{LP est 7}
\end{align}
\end{lemma}
\begin{proof}
    The first formula follows since $\nabla_t=2^{-k}\nabla_{\tau}=2^{-k+1}\tau\nabla_4=2^{-2k+1}t\nabla_4.$ In order to prove (\ref{LP est 2}), we notice that using the bound in (\ref{heat equation 3}) we have that:
    \[\big\|[\nabla^m,P_k]F\big\|_{L^2}\leq\int_0^{\infty}|m_k|\cdot\big\|[\nabla^m,U(z)]F\big\|_{L^2}dz\lesssim 2^{-k} C\big(\|\slashed{Riem}\|_{H^{m-1}}\big)\cdot\big\|F\big\|_{H^{m-1}}.\]
    Similarly, we also have that:
    \[\big\|\nabla[\nabla^m,P_k]F\big\|_{L^2}\leq\int_0^{\infty}|m_k|\cdot\big\|\nabla[\nabla^m,U(z)]F\big\|_{L^2}dz\lesssim 2^{-k} C\big(\|\slashed{Riem}\|_{H^{m}}\big)\cdot\big\|F\big\|_{H^{m}}.\]
    Next, we have that:
    \[\big\|[\nabla_4,P_k]F\big\|_{L^2}\leq\int_0^{\infty}|m_k|\cdot\big\|[\nabla_4,U(z)]F\big\|_{L^2}dz\leq\big\|F\big\|_{L^2}.\]
    Similarly, we also have that:
    \[\big\|\nabla[\nabla_4,P_k]F\big\|_{L^2}\leq\int_0^{\infty}|m_k|\cdot\big\|\nabla[\nabla_4,U(z)]F\big\|_{L^2}dz\leq\big\|F\big\|_{H^1}.\]
    The proofs of (\ref{LP est 6}) and (\ref{LP est 7}) follow from (\ref{heat equation 5}) and (\ref{heat equation 6}).
\end{proof}

One disadvantage of the estimates (\ref{LP est 4}) and (\ref{LP est 5}) is that the right hand side cannot be summed in $k,$ whereas the left hand side can be summed because of the presence of the $P_k$ operator. We address this by proving a refined version of these estimates, which contains a different projection operator on the right hand side. The presence of different projection operators will pose additional difficulties in the analysis of the model systems. We notice that this issue is a consequence of the nontrivial time dependence of the metrics $\slashed{g}.$

We define the symbol $\widetilde{m}\in\mathcal{M}$ given by $\widetilde{m}(z)=zm(z),$ and we denote by $\widetilde{P}_k$ the associated projection operator. Moreover, we also introduce the projection operator $\underline{\widetilde{P}}_k$ which satisfies $\underline{\widetilde{P}}_k^2=\widetilde{P}_k.$

\begin{lemma}\label{Litt Paley lemma refined}
    We have the following estimates for $k\geq 0$:
    \begin{align}
        [\nabla_4,P_k]F&=2^{-2k}\chi\nabla^2\widetilde{P}_kF+O\Big(2^{-k}\big\|F\big\|_{L^2}\Big)\label{LP refined est 1}\\
        \big\|[\nabla_4,P_k]F\big\|_{L^2}&\lesssim\big\|\underline{\widetilde{P}}_kF\big\|_{L^2}+2^{-k}\big\|F\big\|_{L^2}\label{LP refined est 2}\\
        \nabla[\nabla_4,P_k]F&=2^{-2k}\chi\nabla^2\widetilde{P}_k\nabla F+O\Big(2^{-k}\big\|F\big\|_{H^1}\Big)\label{LP refined est 3}\\
        \big\|\nabla[\nabla_4,P_k]F\big\|_{L^2}&\lesssim \big\|\underline{\widetilde{P}}_k\nabla F\big\|_{L^2}+2^{-k}\big\|F\big\|_{H^1}\label{LP refined est 4}.
    \end{align}
\end{lemma}
\begin{proof}
    In order to prove (\ref{LP refined est 1}), we compute the following:
    \[[\nabla_4,P_k]F=\int_0^{\infty}m_k(z)[\nabla_4,U(z)]Fdz=\int_0^{\infty}m_k(z)\int_0^zU(z-z')[\nabla_4,\Delta]U(z')Fdz'dz\]
    \[=\int_0^{\infty}m_k(z)\int_0^zU(z-z')\nabla\big(\chi\nabla U(z')F\big)+\int_0^{\infty}m_k(z)\int_0^zU(z-z')\Big([\nabla_4,\Delta]U(z')F-\nabla\big(\chi\nabla U(z')F\big)\Big).\]
    We can bound the second term in $L^2$ by:
    \[\int_0^{\infty}|m_k(z)|\int_0^z\big\|U(z')F\big\|_{H^1}dz'dz\lesssim\big\|F\big\|_{L^2}\int_0^{\infty}|m_k(z)|\int_0^z(z')^{-1/2}dz'dz\lesssim2^{-k}\big\|F\big\|_{L^2}.\]
    Next, we write the first term as:
    \begin{equation}\label{some random equation two terms}
        \int_0^{\infty}m_k(z)\int_0^z\nabla U(z-z')\big(\chi\nabla U(z')F\big)+\int_0^{\infty}m_k(z)\int_0^z[U(z-z'),\nabla]\big(\chi\nabla U(z')F\big).
    \end{equation}
    The second term in (\ref{some random equation two terms}) can be bounded in $L^2$ by:
    \[\int_0^{\infty}|m_k(z)|\int_0^z(z-z')^{1/2}\big\|\chi\nabla U(z')F\big\|_{L^2}dz'dz\lesssim\big\|F\big\|_{L^2}\int_0^{\infty}|m_k(z)|\int_0^z(z-z')^{1/2}(z')^{-1/2}dz'dz\lesssim2^{-k}\big\|F\big\|_{L^2}.\]
    We write the first term in (\ref{some random equation two terms}) as:
    \[\int_0^{\infty}m_k(z)\int_0^z\nabla\big(\chi\cdot U(z-z')\nabla U(z')F\big)+\int_0^{\infty}m_k(z)\int_0^z\nabla[U(z-z'),\chi]\big(\nabla U(z')F\big).\]
    The latter term in the above is bounded in $L^2$ by:
    \[\int_0^{\infty}|m_k(z)|\int_0^z\big\|\nabla U(z')F\big\|_{L^2}dz'dz\lesssim\big\|F\big\|_{L^2}\int_0^{\infty}|m_k(z)|\int_0^z(z')^{-1/2}dz'dz\lesssim2^{-k}\big\|F\big\|_{L^2}.\]
    As a result, we proved that:
    \[[\nabla_4,P_k]F=\chi\cdot\int_0^{\infty}m_k(z)\int_0^z\nabla U(z-z')\nabla U(z')Fdz'dz+O\Big(2^{-k}\big\|F\big\|_{L^2}\Big)\]
    \[=\chi\cdot\nabla^2\int_0^{\infty}zm_k(z)U(z)Fdz+\chi\cdot\int_0^{\infty}m_k(z)\int_0^z\nabla[U(z-z'),\nabla]U(z')Fdz'dz+O\Big(2^{-k}\big\|F\big\|_{L^2}\Big)\]
    \[=2^{-2k}\chi\cdot\nabla^2\int_0^{\infty}\widetilde{m}_k(z)U(z)Fdz+O\Big(2^{-k}\big\|F\big\|_{L^2}\Big)=2^{-2k}\chi\nabla^2\widetilde{P}_kF+O\Big(2^{-k}\big\|F\big\|_{L^2}\Big).\]
    The proof of (\ref{LP refined est 2}) follows from (\ref{LP refined est 1}), since $\underline{\widetilde{P}}_k^2=\widetilde{P}_k.$

    We now prove (\ref{LP refined est 3}). We notice that using \eqref{commutation formula 2} we can write:
    \[\nabla[\nabla_4,P_k]F=\int_0^{\infty}m_k(z)\int_0^z\nabla U(z-z')[\nabla_4,\Delta]U(z')Fdz'dz\]
    \[=\int_0^{\infty}m_k(z)\int_0^z\nabla U(z-z')\big(\chi\nabla^2 U(z')F\big)+\int_0^{\infty}m_k(z)\int_0^z\nabla U(z-z')\Big([\nabla_4,\Delta]U(z')F-\chi\nabla^2 U(z')F\Big).\]
    The second term can be bounded in $L^2$ by:
    \[\int_0^{\infty}|m_k(z)|\int_0^z(z-z')^{-1/2}\big\|U(z')F\big\|_{H^1}dz'dz\lesssim2^{-k}\big\|F\big\|_{H^1}.\]
    The first term in the above expression can be written as:
    \[\chi\int_0^{\infty}m_k(z)\int_0^z\nabla U(z-z')\nabla^2 U(z')Fdz'dz+\nabla\chi\cdot\int_0^{\infty}m_k(z)\int_0^z U(z-z')\nabla^2 U(z')Fdz'dz+\]\[+\int_0^{\infty}m_k(z)\int_0^z\nabla[U(z-z'),\chi]\nabla^2 U(z')Fdz'dz.\]
    To complete the proof of (\ref{LP refined est 3}), we need to write this expression as $2^{-2k}\chi\nabla^2\widetilde{P}_k\nabla F+O\big(2^{-k}\|F\|_{H^1}\big)$. Indeed, we can bound the last two terms in $L^2$ by:
    \[\int_0^{\infty}|m_k(z)|\int_0^z\big\|\nabla^2 U(z')F\big\|_{L^2}dz'dz\lesssim\int_0^{\infty}|m_k(z)|\int_0^z\big\|\nabla[\nabla,U(z')]F\big\|_{L^2}+\int_0^{\infty}|m_k(z)|\int_0^z\big\|\nabla U(z')\nabla F\big\|_{L^2}\]
    \[\lesssim2^{-2k}\big\|F\big\|_{L^2}+2^{-k}\big\|\nabla F\big\|_{L^2}.\]
    Finally, the leading term can be written as:
    \[\chi\int_0^{\infty}m_k(z)\int_0^z\nabla U(z-z')\nabla U(z')\nabla F+\chi\int_0^{\infty}m_k(z)\int_0^z\nabla U(z-z')\nabla[\nabla,U(z')]F=\]
    \[=\chi\int_0^{\infty}m_k(z)\int_0^z\nabla U(z-z')\nabla U(z')\nabla F+O\Big(2^{-k}\big\|F\big\|_{L^2}\Big)=2^{-2k}\chi\nabla^2\widetilde{P}_k\nabla F+O\Big(2^{-k}\big\|F\big\|_{H^1}\Big),\]
    where in the last step we used the same argument as in the proof of (\ref{LP refined est 1}). This completes the proof of (\ref{LP refined est 3}). Moreover, the proof of (\ref{LP refined est 4}) follows from (\ref{LP refined est 3}) as before.
\end{proof}

We also prove estimates where we trade $1/2$ derivatives on $F$ for $2^{k/2}$ growth, which simplify certain error terms in the analysis of the top order singular component of the first model system in the low frequency regime:

\begin{lemma}\label{fractional LP bounds}
    We have the following estimates for $k\geq 0$:
    \begin{equation}\label{fractional LP est 1}
    \big\|[\nabla_4,P_k]\nabla F\big\|_{L^2}\lesssim2^{k/2}\big\|F\big\|_{H^{1/2}}
\end{equation}
\begin{equation}\label{fractional LP est 2}
    \big\|\nabla[\nabla_4,P_k] F\big\|_{L^2}\lesssim2^{k/2}\big\|F\big\|_{H^{1/2}}.
\end{equation}
\end{lemma}
\begin{proof} In the previous proof we obtained the identity:
\[[\nabla_4,P_k]\nabla F=\int_0^{\infty}m_k(z)\int_0^zU(z-z')\nabla\big(\chi\nabla U(z')\nabla F\big)+\]\[+\int_0^{\infty}m_k(z)\int_0^zU(z-z')\Big([\nabla_4,\Delta]U(z')\nabla F-\nabla\big(\chi\nabla U(z')\nabla F\big)\Big).\]
As a result, we get from Lemma \ref{heat lemma}:
\[\big\|[\nabla_4,P_k]\nabla F\big\|_{L^2}\lesssim\big\|F\big\|_{L^2}+\int_0^{\infty}|m_k(z)|\int_0^z\big((z-z')^{-1/2}+1\big)\big\|\nabla U(z')\nabla F\big\|_{L^2}dz'dz\]\[\lesssim\big\|F\big\|_{L^2}+\int_0^{\infty}|m_k(z)|\int_0^z\big((z-z')^{-1/2}+1\big)(z')^{-3/4}\big\|F\big\|_{H^{1/2}}dz'dz\lesssim\big\|F\big\|_{H^{1/2}}\bigg(1+\int_0^{\infty}|m_k(z)|z^{-1/4}dz\bigg).\]
A similar proof also gives:
    \[\big\|\nabla[\nabla_4,P_k] F\big\|_{L^2}\lesssim\big\|F\big\|_{L^2}+\int_0^{\infty}|m_k(z)|\int_0^z(z-z')^{-1/2}\big\|\nabla^2 U(z')F\big\|_{L^2}dz'dz\lesssim2^{k/2}\big\|F\big\|_{H^{1/2}}.\]
\end{proof}

\textbf{Convention.} For the remainder of the paper, we fix the projection operator $P_k$ to satisfy $\sum_kP_k^2=I.$ We notice that all the estimates established above in this section are valid for any LP projections with symbols in $\mathcal{M}$.

In order to control a top order bulk term with bad sign in the high frequency estimate for the second model system, we need a refined Poincaré inequality for LP projections. The key aspect of this result is that the projection operators on the right hand side have the same symbol as the one on the left hand side. Moreover, all the frequencies higher than $k$ are contained in the last term, which is lower order.
\begin{lemma} For any $k\geq0,$ and $\delta>0$, we have the inequality:
    \begin{equation}\label{ref Poincarey}
    \big\|P_kF\big\|_{L^2}^2\lesssim\frac{1}{\delta}2^{-2k}\big\|\nabla P_kF\big\|_{L^2}^2+\delta\sum_{0\leq l<k}2^{-9k+7l}\big\|\nabla P_lF\big\|_{L^2}^2+\delta^{-1}2^{-4k}\big\|F\big\|_{L^2}^2.
\end{equation}
\end{lemma}
\begin{proof}
    Let $\dot{P}$ be the projection operator with symbol given by $\dot{m}(z)=-\int_z^{\infty}m(z')dz'.$ According to the proof of \cite[Theorem~5.5 (v)]{geometricLP} we have that $2^{2k}P_kF=\Delta\dot{P}_kF.$ This implies the Poincaré inequality:
    \[\big\|P_kF\big\|_{L^2}\lesssim\sqrt{\delta}2^{-k}\big\|\nabla\dot{P}_kF\big\|_{L^2}+\frac{1}{\sqrt{\delta}}2^{-k}\big\|\nabla P_kF\big\|_{L^2}.\]
    Also, for all $l\geq k\geq 0$ we have the following estimate, according to the proof of \cite[Theorem~5.5 (ii)]{geometricLP}:
    \[\big\|P_l\dot{P}_kF\big\|_{L^2}\lesssim2^{-2(l-k)}\big\|P_kF\big\|_{L^2}.\]
    We use these two bounds, together with the other usual bounds for LP projections, to get:
    \[\big\|P_kF\big\|_{L^2}\lesssim\sqrt{\delta}2^{-k}\big\|\nabla \dot{P}_kF\big\|_{L^2}+\frac{1}{\sqrt{\delta}}2^{-k}\big\|\nabla P_kF\big\|_{L^2}\lesssim\sqrt{\delta}2^{-k}\big\|\dot{P}_k\nabla F\big\|_{L^2}+\frac{1}{\sqrt{\delta}}2^{-2k}\big\|F\big\|_{L^2}+\frac{1}{\sqrt{\delta}}2^{-k}\big\|\nabla P_kF\big\|_{L^2}\]\[\lesssim\sqrt{\delta}\sum_{l\in\mathbb{Z}}2^{-k}\big\|P_l^2\dot{P}_k\nabla F\big\|_{L^2}+\frac{1}{\sqrt{\delta}}2^{-2k}\big\|F\big\|_{L^2}+\frac{1}{\sqrt{\delta}}2^{-k}\big\|\nabla P_kF\big\|_{L^2}\]
    \[\lesssim \frac{1}{\sqrt{\delta}}2^{-k}\big\|P_k\nabla F\big\|_{L^2}+\sqrt{\delta}\sum_{0\leq l<k}2^{-k-4|k-l|}\big\|P_l\nabla F\big\|_{L^2}+\frac{1}{\sqrt{\delta}}2^{-2k}\big\|F\big\|_{L^2}+\frac{1}{\sqrt{\delta}}2^{-k}\big\|\nabla P_kF\big\|_{L^2}\]\[\lesssim\frac{1}{\sqrt{\delta}}2^{-k}\big\|\nabla P_kF\big\|_{L^2}+\sqrt{\delta}\sum_{0\leq l<k}2^{-|k-l|/2}\cdot2^{-k-7|k-l|/2}\big\|P_l\nabla F\big\|_{L^2}+\frac{1}{\sqrt{\delta}}2^{-2k}\big\|F\big\|_{L^2}.\]
    We square this inequality and use Cauchy-Schwarz in order to conclude.
\end{proof}

We introduce the $(\log\nabla)$ operator, essential for the renormalization of $h$ in \eqref{renormalization of h}. We define for any smooth tensor $F$ on $S_0$ and $k\geq0$:
\[(\log\nabla)F=\sum_{l\geq0}P_l^2F\cdot\log2^l.\]
We prove the following bound on the $(\log\nabla)$ operator, used in lower order estimates:
\begin{lemma}\label{bound on log nabla lemma} Set $\eta=1/10.$ For any smooth horizontal tensor $F,$ and any $s\geq0,$ we have the estimate:
    \[\big\|(\log\nabla)F\big\|_{H^s}\lesssim \big\|F\big\|_{H^{s+\eta}}.\]
\end{lemma}
\begin{proof} We first notice that we have the inequality:
    \[\big\|(\log\nabla)F\big\|_{L^2}^2\lesssim\bigg(\sum_{k\geq0}\big\|P_k^2F\cdot\log2^k\big\|_{L^2}\bigg)^2\lesssim \bigg(\sum_{k\geq0}2^{k\eta/2}\big\|P_kF\big\|_{L^2}\bigg)^2\lesssim \sum_{k\geq0}2^{2k\eta}\big\|P_kF\big\|_{L^2}^2\lesssim \big\|F\big\|_{H^{\eta}}.\]
    Using this, we can also bound the following:
    \[\big\|(\log\nabla)F\big\|_{H^s}^2\lesssim\big\|(\log\nabla)F\big\|_{L^2}^2+\sum_{l\geq0}2^{2sl}\big\|P_l(\log\nabla)F\big\|_{L^2}^2\lesssim\big\|(\log\nabla)F\big\|_{L^2}^2+\sum_{l\geq0}2^{2sl}\big\|(\log\nabla)P_lF\big\|_{L^2}^2\]\[\lesssim \big\|F\big\|_{H^{\eta}}+\sum_{l,k\geq0}2^{2sl}2^{2\eta k}\big\|P_kP_lF\big\|_{L^2}^2\lesssim\big\|F\big\|_{H^{\eta}}+\sum_{k\geq0}\sum_{l=0}^k2^{2(s+\eta)k}\big\|P_lP_kF\big\|_{L^2}^2+\sum_{l\geq0}\sum_{k=0}^l2^{2(s+\eta)l}\big\|P_kP_lF\big\|_{L^2}^2\]\[\lesssim\big\|F\big\|_{H^{\eta}}+\sum_{k\geq0}2^{2(s+\eta)k}\big\|P_kF\big\|_{L^2}^2\lesssim\big\|F\big\|_{H^{s+\eta}}.\]
\end{proof}

Next, we define the following operator for $k\geq0$:
\begin{equation}\label{definition of R k}
    R_kF=2P_k(\log\nabla)F-2\log2^k\cdot P_kF=2\sum_{l\geq0}\log2\cdot(l-k)\cdot P_kP_l^2F-2\sum_{l<0}\log2^k\cdot P_kP_l^2F.
\end{equation}
The operator $R_k$ appears as a commutation error term when projecting the expansion at $\{\tau=0\}$ of the singular component of $\Phi_0$ in the analysis of the first model system. 

We consider the projection operator $\underline{P}_k$ which satisfies $\underline{P}_k^2=P_k.$ We have the estimates for $R_k:$
\begin{lemma}\label{R k lemma} Let $F$ be any smooth tensor on $S_0$. We extend $R_kF$ to $(0,1]\times S^n$ to be independent of $\tau.$ We also denote $t=2^k\tau.$ Then, for any $k\geq0$ we have:
    \begin{align}
        \big\|\Delta_{\slashed{g}_{\tau}} R_kF\big\|_{L^2}&\lesssim2^{k}\big\|\underline{P}_kF\big\|_{H^1}\label{Rk est 1}\\
        \big\|\nabla R_kF\big\|_{L^2}&\lesssim\big\|\underline{P}_kF\big\|_{H^1}\label{Rk est 2}\\
        2^{k}\big\|R_kF\big\|_{L^2}&\lesssim \big\|\underline{P}_kF\big\|_{H^1}\label{Rk est 3}\\
        \big\|\nabla_tR_kF\big\|_{L^2}&\lesssim2^{-3k}t\big\|\underline{P}_kF\big\|_{H^1}\label{Rk est 4}\\
        \big\|\nabla\nabla_tR_kF\big\|_{L^2}&\lesssim2^{-2k}t\big\|\underline{P}_kF\big\|_{H^1}.\label{Rk est 5}
    \end{align}
\end{lemma}
\begin{proof} Because of \eqref{preliminary estimates for background metric}, we can express the LHS of (\ref{Rk est 1}) using derivatives at $\{\tau=0\}$:
\[2^{-k}\big\|\Delta_{\slashed{g}_{\tau}} R_kF\big\|_{L^2}\lesssim2^{-k}\big\|R_kF\big\|_{H^2(S_0)}\lesssim2^{-k}\big\|\Delta_{\slashed{g}_{0}} R_kF\big\|_{L^2(S_0)}+2^{-k}\big\|R_kF\big\|_{H^1(S_0)}\]
Using the finite band property of \cite{geometricLP} we have:
    \[2^{k}\big\|R_kF\big\|_{L^2(S_0)}\lesssim\sum_{l\geq0}2^{k}|l-k|\cdot \big\|\underline{P}_kP_lP_l\underline{P}_kF\big\|_{L^2}+\sum_{l<0}k2^{k}\cdot \big\|\underline{P}_kP_lP_l\underline{P}_kF\big\|_{L^2}\lesssim\]\[\lesssim\sum_{l\geq0}\frac{2^{k-l}|l-k|}{2^{2|l-k|}}\cdot \big\|\nabla\underline{P}_kF\big\|_{L^2}+\sum_{l<0}k2^{-k+2l}\cdot \big\|\underline{P}_kF\big\|_{L^2}\lesssim\big\|\underline{P}_kF\big\|_{H^1(S_0)},\]
    \[2^{-k}\big\|\Delta_{\slashed{g}_{0}}R_kF\big\|_{L^2}\lesssim\sum_{l\geq0}2^{-k}|l-k|\cdot \big\|\underline{P}_kP_l\Delta_{\slashed{g}_{0}} P_l\underline{P}_kF\big\|_{L^2}+\sum_{l<0}k2^{-k}\cdot \big\|\underline{P}_kP_l\Delta_{\slashed{g}_{0}} P_l\underline{P}_kF\big\|_{L^2}\lesssim\]\[\lesssim\sum_{l\geq0}\frac{2^{l-k}|l-k|}{2^{2|l-k|}}\cdot \big\|\nabla\underline{P}_kF\big\|_{L^2}+\sum_{l<0}k2^{-3k+l}\cdot \big\|\nabla\underline{P}_kF\big\|_{L^2}\lesssim\big\|\underline{P}_kF\big\|_{H^1(S_0)}.\]
    Combining these two estimates, we also have that:
    \[\big\|\nabla R_kF\big\|_{L^2(S_0)}\lesssim\big\|\underline{P}_kF\big\|_{H^1(S_0)}.\]
    So far we proved (\ref{Rk est 1}), (\ref{Rk est 2}), and (\ref{Rk est 3}). In order to prove \eqref{Rk est 4} and \eqref{Rk est 5}, we notice that for any horizontal $k$-tensor $\Phi$ we have the formula:
    \begin{equation}\label{formula for lie time derivative}
        \nabla_{\tau}\Phi=\mathcal{L}_{\tau}\Phi+\tau\chi\cdot\Phi.
    \end{equation}
    Using this, we get:
    \[\big\|\nabla_tR_kF\big\|_{L^2}\lesssim\big\|(\nabla_t-\mathcal{L}_t)R_kF\big\|_{L^2}\lesssim\frac{t}{2^{2k}}\big\|R_kF\big\|_{L^2}\lesssim\frac{t}{2^{3k}}\big\|\underline{P}_kF\big\|_{H^1(S_0)}.\]
    Finally, we have:
    \[\big\|\nabla\nabla_tR_kF\big\|_{L^2}\lesssim\big\|\nabla(\nabla_t-\mathcal{L}_t)R_kF\big\|_{L^2}\lesssim\frac{t}{2^{2k}}\big\|R_kF\big\|_{H^1}\lesssim\frac{t}{2^{2k}}\big\|\underline{P}_kF\big\|_{H^1(S_0)}.\]   
\end{proof}

Finally, we prove the following result which implies that it is equivalent whether we project the expansions at $\tau=0$ with respect to $\slashed{g}_{0}=\slashed{g}(0)$ or $\slashed{g}_{\tau}=\slashed{g}(\tau).$
\begin{lemma}\label{difference of projections lemma}
    We consider $F$ to be a smooth tensor on $S_0,$ extended to be independent of $\tau.$ Denote by $^{(\slashed{g}_{0})}P_k$ the projection with respect to $\big(S^n_{0},\slashed{g}_{0}\big)$ and by $^{(\slashed{g}_{\tau})}P_k$ the projection with respect to $\big(S^n_{\tau},\slashed{g}_{\tau}\big)$. Then we have the estimate for any $s\geq0,\ k\geq0$:
    \begin{equation}\label{difference of Pk}
        \big\|\big(\ ^{(\slashed{g}_{0})}P_k-\ ^{(\slashed{g}_{\tau})}P_k\big)F\big\|_{H^{s}}\lesssim_s\tau^2\big\|F\big\|_{H^{s+2}}.
    \end{equation} 
\end{lemma}
\begin{proof}
    The bound follows using Duhamel's formula as in \cite{duhamelLP}: 
    \[\big\|\big(\ ^{(\slashed{g}_{0})}P_k-\ ^{(\slashed{g}_{\tau})}P_k\big)F\big\|_{H^s}\lesssim\int_0^{\infty}|m_k|\cdot\big\|\ ^{(\slashed{g}_{0})}U(z)F-\ ^{(\slashed{g}_{\tau})}U(z)F\big\|_{H^s}dz\]
    \[\lesssim\int_0^{\infty}|m_k|\int_0^z\Big\|\ ^{(\slashed{g}_{0})}U(z-z')\big(\Delta_{\slashed{g}_{\tau}}-\Delta_{\slashed{g}_{0}}\big)\ ^{(\slashed{g}_{\tau})}U(z')F\Big\|_{H^s}dz'dz\lesssim_s\int_0^{\infty}|m_k|\int_0^z\tau^2\big\|\ ^{(\slashed{g}_{\tau})}U(z')F\big\|_{H^{s+2}}dz'dz,\]
    where we also used the expansion of $\slashed{g}_{\tau}$ at $\tau=0.$ The last term is bounded by $\tau^2\big\|F\big\|_{H^{s+2}}$.
\end{proof}

\section{The First Model System}\label{model system for direction section}

In this section we prove Theorem~\ref{main theorem first system}, obtaining estimates for solutions of the first model system at $\tau\in(0,1]$ in terms of the asymptotic data at $\mathcal{I}^-.$ We follow the strategy from Section~\ref{first model system intro section} of the Introduction, and we advise the reader to refer to the outline for assistance while reading the proof below.

We recall the decomposition of $\Phi_0$ into its singular and regular components. For each $m\leq M$ we have:
\[\nabla^m\Phi_0=\big(\nabla^m\Phi_0\big)_Y+\big(\nabla^m\Phi_0\big)_J,\]
where each singular component $\big(\nabla^m\Phi_0\big)_Y$ satisfies \eqref{equation for Phi Y} and decouples from the rest of the system. In Section~\ref{existence of singular component section} we prove the existence and uniqueness of the singular component.

We prove estimates separately for the regular quantities $(\Phi_0)_J,\Phi_1,\ldots,\Phi_I$ and the singular quantity $(\Phi_0)_Y.$ In Section~\ref{lower order estimates section} we prove the lower order estimates outlined in Section~\ref{first model system intro section}. We prove \eqref{lower order regular estimate intro} in Proposition~\ref{standard estimates regular components propositionn} and \eqref{lower order singular estimate intro} in Proposition~\ref{practical estimate for Phi Y proposition}. We also prove estimates for the commutator term $\mathcal{C}=\big(\nabla^M\Phi_0\big)_Y-\nabla\big(\nabla^{M-1}\Phi_0\big)_Y$ and $H^{1/2}$ estimates for $\nabla(\nabla^{M-1}\Phi_0)_Y$ and $\nabla_{\tau}(\nabla^{M-1}\Phi_0)_Y$ in Section~\ref{fractional estimates section}.

In Section~\ref{first system top order estimates section} we prove the top order estimates outlined in Section~\ref{first model system intro section}. We prove the low frequency regime estimates for the regular quantities in Section~\ref{first system low fre estimates section}, and the corresponding high frequency regime estimates in Section~\ref{first system high fre estimates section}. In Section~\ref{singular component outline section} we state the top order estimates for the singular component which are proved in \cite[Section~7]{Cmain}. Finally, we combine the estimates in Section~\ref{first system combined estimates section} to complete the proof of Theorem~\ref{main theorem first system}.

\textbf{Notation.} Unless otherwise noted, in this section we write $A\lesssim B$ for some quantities $A,B>0$ if there exists a constant $C>0$ depending only on the constants $M,C_0,C_1$ defined in the Introduction, such that $A\leq CB.$ 

\subsection{Construction of the singular component}\label{existence of singular component section}
In this section, we prove an existence and uniqueness result for solutions of \eqref{linear equation for singular component} with asymptotic data at $\mathcal{I}^-.$ In particular, this implies the existence and uniqueness of the singular component defined by \eqref{equation for Phi Y}.

We first remark that we frequently use $\nabla_{\tau}$ as a multiplier to obtain energy estimates. The following lemma implies that the additional terms resulting from differentiating the volume form or the metric can be controlled using Gronwall for $\tau\in(0,1].$ We point out that we usually bound these terms implicitly.

\begin{lemma}\label{volume form lemma}
    For any smooth horizontal tensor $F$ defined on $\mathcal{M},$ we have:
    \[\frac{1}{2}\frac{d}{d\tau}\big\|F\big\|_{L^2(S_{\tau})}^2=\int_{S_{\tau}}F\cdot\nabla_{\tau}FdVol_{\slashed{g}}+O\Big(\tau\big\|F\big\|_{L^2(S_{\tau})}^2\Big).\]
\end{lemma}
\begin{proof}
    We denote $v=\sqrt{\tau},$ and compute that $\partial_{\tau}=2\tau\partial_v$ and $e_4=\partial_v.$ Since $\chi=\mathcal{L}_{e_4}\slashed{g},$ we use the standard formula for $\frac{d}{dv}\int_{S^n}|F|^2(v)dVol_{\slashed{g}}$ to get:
    \[\frac{1}{2}\frac{d}{d\tau}\int_{S_{\tau}}|F|^2=\frac{1}{2}\int_{S_{\tau}}\nabla_{\tau}|F|^2+\int_{S_{\tau}}\tau\text{tr}\chi|F|^2=\int_{S_{\tau}}F\cdot\nabla_{\tau}F+O\Big(\tau\big\|F\big\|_{L^2(S_{\tau})}^2\Big).\]
\end{proof}

The following result implies that we can decompose the solution $\nabla^m\Phi_0$ into its regular and singular components, for each $m\leq M.$
\begin{proposition}\label{existence of singular component proposition}
    For any $\Phi^0,\Phi^1\in C^{\infty}(S^n)$ smooth tensors, there exists a unique solution on $\mathcal{M}$ of:
    \begin{equation}\label{linear equation for singular component}
        \nabla_{\tau}\big(\nabla_{\tau}\Phi\big)+\frac{1}{{\tau}}\nabla_{\tau}\Phi-4\Delta\Phi=\psi\nabla\Phi,
    \end{equation}
    \[\Phi({\tau})=\Phi^0\log({\tau})+\Phi^1+O\big({\tau}^2|\log({\tau})|^2\big),\ \nabla_{\tau}\Phi({\tau})=\frac{\Phi^0}{\tau}+O\big({\tau}|\log({\tau})|^2\big)\text{ in }C^{\infty}(S^n).\]
\end{proposition}
\begin{proof}
    It suffices to prove that for any $K>0$ there exists a unique solution of (\ref{linear equation for singular component}) such that the above expansions hold in $H^K(S^n).$ We introduce the quantity $\widetilde{\Phi}=\Phi-\Phi^0\log({\tau})-\Phi^1,$ which satisfies the equation:
    \begin{equation}\label{linear equation for renormalized singular component}
        \nabla_{\tau}\big(\nabla_{\tau}\widetilde{\Phi}\big)+\frac{1}{{\tau}}\nabla_{\tau}\widetilde{\Phi}-4\Delta\widetilde{\Phi}=\psi\nabla\widetilde{\Phi}+F_1(\Phi^0,\Phi^1)\cdot\log(\tau)+F_2(\Phi^0,\Phi^1).
    \end{equation}
    where $F_1$ and $F_2$ are bounded functions of $\Phi^0,\Phi^1,$ and their angular derivatives. To obtain this equation we use the fact that $\Phi^0$ and $\Phi^1$ are Lie transported in time, so by \eqref{formula for lie time derivative} we have $\nabla_{\tau}\Phi^0=\tau\chi\cdot\Phi_0$ and $\nabla_{\tau}\Phi^1=\tau\chi\cdot\Phi_1.$
    For any $k\geq0$ we obtain the commuted equation:
    \[\nabla_{\tau}\big(\nabla^k\nabla_{\tau}\widetilde{\Phi}\big)+\frac{1}{{\tau}}\nabla^k\nabla_{\tau}\widetilde{\Phi}-4\Delta\nabla^k\widetilde{\Phi}=\psi\nabla\nabla^k\widetilde{\Phi}+F_1^k(\Phi^0,\Phi^1)\cdot\log(\tau)+F_2^k(\Phi^0,\Phi^1)+F_3^k(\widetilde{\Phi}),\]
    where we denote the error terms:
    \[F_1^k(\Phi^0,\Phi^1)=\nabla^k\big(F_1(\Phi^0,\Phi^1)\big),\ F_2^k(\Phi^0,\Phi^1)=\nabla^k\big(F_2(\Phi^0,\Phi^1)\big),\ F_3^k(\widetilde{\Phi})=[\nabla_{\tau},\nabla^k]\nabla_{\tau}\widetilde{\Phi}-4[\Delta,\nabla^k]\widetilde{\Phi}-[\psi\nabla,\nabla^k]\widetilde{\Phi}.\]
    We remark that using the smoothness of $\Phi^0,\Phi^1,$ and the background metric $\slashed{g}$, we obtain that: \[F_1^k(\Phi^0,\Phi^1)\cdot\log(\tau)+F_2^k(\Phi^0,\Phi^1)=O_k\big(1+|\log\tau|\big).\]
    
    We define $\widetilde{\Phi}_{\epsilon}$ to be the solution to (\ref{linear equation for renormalized singular component}) on $[\epsilon,1]\times S^n$ with initial data $\widetilde{\Phi}_{\epsilon}|_{\tau=\epsilon}=\nabla_{\tau}\widetilde{\Phi}_{\epsilon}|_{\tau=\epsilon}=0.$ Contracting the commuted equation for $\widetilde{\Phi}_{\epsilon}$ with $\nabla^k\nabla_{\tau}\widetilde{\Phi}_{\epsilon}$, we obtain the standard energy estimate:
    \[\big\|\nabla^k\nabla_{\tau}\widetilde{\Phi}_{\epsilon}\big\|^2_{L^2}+\big\|\nabla\nabla^k\widetilde{\Phi}_{\epsilon}\big\|^2_{L^2}\lesssim_k\int_{\epsilon}^{\tau}\big\|\nabla^k\nabla_{\tau}\widetilde{\Phi}_{\epsilon}\big\|_{L^2}\big\|\nabla \nabla^k\widetilde{\Phi}_{\epsilon}\big\|_{L^2}+\int_{\epsilon}^{\tau}\big\|\nabla \nabla^k\widetilde{\Phi}_{\epsilon}\big\|_{L^2}\big\|[\nabla^{k+1},\nabla_{\tau}]\widetilde{\Phi}_{\epsilon}\big\|_{L^2}+\]\[+\int_{\epsilon}^{\tau}|\log(\tau')|\cdot\big\|\nabla^k\nabla_{\tau}\widetilde{\Phi}_{\epsilon}\big\|_{L^2}\big\|F_1^k(\Phi_0,\Phi_1)\big\|_{L^2}+\int_{\epsilon}^{\tau}\big\|\nabla^k\nabla_{\tau}\widetilde{\Phi}_{\epsilon}\big\|_{L^2}\big\|F_2^k(\Phi_0,\Phi_1)\big\|_{L^2}+\int_{\epsilon}^{\tau}\big\|\nabla^k\nabla_{\tau}\widetilde{\Phi}_{\epsilon}\big\|_{L^2}\big\|F_3^k(\widetilde{\Phi}_{\epsilon})\big\|_{L^2}\]
    We notice that in the above estimate we dropped the bulk term with a favorable sign. Also, it is essential that $\widetilde{\Phi}_{\epsilon}$ vanishes to sufficiently high order at $\tau=0$ in order to not have any initial data contribution. Moreover, due to Lemma \ref{volume form lemma}, we also have the terms 
    $\int_{\epsilon}^{\tau}\tau'\|\nabla\nabla^k\widetilde{\Phi}_{\epsilon}\|_{L^2}^2$ and $\int_{\epsilon}^{\tau}\tau'\|\nabla^k\nabla_{\tau}\widetilde{\Phi}_{\epsilon}\|_{L^2}^2$ on the RHS, but these can be bounded using Gronwall. 

    Using the higher order version of the commutation formula \eqref{commutation formula 1}, the smoothness of the background metric $\slashed{g}$, and the Gronwall inequality, we obtain the estimate:
    \[\big\|\nabla^k\nabla_{\tau}\widetilde{\Phi}_{\epsilon}\big\|^2_{L^2}+\big\|\nabla\nabla^k\widetilde{\Phi}_{\epsilon}\big\|^2_{L^2}\lesssim_k\int_{\epsilon}^{\tau}\big\|\widetilde{\Phi}_{\epsilon}\big\|_{H^{k+1}}^2d\tau'+\int_{\epsilon}^{\tau}\big\|\nabla_{\tau}\widetilde{\Phi}_{\epsilon}\big\|_{H^{k}}^2d\tau'+\int_{\epsilon}^{\tau}\big\|\nabla_{\tau}\widetilde{\Phi}_{\epsilon}\big\|_{H^{k}}\cdot\big(1+|\log\tau'|\big)d\tau'.\]
    where the implicit constant depends on $k\geq0.$ On the other hand, using Lemma \ref{volume form lemma}, Gronwall, and the commutation formulas, we also have:
    \[\big\|\widetilde{\Phi}_{\epsilon}\big\|^2_{H^k}\lesssim_k\int_{\epsilon}^{\tau}\big\|\widetilde{\Phi}_{\epsilon}\big\|_{H^{k}}^2d\tau'+\int_{\epsilon}^{\tau}\big\|\nabla_{\tau}\widetilde{\Phi}_{\epsilon}\big\|_{H^{k}}^2d\tau'.\]
    We use our previous two estimates and Gronwall to obtain that for any $\tau\in[\epsilon,1]$:
    \[\big\|\nabla_{\tau}\widetilde{\Phi}_{\epsilon}\big\|^2_{H^k}+\big\|\widetilde{\Phi}_{\epsilon}\big\|^2_{H^{k+1}}\lesssim_k\int_{\epsilon}^{\tau}\big\|\nabla_{\tau}\widetilde{\Phi}_{\epsilon}\big\|_{H^{k}}\cdot\big(1+|\log\tau'|\big)d\tau'.\]
    By taking the supremum on $[\epsilon,\tau]$ in the above inequality for each $\tau\in[\epsilon,1]$, we obtain that:
    \[\big\|\nabla_{\tau}\widetilde{\Phi}_{\epsilon}\big\|_{H^k}=O_k\big(\tau(1+|\log\tau|)\big).\]
    We use this and a similar argument for $\big\|\widetilde{\Phi}_{\epsilon}\big\|_{H^{k}}$ to obtain the bound:
    \[\big\|\widetilde{\Phi}_{\epsilon}\big\|_{H^{k}}=O_k\big(\tau^2(1+|\log\tau|)^2\big).\]
    Using the Banach-Alaoglu theorem and compactness, we obtain that for $K\ll k$ there exists $\widetilde{\Phi}$ a solution of (\ref{linear equation for renormalized singular component}) on $\mathcal{M}$ such that:
    \[\big\|\nabla_{\tau}\widetilde{\Phi}\big\|_{H^K}=O_K\big(\tau(1+|\log\tau|)^2\big),\ \big\|\widetilde{\Phi}\big\|_{H^{K}}=O_K\big(\tau^2(1+|\log\tau|)^2\big).\]
    As a result, we obtain that $\Phi=\widetilde{\Phi}+\Phi^0\log({\tau})+\Phi^1$ is a solution of (\ref{linear equation for singular component}) and satisfies the desired expansions. Finally, we remark that uniqueness follows by using our standard energy estimate on $[0,\tau]$ for the difference of two solutions with the same asymptotic expansion.
\end{proof}

\subsection{Lower order estimates}\label{lower order estimates section}

The goal of this section is to establish estimates that are lower order in terms of the number of angular derivatives. We point out that the estimates of this section are not sharp, but it is essential that we use only the quantities that appear on the right hand side of the estimates in Theorem~\ref{main theorem first system}.

The lower order estimates are carried out in two parts. First, we prove estimates for an integer number of angular derivatives $m<M,$ establishing \eqref{lower order regular estimate intro}, \eqref{lower order singular estimate intro}, and the bounds for the commutator term $\mathcal{C}.$ We then also prove estimates for $M-\frac{1}{2}$ derivatives of the singular component $\big(\Phi_0\big)_Y.$

\subsubsection{Standard estimates}

We first prove the lower order estimates for the singular component using the strategy outlined in Section~\ref{first model system intro section}. We further decompose for every $m<M$:
\[\big(\nabla^m\Phi_0\big)_Y=\big(\nabla^m\Phi_0\big)_Y^1+\big(\nabla^m\Phi_0\big)_Y^2,\]
where using Proposition \ref{existence of singular component proposition} we have that $\big(\nabla^m\Phi_0\big)_Y^1,\ \big(\nabla^m\Phi_0\big)_Y^2$ are the solutions of (\ref{equation for Phi Y}) such that:
\begin{align*}
    &\big(\nabla^m\Phi_0\big)_Y^1({\tau})=2\nabla^m\mathcal{O}\log({\tau})+ O\big({\tau}^2|\log({\tau})|^2\big),\ \nabla_{\tau}\big(\nabla^m\Phi_0\big)_Y^1({\tau})=\frac{2\nabla^m\mathcal{O}}{\tau}+ O\big({\tau}|\log({\tau})|^2\big),\\
    &\big(\nabla^m\Phi_0\big)_Y^2({\tau})=2(\log\nabla)\nabla^m\mathcal{O}+ O\big({\tau}^2|\log({\tau})|^2\big),\ \nabla_{\tau}\big(\nabla^m\Phi_0\big)_Y^2({\tau})=O\big({\tau}|\log({\tau})|^2\big).
\end{align*}
For convenience of notation, we often write $\nabla^m\Phi_{0Y}^1$ and $\nabla^m\Phi_{0Y}^2$ instead of  $\big(\nabla^m\Phi_0\big)_Y^1$ and $\big(\nabla^m\Phi_0\big)_Y^2$. Using this decomposition, we prove the following lower order estimates on the singular component:
\begin{proposition}\label{practical estimate for Phi Y proposition}
Set $\eta=1/10$. The singular component satisfies the following estimates for any $m<M$:
\begin{align*}
    \bigg\|\nabla_{\tau}\frac{\nabla^m\Phi_{0Y}^1}{\log\tau}\bigg\|^2_{L^2}+\bigg\|\frac{\nabla^m\Phi_{0Y}^1}{\log\tau}\bigg\|_{H^1}^2&\lesssim\big\|\mathcal{O}\big\|^2_{H^{m+1}}\text{ for }\tau\in\bigg(0,\frac{1}{2}\bigg],\\
    \big\|\nabla_{\tau}\nabla^m\Phi_{0Y}^1\big\|^2_{L^2}+\big\|\nabla^m\Phi_{0Y}^1\big\|^2_{H^1}&\lesssim\big\|\mathcal{O}\big\|^2_{H^{m+1}}\text{ for }\tau\in\bigg[\frac{1}{2},1\bigg],\\
    \big\|\nabla_{\tau}\nabla^m\Phi_{0Y}^2\big\|^2_{L^2}+\big\|\nabla^m\Phi_{0Y}^2\big\|^2_{H^1}&\lesssim\big\|\mathcal{O}\big\|^2_{H^{m+1+\eta}}\text{ for }\tau\in(0,1].
\end{align*}
In particular, the singular component satisfies the estimate \eqref{lower order singular estimate intro} for any $m<M$:
    \[\big\|\nabla^m\Phi_{0Y}\big\|^2_{H^1}\lesssim\big(1+|\log\tau|^2\big)\big\|\mathcal{O}\big\|^2_{H^{m+1+\eta}}.\]
\end{proposition}
\begin{proof}
    We start with the estimate for $\nabla^m\Phi_{0Y}^2.$ Using the standard $\nabla_{\tau}$ multiplier in \eqref{equation for Phi Y}, we get for $\tau\in(0,1]$:
    \begin{align*}
        \big\|\nabla_{\tau}\nabla^m\Phi_{0Y}^2\big\|^2_{L^2}+\big\|\nabla \nabla^m\Phi_{0Y}^2\big\|^2_{L^2}&\lesssim\big\|\nabla \nabla^m\Phi_{0Y}^2\big\|^2_{L^2}\big|_{\tau=0}+\int_0^{\tau}\big\|\nabla_{\tau}\nabla^m\Phi_{0Y}^2\big\|_{L^2}\big\|\nabla\nabla^m\Phi_{0Y}^2\big\|_{L^2}d\tau'+\\
        &+\int_0^{\tau}\big\|\nabla\nabla^m\Phi_{0Y}^2\big\|_{L^2}\big\|[\nabla,\nabla_{\tau}]\nabla^m\Phi_{0Y}^2\big\|_{L^2}d\tau'.
    \end{align*}
    We notice that due to Lemma \ref{volume form lemma}, we also have the terms 
    $\int_0^{\tau}\tau'\|\nabla\nabla^m\Phi_{0Y}^2\|_{L^2}^2$ and $\int_0^{\tau}\tau'\|\nabla_{\tau}\nabla^m\Phi_{0Y}^2\|_{L^2}^2$ on the RHS, but these can be bounded using Gronwall. By the commutation formula \eqref{commutation formula 1}, Lemma \ref{bound on log nabla lemma}, and Gronwall, we get:
    \[\big\|\nabla_{\tau}\nabla^m\Phi_{0Y}^2\big\|^2_{L^2}+\big\|\nabla \nabla^m\Phi_{0Y}^2\big\|^2_{L^2}\lesssim\big\|\mathcal{O}\big\|^2_{H^{m+1+\eta}}+\int_0^{\tau}\big\|\nabla^m\Phi_{0Y}^2\big\|^2_{L^2}d\tau'.\]
    On the other hand, we also have the estimate:
    \begin{align*}
        \big\|\nabla^m\Phi_{0Y}^2\big\|^2_{L^2}&\lesssim\big\|\mathcal{O}\big\|^2_{H^{m+\eta}}+\int_0^{\tau}\big\|\nabla^m\Phi_{0Y}^2\big\|_{L^2}\big\|\nabla_{\tau}\nabla^m\Phi_{0Y}^2\big\|_{L^2}d\tau'+\int_0^{\tau}\big\|\nabla^m\Phi_{0Y}^2\big\|^2_{L^2}d\tau'\\
        &\lesssim\big\|\mathcal{O}\big\|^2_{H^{m+\eta}}+\int_0^{\tau}\big\|\nabla_{\tau}\nabla^m\Phi_{0Y}^2\big\|^2_{L^2}d\tau'.
    \end{align*}
    Combining the last two bounds, we obtain the desired estimate for $\big(\nabla^m\Phi_{0}\big)_Y^2.$

    We notice that $\nabla^m\Phi_{0Y}^1$ satisfies the equation:
    \[\nabla_{\tau}\bigg(\nabla_{\tau}\frac{\nabla^m\Phi_{0Y}^1}{\log\tau}\bigg)+\frac{1}{\tau}\bigg(1+\frac{2}{\log\tau}\bigg)\nabla_{\tau}\frac{\nabla^m\Phi_{0Y}^1}{\log\tau}-4\Delta\frac{\nabla^m\Phi_{0Y}^1}{\log\tau}=\psi\nabla\frac{\nabla^m\Phi_{0Y}^1}{\log\tau}.\]
    We use $\nabla_{\tau}$ as a multiplier to obtain the estimate for any $\tau\in(0,1/2]$:
    \begin{align*}
        \bigg\|\nabla_{\tau}\frac{\nabla^m\Phi_{0Y}^1}{\log\tau}\bigg\|^2_{L^2}+\bigg\|\nabla\frac{\nabla^m\Phi_{0Y}^1}{\log\tau}\bigg\|^2_{L^2}\lesssim\bigg\|\nabla\frac{\nabla^m\Phi_{0Y}^1}{\log\tau}\bigg\|^2_{L^2}\bigg|_{\tau=0}+\int_0^{\tau}\frac{\textbf{1}_{[1/10,1/2]}}{\tau'|\log\tau'|}\bigg\|\nabla_{\tau}\frac{\nabla^m\Phi_{0Y}^1}{\log\tau}\bigg\|^2_{L^2}d\tau'+&\\
        +\int_0^{\tau}\bigg\|\nabla_{\tau}\frac{\nabla^m\Phi_{0Y}^1}{\log\tau}\bigg\|_{L^2}\bigg\|\nabla\frac{\nabla^m\Phi_{0Y}^1}{\log\tau}\bigg\|_{L^2}d\tau'+\int_0^{\tau}\bigg\|\nabla\frac{\nabla^m\Phi_{0Y}^1}{\log\tau}\bigg\|_{L^2}\bigg\|[\nabla,\nabla_{\tau}]\frac{\nabla^m\Phi_{0Y}^1}{\log\tau}\bigg\|_{L^2}d\tau'&.
    \end{align*}
    We point out that the error term $\frac{1}{\tau'|\log\tau'|}\big\|\nabla_{\tau}\nabla^m\Phi_{0Y}^1/\log\tau\big\|^2_{L^2}\cdot\textbf{1}_{[1/10,1/2]}$ appears on the right hand side because for $\tau\in[0,1/10]$ we have $1+2/\log\tau\gtrsim1,$ so the bulk term has a favorable sign. The error terms are estimated as usual, and we use Gronwall to obtain:
    \[\bigg\|\nabla_{\tau}\frac{\nabla^m\Phi_{0Y}^1}{\log\tau}\bigg\|^2_{L^2}+\bigg\|\nabla\frac{\nabla^m\Phi_{0Y}^1}{\log\tau}\bigg\|^2_{L^2}\lesssim\big\|\mathcal{O}\big\|^2_{H^{m+1}}+\int_0^{\tau}\bigg\|\frac{\nabla^m\Phi_{0Y}^1}{\log\tau}\bigg\|_{L^2}^2d\tau'.\]
    We also have the estimate:
    \[\bigg\|\frac{\nabla^m\Phi_{0Y}^1}{\log\tau}\bigg\|_{L^2}^2\lesssim\big\|\mathcal{O}\big\|^2_{H^{m}}+\int_0^{\tau}\bigg\|\frac{\nabla^m\Phi_{0Y}^1}{\log\tau}\bigg\|_{L^2}\bigg\|\nabla_{\tau}\frac{\nabla^m\Phi_{0Y}^1}{\log\tau}\bigg\|_{L^2}+\int_0^{\tau}\bigg\|\frac{\nabla^m\Phi_{0Y}^1}{\log\tau}\bigg\|_{L^2}^2\lesssim\big\|\mathcal{O}\big\|^2_{H^{m}}+\int_0^{\tau}\bigg\|\nabla_{\tau}\frac{\nabla^m\Phi_{0Y}^1}{\log\tau}\bigg\|_{L^2}^2.\]
    Combining the last two bounds, we obtain the desired estimate for $\big(\nabla^m\Phi_{0}\big)_Y^1$ on $\tau\in(0,1/2].$

    In the case of $\nabla^m\Phi_{0Y}^1$ for $\tau\in[1/2,1],$ we use the same estimates as for $\nabla^m\Phi_{0Y}^2$, but with data at $\tau=1/2.$ This allows us to obtain the desired estimates for $\big(\nabla^m\Phi_{0}\big)_Y^1$ on $\tau\in[1/2,1].$
\end{proof}

\begin{remark}
    We notice that the only place in the above proof where we use an inequality that is not sharp is when bounding the $(\log\nabla)$ operator using Lemma \ref{bound on log nabla lemma}. In order to prove top order estimates, we will take a different approach to avoid this issue in Section~\ref{first system top order estimates section}.
\end{remark}

Next, we prove the following lower order estimates for the regular components, establishing \eqref{lower order regular estimate intro}:
\begin{proposition}\label{standard estimates regular components propositionn}
Set $\eta=1/10.$ For any $m<M$, we have the following estimate:
\[\big\|\nabla_{\tau}\nabla^m\Phi_{0J}\big\|^2_{L^2}+\big\|\nabla^m\Phi_{0J}\big\|^2_{H^1}+\sum_{i=1}^I\big\|\nabla_{\tau}\nabla^m\Phi_i\big\|^2_{L^2}+\sum_{i=1}^I\big\|\nabla^m\Phi_i\big\|^2_{H^1}\lesssim\]
\[\lesssim\big\|\mathfrak{h}\big\|^2_{H^{m+1}}+\sum_{i=1}^I\big\| \Phi_i^0\big\|^2_{H^{m+1}}+\big\|\mathcal{O}\big\|^2_{H^{m+1+\eta}}+\sum_{i=0}^I\int_0^{\tau}\big\|F_m^i\big\|_{L^2}^2d\tau'.\]
\end{proposition}
\begin{proof}
    Using $\nabla_{\tau}$ as a multiplier in the equations \eqref{equation for Phi J} and \eqref{equation for Phi i} satisfied by the regular quantities, we get:
    \[\big\|\nabla_{\tau}\nabla^m\Phi_{0J}\big\|^2_{L^2}+\big\|\nabla \nabla^m\Phi_{0J}\big\|^2_{L^2}\lesssim\big\|\mathfrak{h}_m\big\|^2_{H^{1}}+\sum_{i=0}^I\int_0^{\tau}\big\|F_m^i\big\|_{L^2}^2+\sum_{i=1}^I\int_0^{\tau}\big\|\nabla^{m+1}\Phi_i\big\|_{L^2}^2+\int_0^{\tau}\big\|\nabla^m\Phi_{0J}\big\|_{L^2}^2,\]
    \[\sum_{i=1}^I\big\|\nabla_{\tau}\nabla^m\Phi_i\big\|^2_{L^2}+\sum_{i=1}^I\big\| \nabla^{m+1}\Phi_i\big\|^2_{L^2}\lesssim\]\[\lesssim\sum_{i=1}^I\big\| \Phi_i^0\big\|^2_{H^{m+1}}+\sum_{i=0}^I\int_0^{\tau}\big\|F_m^i\big\|_{L^2}^2+\int_0^{\tau}\big\|\nabla\nabla^m\Phi_{0Y}\big\|_{L^2}^2+\sum_{i=1}^I\int_0^{\tau}\big\|\nabla^m\Phi_i\big\|_{L^2}^2+\int_0^{\tau}\big\|\nabla\nabla^m\Phi_{0J}\big\|_{L^2}^2.\]
    In the second estimate, we dropped the bulk term with a favorable sign obtained because $\sigma=1$ in \eqref{equation for Phi i}. Using the above two bounds, together with Lemma \ref{commute with log lemma} and Gronwall, we obtain:
    \[\big\|\nabla_{\tau}\nabla^m\Phi_{0J}\big\|^2_{L^2}+\big\|\nabla \nabla^m\Phi_{0J}\big\|^2_{L^2}+\sum_{i=1}^I\big\|\nabla_{\tau}\nabla^m\Phi_i\big\|^2_{L^2}+\sum_{i=1}^I\big\|\nabla^{m+1}\Phi_i\big\|^2_{L^2}\lesssim\]\[\lesssim\big\|\mathfrak{h}\big\|^2_{H^{m+1}}+\sum_{i=1}^I\big\| \Phi_i^0\big\|^2_{H^{m+1}}+\big\|\mathcal{O}\big\|^2_{H^{m+1+\eta}}+\sum_{i=0}^I\int_0^{\tau}\big\|F_m^i\big\|_{L^2}^2+\sum_{i=1}^I\int_0^{\tau}\big\|\nabla^m\Phi_i\big\|_{L^2}^2+\int_0^{\tau}\big\|\nabla^m\Phi_{0J}\big\|_{L^2}^2.\]
    We can estimate the last two error terms on the RHS as before, which implies the conclusion.    
\end{proof}

In the above proof, we used the following bound on $\mathfrak{h}_m$:
\begin{lemma}\label{commute with log lemma} We have the estimate for any $m<M$:
    \[\big\|\mathfrak{h}_m\big\|_{H^1}\lesssim\big\|\mathfrak{h}\big\|_{H^{m+1}}+\big\|\mathcal{O}\big\|_{H^{m}}.\]
\end{lemma}
\begin{proof} We can write:
    \[\mathfrak{h}_m=\nabla^mh-2(\log\nabla)\nabla^m\mathcal{O}=\nabla^m\mathfrak{h}-2[(\log\nabla),\nabla^m]\mathcal{O}.\]
    Using (\ref{LP est 2}) and (\ref{LP est 3}) applied to the projection operator $P_k^2,$ we get for all $k\geq0$:
    \[\big\|[\nabla^m,P_k^2]\mathcal{O}\big\|_{H^1}\lesssim2^{-k} C\big(\|\slashed{Riem}_0\|_{H^m}\big)\cdot\big\|\mathcal{O}\big\|_{H^{m}}.\]
    As a result, we have that:
     \[\big\|[(\log\nabla),\nabla^m]\mathcal{O}\big\|_{H^1}\lesssim C\big(\|\slashed{Riem}_0\|_{H^m}\big)\cdot\big\|\mathcal{O}\big\|_{H^{m}}\sum_{k\geq0}2^{-k}\log2^k.\]
     Moreover, we notice that $\|\slashed{Riem}_0\|_{H^{M-1}}\leq C_1\lesssim1$ by our bounds on the background spacetime in Theorem~\ref{main theorem first system}.
\end{proof}

We combine the above results to obtain the following lower order version of \eqref{practical estimate first model system}:
\begin{corollary}
    We have the lower order estimate for $\nabla^m\Phi_0$, with $m<M:$
    \[\big\|\nabla^m\Phi_0\big\|^2_{H^1}\lesssim\big(1+|\log\tau|^2\big)\big\|\mathcal{O}\big\|^2_{H^{m+1+\eta}}+\big\|\mathfrak{h}\big\|^2_{H^{m+1}}+\sum_{i=1}^I\big\| \Phi_i^0\big\|^2_{H^{m+1}}+\sum_{i=0}^I\int_0^{\tau}\big\|F_m^i\big\|_{L^2}^2d\tau'.\]
\end{corollary}

We also need estimates for the commutator term $\mathcal{C}=\big(\nabla^M\Phi_0\big)_Y-\nabla\big(\nabla^{M-1}\Phi_0\big)_Y$. To compute the equation satisfied by $\mathcal{C},$ we first need to commute the equation of $\big(\nabla^{M-1}\Phi_0\big)_Y$ using the commutation formula \eqref{commutation formula 1}:
\[\nabla_{\tau}\big(\nabla_{\tau}\nabla\big(\nabla^{M-1}\Phi_0\big)_Y\big)+\frac{1}{{\tau}}\nabla_{\tau}\nabla\big(\nabla^{M-1}\Phi_0\big)_Y-4\Delta\nabla\big(\nabla^{M-1}\Phi_0\big)_Y=\psi\nabla\nabla\big(\nabla^{M-1}\Phi_0\big)_Y+\]\[+O\Big(\big|\big(\nabla^{M-1}\Phi_0\big)_Y\big|+\big|\nabla\big(\nabla^{M-1}\Phi_0\big)_Y\big|+\big|\tau\nabla_{\tau}\big(\nabla^{M-1}\Phi_0\big)_Y\big|+\big|\tau\nabla_{\tau}\nabla\big(\nabla^{M-1}\Phi_0\big)_Y\big|\Big),\]
where we used \eqref{preliminary estimates for background metric} and \eqref{preliminary estimates for background psi}. As a result, we obtain that $\mathcal{C}$ satisfies the equation:
\begin{align*}
    \nabla_{\tau}\big(\nabla_{\tau}\mathcal{C}\big)+\frac{1}{{\tau}}\nabla_{\tau}\mathcal{C}-4\Delta\mathcal{C}&=\psi\nabla\mathcal{C}+O\Big(\big|\tau\nabla_{\tau}\mathcal{C}\big|\Big)+\\
    &+O\Big(\big|\big(\nabla^{M-1}\Phi_0\big)_Y\big|+\big|\nabla\big(\nabla^{M-1}\Phi_0\big)_Y\big|+\big|\tau\nabla_{\tau}\big(\nabla^{M-1}\Phi_0\big)_Y\big|+\big|\tau\nabla_{\tau}\big(\nabla^{M}\Phi_0\big)_Y\big|\big).
\end{align*}
Moreover, we notice that $\mathcal{C}$ is a regular quantity at $\tau=0,$ satisfying the expansion:
\[\mathcal{C}=2[\log\nabla,\nabla]\nabla^{M-1}\mathcal{O}+O\big({\tau}^2|\log({\tau})|^2\big),\ \nabla_{\tau}\mathcal{C}=O\big({\tau}|\log({\tau})|^2\Big).\]
Note that as in the proof of Lemma \ref{commute with log lemma} we get:
\[\big\|[\log\nabla,\nabla]\nabla^{M-1}\mathcal{O}\big\|_{H^1}\lesssim\big\|\mathcal{O}\big\|_{H^M}.\]
We conclude the section by proving the estimate:
\begin{proposition}\label{singular component commutator proposition}
    The commutator term $\mathcal{C}=\big(\nabla^M\Phi_0\big)_Y-\nabla\big(\nabla^{M-1}\Phi_0\big)_Y$ satisfies the following estimate:
    \[\big\|\nabla_{\tau}\mathcal{C}\big\|^2_{L^2}+\big\|\mathcal{C}\big\|^2_{H^1}\lesssim\big\|\mathcal{O}\big\|^2_{H^{M+\eta}}+\int_0^{\tau}\big\|\tau'\nabla_{\tau}\nabla^{M}\Phi_{0Y}\big\|_{L^2}^2d\tau'.\]
    Additionally, for any $k\geq0$ and $\tau\in(0,2^{-k-1}]$ we have:
    \[\big\|\mathcal{C}\big\|^2_{L^2}\lesssim\big\|\mathcal{O}\big\|^2_{H^{M+\eta}}+2^{-k}\int_0^{\tau}\big\|\tau'\nabla_{\tau}\nabla^{M}\Phi_{0Y}\big\|_{L^2}^2d\tau'.\]
\end{proposition}
\begin{proof}
    Using the standard $\nabla_{\tau}$ multiplier and the previous lower order estimates we get:
    \[\big\|\nabla_{\tau}\mathcal{C}\big\|^2_{L^2}+\big\|\mathcal{C}\big\|^2_{H^1}\lesssim\big\|\mathcal{O}\big\|^2_{H^{M}}+\int_0^{\tau}\big\|\nabla^{M-1}\Phi_{0Y}\big\|_{H^1}^2+\int_0^{\tau}\big\|\tau'\nabla_{\tau}\nabla^{M-1}\Phi_{0Y}\big\|_{L^2}^2+\int_0^{\tau}\big\|\tau'\nabla_{\tau}\nabla^{M}\Phi_{0Y}\big\|_{L^2}^2\]
    \[\lesssim\big\|\mathcal{O}\big\|^2_{H^{M+\eta}}+\int_0^{\tau}\big\|\tau'\nabla_{\tau}\nabla^{M}\Phi_{0Y}\big\|_{L^2}^2\]
    As before, we can also estimate $\big\|\mathcal{C}\big\|^2_{L^2}$ in order to obtain the conclusion.
\end{proof}

\subsubsection{Fractional estimates}\label{fractional estimates section}
In this section, we prove estimates in $H^{1/2}$ for $\nabla(\nabla^{M-1}\Phi_0)_Y$ and $\nabla_{\tau}(\nabla^{M-1}\Phi_0)_Y$. As explained in Section~[Section~7]\ref{first model system intro section}, the error terms obtained in \cite{Cmain} in the proof of the top order estimates for the singular component $\big(\nabla^M\Phi_0\big)_Y$ can be simplified significantly using the fractional estimates proved in this section.

For the rest of the section we prove the following result:
\begin{proposition}
    The singular component $(\nabla^{M-1}\Phi_0)_Y$ satisfies the estimates:
    \begin{align*}
        \big\|\nabla(\nabla^{M-1}\Phi_0)_Y\big\|^2_{H^{1/2}}&\lesssim\big(1+|\log\tau|^2\big)\big\|\mathcal{O}\big\|^2_{H^{M+1/2+\eta}}\\
        \big\|\nabla_{\tau}(\nabla^{M-1}\Phi_0)_Y\big\|^2_{H^{1/2}}&\lesssim\frac{1}{\tau^2}\big\|\mathcal{O}\big\|^2_{H^{M+1/2+\eta}}.
    \end{align*}
\end{proposition}

In order to prove this result we treat separately the components $\big(\nabla^{M-1}\Phi_0\big)_Y^1$ and $\big(\nabla^{M-1}\Phi_0\big)_Y^2$. The above fractional estimates will be a consequence of Propositions~\ref{A proposition} and \ref{B proposition}. We start with the component $\mathcal{A}:=\big(\nabla^{M-1}\Phi_0\big)_Y^2,$ which satisfies the equation:
\[\nabla_{\tau}\big(\nabla_{\tau}\mathcal{A}\big)+\frac{1}{{\tau}}\nabla_{\tau}\mathcal{A}-4\Delta\mathcal{A}=\psi\nabla\mathcal{A},\]
\[\mathcal{A}({\tau})=2(\log\nabla)\nabla^{M-1}\mathcal{O}+ O\big({\tau}^2|\log({\tau})|^2\big),\ \nabla_{\tau}\mathcal{A}({\tau})=O\big({\tau}|\log({\tau})|^2\big).\]
\begin{proposition}\label{A proposition}
    $\mathcal{A}$ satisfies the estimate:
    \[\big\|\nabla_{\tau}\mathcal{A}\big\|^2_{H^{1/2}}+\big\|\nabla \mathcal{A}\big\|^2_{H^{1/2}}\lesssim\big\|\mathcal{O}\big\|^2_{H^{M+1/2+\eta}}.\]
\end{proposition}
\begin{proof}
For any $k\geq0,$ we apply $P_k$ to the equation satisfied by $\mathcal{A}$ to get:
    \[\nabla_{\tau}\big(P_k\nabla_{\tau}\mathcal{A}\big)+\frac{1}{\tau}P_k\nabla_{\tau}\mathcal{A}-4\Delta P_k\mathcal{A}=\psi\nabla P_k \mathcal{A}+[P_k,\psi]\nabla\mathcal{A}+\psi[P_k,\nabla]\mathcal{A}+[\nabla_{\tau},P_k]\nabla_{\tau}\mathcal{A}.\]
    We contract this equation with $P_k\nabla_{\tau}\mathcal{A}$ and integrate by parts to obtain the energy estimate:
    \[\big\|P_k\nabla_{\tau}\mathcal{A}\big\|^2_{L^2}+\big\|\nabla P_k\mathcal{A}\big\|^2_{L^2}\lesssim\big\|\nabla P_k\mathcal{A}|_{\tau=0}\big\|^2_{L^2}+\int_0^{\tau}\tau'\big\|\nabla P_k\mathcal{A}\big\|_{L^2}\cdot\big\|\nabla[P_k,\nabla_4]\mathcal{A}\big\|_{L^2}+\]\[+\int_0^{\tau}\tau'\big\|\nabla P_k\mathcal{A}\big\|_{L^2}\cdot\big\|[\nabla,\nabla_4]P_k\mathcal{A}\big\|_{L^2}+\int_0^{\tau}\big\|\nabla P_k\mathcal{A}\big\|_{L^2}\cdot\big\|P_k\nabla_{\tau}\mathcal{A}\big\|_{L^2}+\int_0^{\tau}\big\|[P_k,\psi]\nabla\mathcal{A}\big\|_{L^2}\cdot\big\|P_k\nabla_{\tau}\mathcal{A}\big\|_{L^2}\]\[+\int_0^{\tau}\big\|[P_k,\nabla]\mathcal{A}\big\|_{L^2}\cdot\big\|P_k\nabla_{\tau}\mathcal{A}\big\|_{L^2}+\int_0^{\tau}\tau'\big\|[\nabla_{4},P_k]\nabla_{\tau}\mathcal{A}\big\|_{L^2}\cdot\big\|P_k\nabla_{\tau}\mathcal{A}\big\|_{L^2}.\]
    We use Gronwall:
    \[\big\|P_k\nabla_{\tau}\mathcal{A}\big\|^2_{L^2}+\big\|\nabla P_k\mathcal{A}\big\|^2_{L^2}\lesssim\big\|\nabla P_k\mathcal{A}|_{\tau=0}\big\|^2_{L^2}+\int_0^{\tau}\big\|\nabla[P_k,\nabla_4]\mathcal{A}\big\|_{L^2}^2+\]\[+\int_0^{\tau}\big\|[\nabla,\nabla_4]P_k\mathcal{A}\big\|_{L^2}^2+\int_0^{\tau}\big\|[P_k,\psi]\nabla\mathcal{A}\big\|_{L^2}^2+\int_0^{\tau}\big\|[P_k,\nabla]\mathcal{A}\big\|_{L^2}^2+\int_0^{\tau}\big\|[\nabla_{4},P_k]\nabla_{\tau}\mathcal{A}\big\|_{L^2}^2.\]
    We use the bounds in Lemma \ref{Litt Paley lemma} and Lemma \ref{Litt Paley lemma refined} to control the commutation terms:
    \[\big\|P_k\nabla_{\tau}\mathcal{A}\big\|^2_{L^2}+\big\|\nabla P_k\mathcal{A}\big\|^2_{L^2}\lesssim\big\|\nabla P_k\mathcal{A}|_{\tau=0}\big\|^2_{L^2}+\int_0^{\tau}\big\|\underline{\widetilde{P}}_k\nabla\mathcal{A}\big\|_{L^2}^2+\int_0^{\tau}2^{-2k}\big\|\mathcal{A}\big\|_{H^1}^2\]\[+\int_0^{\tau}\big\|P_k\mathcal{A}\big\|_{H^1}^2+\int_0^{\tau}\big\|\underline{\widetilde{P}}_k\nabla_{\tau}\mathcal{A}\big\|_{L^2}^2+\int_0^{\tau}2^{-2k}\big\|\nabla_{\tau}\mathcal{A}\big\|_{L^2}^2.\]
    As a result, we get that:
    \[\big\|P_k\nabla_{\tau}\mathcal{A}\big\|^2_{L^2}+\big\|P_k\nabla\mathcal{A}\big\|^2_{L^2}\lesssim\big\|P_k\nabla\mathcal{A}|_{\tau=0}\big\|^2_{L^2}+2^{-2k}\big\|\mathcal{A}|_{\tau=0}\big\|^2_{L^2}+\int_0^{\tau}\big\|\underline{\widetilde{P}}_k\nabla\mathcal{A}\big\|_{L^2}^2+\]\[+\int_0^{\tau}2^{-2k}\big\|\mathcal{A}\big\|_{H^1}^2+\int_0^{\tau}\big\|P_k\mathcal{A}\big\|_{L^2}^2+\int_0^{\tau}\big\|\underline{\widetilde{P}}_k\nabla_{\tau}\mathcal{A}\big\|_{L^2}^2+\int_0^{\tau}2^{-2k}\big\|\nabla_{\tau}\mathcal{A}\big\|_{L^2}^2.\]
    We multiply by $2^k$ and sum over all $k\geq0.$ This amounts to taking half of a derivative.
    \[\sum_{k\geq0}2^k\Big(\big\|P_k\nabla_{\tau}\mathcal{A}\big\|^2_{L^2}+\big\|P_k\nabla\mathcal{A}\big\|^2_{L^2}\Big)\lesssim\big\|\mathcal{A}|_{\tau=0}\big\|^2_{L^2}+\int_0^{\tau}\big\|\mathcal{A}\big\|_{H^1}^2+\int_0^{\tau}\big\|\nabla_{\tau}\mathcal{A}\big\|_{L^2}^2+\sum_{k\geq0}2^k\big\|P_k\nabla\mathcal{A}|_{\tau=0}\big\|^2_{L^2}+\]\[+\int_0^{\tau}\sum_{k\geq0}2^k\big\|\underline{\widetilde{P}}_k\nabla\mathcal{A}\big\|_{L^2}^2+\int_0^{\tau}\sum_{k\geq0}2^k\big\|P_k\mathcal{A}\big\|_{L^2}^2+\int_0^{\tau}\sum_{k\geq0}2^k\big\|\underline{\widetilde{P}}_k\nabla_{\tau}\mathcal{A}\big\|_{L^2}^2.\]
    We recall that from the standard lower order estimates we have:
    \[\big\|\nabla_{\tau}\mathcal{A}\big\|^2_{L^2}+\big\|\mathcal{A}\big\|^2_{H^1}\lesssim\big\|\mathcal{O}\big\|^2_{H^{M+\eta}}.\]
    Since $\sum_kP_k^2=I$, we obtain using our definition of fractional Sobolev spaces in Section \ref{LP section}:
    \begin{align*}
        \big\|\nabla_{\tau}\mathcal{A}\big\|^2_{H^{1/2}}+\big\|\nabla\mathcal{A}\big\|^2_{H^{1/2}}&\lesssim\big\|\mathcal{O}\big\|^2_{H^{M+\eta}}+\big\|\nabla\mathcal{A}|_{\tau=0}\big\|^2_{H^{1/2}}+\int_0^{\tau}\big\|\nabla_{\tau}\mathcal{A}\big\|^2_{H^{1/2}}+\big\|\nabla\mathcal{A}\big\|^2_{H^{1/2}}d\tau'\\
        &\lesssim\big\|\mathcal{O}\big\|^2_{H^{M+1/2+\eta}}+\int_0^{\tau}\big\|\nabla_{\tau}\mathcal{A}\big\|^2_{H^{1/2}}+\big\|\nabla\mathcal{A}\big\|^2_{H^{1/2}}d\tau'.
    \end{align*}
    Finally, we obtain the desired conclusion by Gronwall.
\end{proof}

Next, we consider the component $\mathcal{B}:=\big(\nabla^{M-1}\Phi_0\big)_Y^1,$ which satisfies the equation:
\[\nabla_{\tau}\big(\nabla_{\tau}\mathcal{B}\big)+\frac{1}{{\tau}}\nabla_{\tau}\mathcal{B}-4\Delta\mathcal{B}=\psi\nabla\mathcal{B},\]
\[\mathcal{B}({\tau})=2\nabla^{M-1}\mathcal{O}\log({\tau})+ O\big({\tau}^2|\log({\tau})|^2\big),\ \nabla_{\tau}\mathcal{B}({\tau})=\frac{2\nabla^{M-1}\mathcal{O}}{\tau}+O\big({\tau}|\log({\tau})|^2\big).\]
\begin{proposition}\label{B proposition}
    $\mathcal{B}$ satisfies the estimates:
    \begin{align*}
        \big\|\nabla\mathcal{B}\big\|^2_{H^{1/2}}&\lesssim\big(1+|\log\tau|^2\big)\big\|\mathcal{O}\big\|^2_{H^{M+1/2}},\\
        \big\|\nabla_{\tau}\mathcal{B}\big\|^2_{H^{1/2}}&\lesssim\frac{1}{\tau^2}\big\|\mathcal{O}\big\|^2_{H^{M+1/2}}.
    \end{align*}
\end{proposition}
\begin{proof} We denote $\mathcal{D}=\mathcal{B}/\log\tau.$ Then $\mathcal{D}$ satisfies the equation:
    \[\nabla_{\tau}\big(\nabla_{\tau}\mathcal{D}\big)+\frac{1}{\tau}\bigg(1+\frac{2}{\log\tau}\bigg)\nabla_{\tau}\mathcal{D}-4\Delta\mathcal{D}=\psi\nabla\mathcal{D},\]
    \[\mathcal{D}({\tau})=2\nabla^{M-1}\mathcal{O}+O\big({\tau}^2|\log({\tau})|^2\big),\ \nabla_{\tau}\mathcal{D}({\tau})=O\big({\tau}|\log({\tau})|^2\big).\]
    This has the same properties needed to do energy estimates as the equation satisfied by $\mathcal{A}.$ Thus, an analogous proof gives for $\tau\in(0,1/2]$ and any $k\geq0$:
    \[\big\|P_k\nabla_{\tau}\mathcal{D}\big\|^2_{L^2}+\big\|\nabla P_k\mathcal{D}\big\|^2_{L^2}\lesssim\big\|\nabla P_k\mathcal{D}|_{\tau=0}\big\|^2_{L^2}+\int_0^{\tau}\big\|\underline{\widetilde{P}}_k\nabla\mathcal{D}\big\|_{L^2}^2+\int_0^{\tau}2^{-2k}\big\|\mathcal{D}\big\|_{H^1}^2\]\[+\int_0^{\tau}\big\|P_k\mathcal{D}\big\|_{H^1}^2+\int_0^{\tau}\big\|\underline{\widetilde{P}}_k\nabla_{\tau}\mathcal{D}\big\|_{L^2}^2+\int_0^{\tau}2^{-2k}\big\|\nabla_{\tau}\mathcal{D}\big\|_{L^2}^2+\int_0^{\tau}\frac{\textbf{1}_{[1/10,1/2]}}{\tau'|\log\tau'|}\big\|P_k\nabla_{\tau}\mathcal{D}\big\|^2_{L^2}d\tau'.\]
    We point out that the error term $\frac{1}{\tau'|\log\tau'|}\big\|P_k\nabla_{\tau}\mathcal{D}\big\|^2_{L^2}\cdot\textbf{1}_{[1/10,1/2]}$ appears on the RHS because for $\tau\in[0,1/10]$ we have $1+2/\log\tau\gtrsim1.$ This can be estimated as usual using Gronwall. We multiply by $2^k$ and sum over $k:$
    \begin{align*}
        \sum_{k\geq0}2^k\Big(\big\|P_k\nabla_{\tau}\mathcal{D}\big\|^2_{L^2}+\big\|P_k\nabla\mathcal{D}\big\|^2_{L^2}\Big)&\lesssim\big\|\mathcal{D}|_{\tau=0}\big\|^2_{L^2}+\int_0^{\tau}\big\|\mathcal{D}\big\|_{H^1}^2+\int_0^{\tau}\big\|\nabla_{\tau}\mathcal{D}\big\|_{L^2}^2+\sum_{k\geq0}2^k\big\|P_k\nabla\mathcal{D}|_{\tau=0}\big\|^2_{L^2}+\\
        &+\int_0^{\tau}\sum_{k\geq0}2^k\big\|\underline{\widetilde{P}}_k\nabla\mathcal{D}\big\|_{L^2}^2+\int_0^{\tau}\sum_{k\geq0}2^k\big\|P_k\mathcal{D}\big\|_{L^2}^2+\int_0^{\tau}\sum_{k\geq0}2^k\big\|\underline{\widetilde{P}}_k\nabla_{\tau}\mathcal{D}\big\|_{L^2}^2.
    \end{align*}
    We recall that from the standard lower order estimates we have for $\tau\in(0,1/2]$:
    \[\big\|\nabla_{\tau}\mathcal{D}\big\|^2_{L^2}+\big\|\mathcal{D}\big\|^2_{H^1}\lesssim\big\|\mathcal{O}\big\|^2_{H^{M}}.\]
    Since $\sum_kP_k^2=I$,  we obtain using our definition of fractional Sobolev spaces in Section \ref{LP section}:
    \begin{align*}
        \big\|\nabla_{\tau}\mathcal{D}\big\|^2_{H^{1/2}}+\big\|\nabla\mathcal{D}\big\|^2_{H^{1/2}}&\lesssim\big\|\mathcal{O}\big\|^2_{H^{M}}+\big\|\nabla\mathcal{D}|_{\tau=0}\big\|^2_{H^{1/2}}+\int_0^{\tau}\big\|\nabla_{\tau}\mathcal{D}\big\|^2_{H^{1/2}}+\big\|\nabla\mathcal{D}\big\|^2_{H^{1/2}}d\tau'\\
        &\lesssim\big\|\mathcal{O}\big\|^2_{H^{M+1/2}}+\int_0^{\tau}\big\|\nabla_{\tau}\mathcal{D}\big\|^2_{H^{1/2}}+\big\|\nabla\mathcal{D}\big\|^2_{H^{1/2}}d\tau'.
    \end{align*}
    We obtain by Gronwall that for any $\tau\in(0,1/2]$:
    \[\big\|\nabla_{\tau}\mathcal{D}\big\|^2_{H^{1/2}}+\big\|\nabla\mathcal{D}\big\|^2_{H^{1/2}}\lesssim\big\|\mathcal{O}\big\|^2_{H^{M+1/2}}.\]
    This implies that for any $\tau\in(0,1/2]$:
    \[\big\|\nabla\mathcal{B}\big\|^2_{H^{1/2}}\lesssim|\log\tau|^2\big\|\mathcal{O}\big\|^2_{H^{M+1/2}},\]
    \[\big\|\nabla_{\tau}\mathcal{B}\big\|^2_{H^{1/2}}\lesssim|\log\tau|^2\big\|\mathcal{O}\big\|^2_{H^{M+1/2}}+\frac{1}{\tau^2|\log\tau|^2}\big\|\mathcal{B}\big\|^2_{H^{1}}\lesssim|\log\tau|^2\big\|\mathcal{O}\big\|^2_{H^{M+1/2}}+\frac{1}{\tau^2}\big\|\mathcal{O}\big\|^2_{H^{M}}.\]
    In particular, this also implies that:
    \[\big\|\nabla_{\tau}\mathcal{B}\big\|^2_{H^{1/2}}\big|_{\tau=1/2}+\big\|\nabla\mathcal{B}\big\|^2_{H^{1/2}}\big|_{\tau=1/2}\lesssim\big\|\mathcal{O}\big\|^2_{H^{M+1/2}}.\]
    Using the equation for $\mathcal{B}$ on the time interval $\tau\in[1/2,1]$, we repeat the energy estimate that we did for $\mathcal{A}:$
    \[\sum_{k\geq0}2^k\Big(\big\|P_k\nabla_{\tau}\mathcal{B}\big\|^2_{L^2}+\big\|P_k\nabla\mathcal{B}\big\|^2_{L^2}\Big)\lesssim\big\|\mathcal{B}\big|_{\tau=\frac{1}{2}}\big\|^2_{L^2}+\int_{\frac{1}{2}}^{\tau}\big\|\mathcal{B}\big\|_{H^1}^2+\int_{\frac{1}{2}}^{\tau}\big\|\nabla_{\tau}\mathcal{B}\big\|_{L^2}^2+\sum_{k\geq0}2^k\big\|P_k\nabla\mathcal{B}\big|_{\tau=\frac{1}{2}}\big\|^2_{L^2}+\]\[+\sum_{k\geq0}2^k\big\|P_k\nabla_{\tau}\mathcal{B}\big|_{\tau=\frac{1}{2}}\big\|^2_{L^2}+\int_{\frac{1}{2}}^{\tau}\sum_{k\geq0}2^k\big\|\underline{\widetilde{P}}_k\nabla\mathcal{B}\big\|_{L^2}^2+\int_{\frac{1}{2}}^{\tau}\sum_{k\geq0}2^k\big\|P_k\mathcal{B}\big\|_{L^2}^2+\int_{\frac{1}{2}}^{\tau}\sum_{k\geq0}2^k\big\|\underline{\widetilde{P}}_k\nabla_{\tau}\mathcal{B}\big\|_{L^2}^2.\]
    We recall that from the standard lower order estimates we have for $\tau\in[1/2,1]$:
    \[\big\|\nabla_{\tau}\mathcal{B}\big\|^2_{L^2}+\big\|\mathcal{B}\big\|^2_{H^1}\lesssim\big\|\mathcal{O}\big\|^2_{H^{M}}.\]
    As a result, we have that for $\tau\in[1/2,1]$:
    \[\big\|\nabla_{\tau}\mathcal{B}\big\|^2_{H^{1/2}}+\big\|\nabla\mathcal{B}\big\|^2_{H^{1/2}}\lesssim\big\|\mathcal{O}\big\|^2_{H^{M+1/2}}+\int_{\frac{1}{2}}^{\tau}\big\|\nabla_{\tau}\mathcal{B}\big\|^2_{H^{1/2}}+\int_{\frac{1}{2}}^{\tau}\big\|\nabla\mathcal{B}\big\|^2_{H^{1/2}}.\]
    We combine the estimates for $\tau\in(0,1/2]$ and $\tau\in[1/2,1]$ to obtain the conclusion.
\end{proof}

\subsection{Top order estimates}\label{first system top order estimates section}

We prove top order estimates for the regular quantities $\big(\nabla^M\Phi_0\big)_J,\nabla^M\Phi_1,\ldots\nabla^M\Phi_I$ as outlined in Section~\ref{first model system intro section}. The analysis requires a precise understanding of the behavior of the $P_k$ projections of each component. We treat separately the low frequency regime $\tau\in(0,2^{-k-1}]$ in Section~\ref{first system low fre estimates section} and the high frequency regime $\tau\in[2^{-k-1},1]$ in Section~\ref{first system high fre estimates section}. In Section~\ref{singular component outline section}, we state the top order estimates of \cite[Section~7]{Cmain} for the singular component, and we explain their proof in analogy to the argument of the present paper. Finally, we combine the estimates in Section~\ref{first system combined estimates section} to complete the proof of Theorem~\ref{main theorem first system}.

In this section, we often work with the new time variable $t=2^k\tau$. We notice that we have a similar result to Lemma~\ref{volume form lemma}, which allows us to control the error terms resulting from time derivatives of the metric and volume form. We note that for these terms we can apply Gronwall on the whole interval $t\in(0,2^k].$

\begin{lemma}\label{volume form lemma t}
    For any smooth horizontal tensor $F$ defined on $\mathcal{M},$ we have:
    \[\frac{1}{2}\frac{d}{dt}\big\|F\big\|_{L^2(S_t)}^2=\int_{S_{t}}F\cdot\nabla_{t}FdVol_{\slashed{g}}+O\Big(2^{-2k}t\big\|F\big\|_{L^2(S_t)}^2\Big).\]
\end{lemma}
\begin{proof}
    The proof follows from Lemma~\ref{volume form lemma}, since $\partial_{t}=2^{-k}\partial_{\tau}.$
\end{proof}

\subsubsection{Low frequency regime estimates}\label{first system low fre estimates section}
We prove a low frequency regime estimate for the regular components $P_k\big(\nabla^M\Phi_0\big)_J,P_k\nabla^M\Phi_1,\ldots,P_k\nabla^M\Phi_I$, for each $k\geq0$. The main idea is to propagate for $\tau\in(0,2^{-k-1}]$ the $L^2$ bounds satisfied by the asymptotic data at $\mathcal{I}^-$, using $\nabla_{\tau}$ as a multiplier.
\begin{proposition}\label{low frequency regular proposition} For any $k\geq0$ and $\tau\leq2^{-k-1}$, we have the estimate:
\begin{align*}
    &\big\|P_k\nabla_{\tau}\big(\nabla^M\Phi_0\big)_J\big\|^2_{L^2}+\big\|\nabla P_k\big(\nabla^M\Phi_0\big)_J\big\|^2_{L^2}+2^{2k}\big\|P_k\big(\nabla^M\Phi_0\big)_J\big\|^2_{L^2}+\\
    &+\sum_{i=1}^I\bigg(\big\|P_k\nabla_{\tau}\nabla^M\Phi_i\big\|^2_{L^2}+\big\|\nabla P_k\nabla^M\Phi_i\big\|^2_{L^2}+2^{2k}\big\|P_k\nabla^M\Phi_i\big\|^2_{L^2}\bigg)\lesssim\\
    &\lesssim\big\|\nabla P_k\mathfrak{h}_M\big\|^2_{L^2}+2^{2k}\big\|P_k\mathfrak{h}_M\big\|^2_{L^2}+\sum_{i=1}^I\big\|\nabla P_k\nabla^M\Phi_i^0\big\|^2_{L^2}+\sum_{i=1}^I2^{2k}\big\|P_k\nabla^M\Phi_i^0\big\|^2_{L^2}+\\
    &+\sum_{i=1}^I\int_0^{\tau}\bigg(\frac{1}{2^{3k}}\big\|\nabla^M\Phi_i\big\|_{H^1}^2+\frac{(\tau')^2}{2^{k}}\big\|\nabla_{\tau}\nabla^M\Phi_i\big\|_{L^2}^2\bigg)d\tau'+\int_0^{\tau}\frac{(\tau')^2}{2^{k}}\big\|\nabla_{\tau}\big(\nabla^M\Phi_0\big)_J\big\|_{L^2}^2d\tau'+\\
    &+\int_0^{\tau}\frac{1}{2^{3k}}\big\|\big(\nabla^M\Phi_0\big)_J\big\|_{H^1}^2d\tau'+\sum_{i=0}^I\int_0^{\tau}\frac{1}{2^{k}}\big\|P_kF_M^i\big\|_{L^2}^2d\tau'+\int_0^{\tau}\frac{1}{2^{k}}\big\|P_k\big(\psi\nabla\big(\nabla^M\Phi_{0}\big)_Y\big)\big\|_{L^2}^2d\tau'.
\end{align*}
\end{proposition}
\begin{proof}
    We denote $\xi_{0}=\big(\nabla^M\Phi_0\big)_J$ and for 
    $1\leq i\leq I$ we denote $\xi_i=\nabla^M\Phi_i$. We introduce the new time variable $t=2^k\tau$. Equations \eqref{equation for Phi i} and \eqref{equation for Phi J} with $\sigma=1,m=M$ can be written as follows for all $0\leq i\leq I$:
    \[\nabla_{t}\big(\nabla_{t}\xi_i\big)+\frac{1}{t}\nabla_{t}\xi_i-\frac{4}{2^{2k}}\cdot\Delta\xi_i=\sum_{j=0}^I\frac{1}{2^{2k}}\cdot\psi\nabla\xi_j+\frac{1}{2^{2k}}\cdot F_M^i+\frac{1}{2^{2k}}\cdot\psi\nabla\big(\nabla^M\Phi_{0}\big)_Y.\]
    For any $k\geq0,$ we apply $P_k$ to each equation:
    \begin{align*}
        &\nabla_{t}\big(P_k\nabla_{t}\xi_i\big)+\frac{1}{t}P_k\nabla_{t}\xi_i-\frac{4}{2^{2k}}\cdot\Delta P_k\xi_i=\sum_{j=0}^I\frac{1}{2^{2k}}\cdot\psi\nabla P_k\xi_j+\sum_{j=0}^I\frac{1}{2^{2k}}\psi[P_k,\nabla]\xi_j+\\
        &+\sum_{j=0}^I\frac{1}{2^{2k}}[P_k,\psi]\nabla\xi_j+\frac{1}{2^{2k}}\cdot P_kF_M^i+[\nabla_t,P_k]\nabla_{t}\xi_i+\frac{1}{2^{2k}}\cdot P_k\big(\psi\nabla\big(\nabla^M\Phi_0\big)_Y\big).
    \end{align*}
    We contract each equation with $P_k\nabla_t\xi_i$ and integrate by parts to obtain the energy estimate:
    \[\sum_{i=0}^I\big\|P_k\nabla_{t}\xi_i\big\|^2_{L^2}+\sum_{i=0}^I\frac{1}{2^{2k}}\big\|\nabla P_k\xi_i\big\|^2_{L^2}\lesssim\sum_{i=0}^I\frac{1}{2^{2k}}\big\|\nabla P_k\xi_i^0\big\|^2_{L^2}+\sum_{i=0}^I\int_0^t\frac{1}{2^{2k}}\big\|\nabla P_k\xi_i\big\|_{L^2}\cdot\big\|\nabla[P_k,\nabla_t]\xi_i\big\|_{L^2}+\]\[+\sum_{i=0}^I\int_0^t\frac{1}{2^{2k}}\big\|\nabla P_k\xi_i\big\|_{L^2}\cdot\big\|[\nabla,\nabla_t]P_k\xi_i\big\|_{L^2}+\sum_{i=0}^I\int_0^t\frac{1}{2^{2k}}\big\|\nabla P_k\xi_i\big\|_{L^2}\cdot\big\|P_k\nabla_t\xi_i\big\|_{L^2}+\]\[+\sum_{i=0}^I\int_0^t\frac{1}{2^{2k}}\big\|P_k\nabla_t\xi_i\big\|_{L^2}\cdot\big\|[P_k,\nabla]\xi_i\big\|_{L^2}+\sum_{i=0}^I\int_0^t\frac{1}{2^{2k}}\big\|P_k\nabla_t\xi_i\big\|_{L^2}\cdot\big\|[P_k,\psi]\nabla\xi_i\big\|_{L^2}+\]\[+\sum_{i=0}^I\int_0^t\big\|P_k\nabla_t\xi_i\big\|_{L^2}\cdot\big\|[\nabla_t,P_k]\nabla_{t}\xi_i\big\|_{L^2}+\sum_{i=0}^I\int_0^t\frac{1}{2^{2k}}\big\|P_k\nabla_t\xi_i\big\|_{L^2}\cdot\bigg(\big\|P_kF_M^i\big\|_{L^2}+\big\|P_k\big(\psi\nabla\big(\nabla^M\Phi_{0}\big)_Y\big)\big\|_{L^2}\bigg)dt'.\]
    We note that we dropped the bulk term with a favorable sign in the above estimate. We use Gronwall for $t\in\big[0,1/2\big]:$
    \[\sum_{i=0}^I\big\|P_k\nabla_{t}\xi_i\big\|^2_{L^2}+\sum_{i=0}^I\frac{1}{2^{2k}}\big\|\nabla P_k\xi_i\big\|^2_{L^2}\lesssim\sum_{i=0}^I\frac{1}{2^{2k}}\big\|\nabla P_k\xi_i^0\big\|^2_{L^2}+\sum_{i=0}^I\int_0^t\frac{1}{2^{2k}}\big\|\nabla[P_k,\nabla_t]\xi_i\big\|_{L^2}^2+\]\[+\sum_{i=0}^I\int_0^t\frac{1}{2^{2k}}\big\|[\nabla,\nabla_t]P_k\xi_i\big\|_{L^2}^2+\sum_{i=0}^I\int_0^t\frac{1}{2^{4k}}\big\|[P_k,\nabla]\xi_i\big\|_{L^2}^2+\sum_{i=0}^I\int_0^t\frac{1}{2^{4k}}\big\|[P_k,\psi]\nabla\xi_i\big\|_{L^2}^2+\]\[+\sum_{i=0}^I\int_0^t\big\|[\nabla_t,P_k]\nabla_{t}\xi_i\big\|_{L^2}^2+\sum_{i=0}^I\int_0^t\frac{1}{2^{4k}}\big\|P_kF_i\big\|_{L^2}^2+\int_0^t\frac{1}{2^{4k}}\big\|P_k\big(\psi\nabla\big(\nabla^M\Phi_{0}\big)_Y\big)\big\|_{L^2}^2.\]
    We use Lemma \ref{Litt Paley lemma} in order to bound the commutation terms. Thus, we get:
    \begin{align*}
        &\sum_{i=0}^I\big\|P_k\nabla_{t}\xi_i\big\|^2_{L^2}+\sum_{i=0}^I\frac{1}{2^{2k}}\big\|\nabla P_k\xi_i\big\|^2_{L^2}\lesssim\sum_{i=0}^I\frac{1}{2^{2k}}\big\|\nabla P_k\xi_i^0\big\|^2_{L^2}+\sum_{i=0}^I\int_0^t\frac{1}{2^{6k}}\big\|\xi_i\big\|_{H^1}^2+\\
        &+\sum_{i=0}^I\int_0^t\frac{(t')^2}{2^{4k}}\big\|\nabla_{t}\xi_i\big\|_{L^2}^2+\sum_{i=0}^I\int_0^t\frac{1}{2^{4k}}\big\|P_kF_M^i\big\|_{L^2}^2+\int_0^t\frac{1}{2^{4k}}\big\|P_k\big(\psi\nabla\big(\nabla^M\Phi_{0}\big)_Y\big)\big\|_{L^2}^2.
    \end{align*}
    We change variables to $\tau$ and get for all $\tau\in[0,2^{-k-1}]:$
    \begin{align*}
        &\sum_{i=0}^I\big\|P_k\nabla_{\tau}\xi_i\big\|^2_{L^2}+\big\|\nabla P_k\xi_i\big\|^2_{L^2}\lesssim\sum_{i=0}^I\big\|\nabla P_k\xi_i^0\big\|^2_{L^2}+\sum_{i=0}^I\int_0^{\tau}\frac{1}{2^{3k}}\big\|\xi_i\big\|_{H^1}^2+\\
        &+\sum_{i=0}^I\int_0^{\tau}\frac{(\tau')^2}{2^{k}}\big\|\nabla_{\tau}\xi_i\big\|_{L^2}^2+\sum_{i=0}^I\int_0^{\tau}\frac{1}{2^{k}}\big\|P_kF_M^i\big\|_{L^2}^2+\int_0^{\tau}\frac{1}{2^{k}}\big\|P_k\big(\psi\nabla\big(\nabla^M\Phi_{0}\big)_Y\big)\big\|_{L^2}^2.
    \end{align*}
    Next, using Lemma~\ref{volume form lemma} and \eqref{LP refined est 2}, we also have the bound for all $\tau\in[0,2^{-k-1}]:$
    \begin{align*}
        \sum_{i=0}^I\big\|P_k\xi_i\big\|^2_{L^2}&\lesssim\sum_{i=0}^I\big\|P_k\xi_i^0\big\|^2_{L^2}+\sum_{i=0}^I2^{-k}\int_0^{\tau}\big\|\nabla_{\tau}P_k\xi_i\big\|^2_{L^2}+\sum_{i=0}^I2^{k}\int_0^{\tau}\big\|P_k\xi_i\big\|^2_{L^2}\\
        &\lesssim\sum_{i=0}^I\big\|P_k\xi_i^0\big\|^2_{L^2}+\sum_{i=0}^I2^{-k}\int_0^{\tau}\big\|P_k\nabla_{\tau}\xi_i\big\|^2_{L^2}+\sum_{i=0}^I2^{-5k}\int_0^{\tau}\big\|\xi_i\big\|^2_{H^1}+\sum_{i=0}^I2^{k}\int_0^{\tau}\big\|P_k\xi_i\big\|^2_{L^2}.
    \end{align*}
    Using Gronwall and the previous estimate, we conclude that:
    \[\sum_{i=0}^I2^{2k}\big\|P_k\xi_i\big\|^2_{L^2}\lesssim\sum_{i=0}^I2^{2k}\big\|P_k\xi_i^0\big\|^2_{L^2}+\sum_{i=0}^I\big\|\nabla P_k\xi_i^0\big\|^2_{L^2}+\sum_{i=0}^I\int_0^{\tau}\frac{1}{2^{3k}}\big\|\xi_i\big\|_{H^1}^2+\]\[+\sum_{i=0}^I\int_0^{\tau}\frac{(\tau')^2}{2^{k}}\big\|\nabla_{\tau}\xi_i\big\|_{L^2}^2+\sum_{i=0}^I\int_0^{\tau}\frac{1}{2^{k}}\big\|P_kF_M^i\big\|_{L^2}^2+\int_0^{\tau}\frac{1}{2^{k}}\big\|P_k\big(\psi\nabla\big(\nabla^M\Phi_{0}\big)_Y\big)\big\|_{L^2}^2.\]
\end{proof}

\subsubsection{High frequency regime estimates}\label{first system high fre estimates section}
We prove a high frequency regime estimate for the regular components $P_k\big(\nabla^M\Phi_0\big)_J,P_k\nabla^M\Phi_1,\ldots,P_k\nabla^M\Phi_I$, for each $k\geq0$. The idea of the proof is to use $2^k\tau\nabla_{\tau}$ as a multiplier for $\tau\in[2^{-k-1},1].$ As in the previous section, our estimates are simplified by the presence of bulk terms with favorable signs. We also notice that the same argument applies to prove the high frequency regime estimate for the singular component $P_k\big(\nabla^M\Phi_0\big)_Y$, see already Section~\ref{singular component outline section}.

\begin{proposition}\label{high frequency forward estimate} For any $k\geq0,\ \tau\in\big[2^{-k-1},1\big]$ we have the estimate for the regular components:
    \[\tau\big\|P_k\nabla_{\tau}\big(\nabla^M\Phi_{0}\big)_J\big\|^2_{L^2}+\frac{1}{\tau}\big\|P_k\big(\nabla^M\Phi_{0}\big)_J\big\|^2_{L^2}+\tau\big\|\nabla P_k\big(\nabla^M\Phi_{0}\big)_J\big\|^2_{L^2}+\]\[+\sum_{i=1}^I\tau\big\|P_k\nabla_{\tau}\nabla^M\Phi_i\big\|^2_{L^2}+\sum_{i=1}^I\frac{1}{\tau}\big\|P_k\nabla^M\Phi_i\big\|^2_{L^2}+\sum_{i=1}^I\tau\big\|\nabla P_k\nabla^M\Phi_i\big\|^2_{L^2}\lesssim\]\[\lesssim\frac{1}{2^k}\bigg(\big\|P_k\nabla_{\tau}\big(\nabla^M\Phi_{0}\big)_J\big\|^2_{L^2}+2^{2k}\big\|P_k\big(\nabla^M\Phi_{0}\big)_J\big\|^2_{L^2}+\big\|\nabla P_k\big(\nabla^M\Phi_{0}\big)_J\big\|^2_{L^2}\bigg)\bigg|_{\tau=2^{-k-1}}+\int_{2^{-k-1}}^{\tau}\tau'\big\|\underline{\widetilde{P}}_k\big(\nabla^M\Phi_{0}\big)_J\big\|_{L^2}^2\]\[+\sum_{i=1}^I\frac{1}{2^k}\bigg(\big\|P_k\nabla_{\tau}\nabla^M\Phi_i\big\|^2_{L^2}+2^{2k}\big\|P_k\nabla^M\Phi_i\big\|^2_{L^2}+\big\|\nabla P_k\nabla^M\Phi_i\big\|^2_{L^2}\bigg)\bigg|_{\tau=2^{-k-1}}+\sum_{i=1}^I\int_{2^{-k-1}}^{\tau}\tau'\big\|\underline{\widetilde{P}}_k\nabla^M\Phi_i\big\|_{L^2}^2\]\[+\int_{2^{-k-1}}^{\tau}(\tau')^3\big\|\underline{\widetilde{P}}_k\nabla\big(\nabla^M\Phi_{0}\big)_J\big\|_{L^2}^2+\int_{2^{-k-1}}^{\tau}(\tau')^3\big\|\underline{\widetilde{P}}_k\nabla_{\tau}\big(\nabla^M\Phi_{0}\big)_J\big\|_{L^2}^2+\int_{2^{-k-1}}^{\tau}\frac{\tau'}{2^{2k}}\big\|\big(\nabla^M\Phi_{0}\big)_J\big\|_{H^1}^2\]\[+\sum_{i=1}^I\int_{2^{-k-1}}^{\tau}\bigg((\tau')^3\big\|\underline{\widetilde{P}}_k\nabla^{M+1}\Phi_i\big\|_{L^2}^2+(\tau')^3\big\|\underline{\widetilde{P}}_k\nabla_{\tau}\nabla^M\Phi_i\big\|_{L^2}^2+\frac{\tau'}{2^{2k}}\big\|\nabla^{M}\Phi_i\big\|_{H^1}^2+\frac{(\tau')^3}{2^{2k}}\big\|\nabla_{\tau}\nabla^M\Phi_i\big\|_{L^2}^2\bigg)d\tau'\]\[+\int_{2^{-k-1}}^{\tau}\frac{(\tau')^3}{2^{2k}}\big\|\nabla_{\tau}\big(\nabla^M\Phi_{0}\big)_J\big\|_{L^2}^2+\int_{2^{-k-1}}^{\tau}\tau'\big\|P_k\big(\psi\nabla\big(\nabla^M\Phi_{0}\big)_Y\big)\big\|_{L^2}^2+\sum_{i=0}^I\int_{2^{-k-1}}^{\tau}\tau'\big\|P_kF_M^i\big\|_{L^2}^2.\]
\end{proposition}
\begin{proof}We denote $\xi_{0}=\big(\nabla^M\Phi_{0}\big)_J$ and for 
$1\leq i\leq I$ we denote $\xi_i=\nabla^M\Phi_i$. We introduce the new time variable $t=2^k\tau$. As before, equations \eqref{equation for Phi i} and \eqref{equation for Phi J} with $\sigma=1,m=M$ can be written for all $0\leq i\leq I$:
    \[\nabla_{t}\big(\nabla_{t}\xi_i\big)+\frac{1}{t}\nabla_{t}\xi_i-\frac{4}{2^{2k}}\cdot\Delta\xi_i=\sum_{j=0}^I\frac{1}{2^{2k}}\cdot\psi\nabla\xi_j+\frac{1}{2^{2k}}\cdot F_M^i+\frac{1}{2^{2k}}\cdot\psi\nabla\big(\nabla^M\Phi_{0}\big)_Y.\]
    We multiply by $\sqrt{t}$ to get for all $0\leq i\leq I$:
    \[\nabla_{t}\big(\nabla_{t}(\xi_i\sqrt{t})\big)+\frac{1}{4t^2}\xi_i\sqrt{t}-\frac{4}{2^{2k}}\cdot\Delta\xi_i\sqrt{t}=\sum_{j=0}^I\frac{1}{2^{2k}}\cdot\psi\nabla\xi_j\sqrt{t}+\frac{1}{2^{2k}}\cdot F_M^i\sqrt{t}+\frac{1}{2^{2k}}\cdot\psi\nabla\big(\nabla^M\Phi_{0}\big)_Y\sqrt{t}.\]
    For any $k\geq0$, we apply $P_k$ to obtain the equations for all $0\leq i\leq I$:
    \[\nabla_{t}\big(P_k\nabla_{t}(\xi_i\sqrt{t})\big)+\frac{1}{4t^2}P_k\xi_i\sqrt{t}-\frac{4}{2^{2k}}\Delta P_k\xi_i\sqrt{t}=\sum_{j=0}^I\frac{1}{2^{2k}}\psi\nabla P_k\xi_j\sqrt{t}+\sum_{j=0}^I\frac{1}{2^{2k}}\psi[P_k,\nabla]\xi_j\sqrt{t}+\]\[+\sum_{j=0}^I\frac{1}{2^{2k}}[P_k,\psi]\nabla \xi_j\sqrt{t}+\frac{P_k\big(F_M^i\sqrt{t}+\psi\nabla\big(\nabla^M\Phi_{0}\big)_Y\sqrt{t}\big)}{2^{2k}}+[\nabla_t,P_k]\nabla_{t}\xi_i\sqrt{t}\]
    We contract each equation with $P_k\nabla_t(\xi_i\sqrt{t})$ and integrate by parts to obtain the following energy estimate:
    \[\big\|P_k\nabla_{t}\xi_i\sqrt{t}\big\|^2_{L^2}+\frac{1}{t^2}\big\|P_k\xi_i\sqrt{t}\big\|^2_{L^2}+\frac{1}{2^{2k}}\big\|\nabla P_k\xi_i\sqrt{t}\big\|^2_{L^2}+\int_{1/2}^t\frac{1}{(t')^2}\big\|P_k\xi_i\big\|_{L^2}^2dt'\lesssim\]\[\lesssim\bigg(\big\|P_k\nabla_{t}\xi_i\big\|^2_{L^2}+\big\|P_k\xi_i\big\|^2_{L^2}+\frac{1}{2^{2k}}\big\|\nabla P_k\xi_i\big\|^2_{L^2}\bigg)\bigg|_{t=1/2}+\]\[+\int_{1/2}^t\int_{S^n}\frac{1}{(t')^2}\big|P_k\xi_i\sqrt{t'}\big|\cdot\big|[P_k,\nabla_t]\xi_i\sqrt{t'}\big|+\int_{1/2}^t\int_{S^n}\frac{1}{2^{2k}}\big|\nabla P_k\xi_i\sqrt{t'}\big|\cdot\big|\nabla[P_k,\nabla_t]\xi_i\sqrt{t'}\big|+\]\[+\int_{1/2}^t\int_{S^n}\frac{1}{2^{2k}}\big|\nabla P_k\xi_i\sqrt{t'}\big|\cdot\big|[\nabla,\nabla_t]P_k\xi_i\sqrt{t'}\big|+\sum_{j=0}^I\int_{1/2}^t\int_{S^n}\frac{1}{2^{2k}}\big|P_k\nabla_t\xi_i\sqrt{t'}\big|\cdot\big|\nabla P_k\xi_j\sqrt{t'}\big|+\]\[+\sum_{j=0}^I\int_{1/2}^t\int_{S^n}\frac{1}{2^{2k}}\big|P_k\nabla_t\xi_i\sqrt{t'}\big|\cdot\big|[P_k,\nabla]\xi_j\sqrt{t'}\big|+\sum_{j=0}^I\int_{1/2}^t\int_{S^n}\frac{1}{2^{2k}}\big|P_k\nabla_t\xi_i\sqrt{t'}\big|\cdot\big|[P_k,\psi]\nabla\xi_j\sqrt{t'}\big|+\]\[+\int_{1/2}^t\int_{S^n}\big|P_k\nabla_t\xi_i\sqrt{t'}\big|\cdot\big|[\nabla_t,P_k]\nabla_{t}\xi_i\sqrt{t'}\big|+\int_{1/2}^t\int_{S^n}\frac{1}{2^{2k}}\big|P_k\nabla_t\xi_i\sqrt{t'}\big|\cdot\big|P_k\big(F_M^i\sqrt{t'}+\psi\nabla\big(\nabla^M\Phi_{0}\big)_Y\sqrt{t'}\big)\big|.\]
    We point out that the bulk term with a favorable sign provides a significant simplification for our analysis. On the other hand, in the analysis of the second model system the corresponding term will create several complications. We use Gronwall for $t\in\big[1/2,2^k\big],$ and the bounds in Lemma \ref{Litt Paley lemma refined} to get for all $0\leq i\leq I$:
    \[\big\|P_k\nabla_{t}\xi_i\sqrt{t}\big\|^2_{L^2}+\frac{1}{t^2}\big\|P_k\xi_i\sqrt{t}\big\|^2_{L^2}+\frac{1}{2^{2k}}\big\|\nabla P_k\xi_i\sqrt{t}\big\|^2_{L^2}\lesssim\bigg(\big\|P_k\nabla_{t}\xi_i\big\|^2_{L^2}+\big\|P_k\xi_i\big\|^2_{L^2}+\frac{1}{2^{2k}}\big\|\nabla P_k\xi_i\big\|^2_{L^2}\bigg)\bigg|_{t=1/2}\]\[+\int_{1/2}^t\frac{1}{2^{2k}t'}\big\|P_k\xi_i\sqrt{t'}\big\|_{L^2}\cdot\Big(\big\|\underline{\widetilde{P}}_k\xi_i\sqrt{t'}\big\|_{L^2}+2^{-k}\big\|\xi_i\sqrt{t'}\big\|_{L^2}\Big)+\sum_{j=0}^I\int_{1/2}^t\frac{1}{2^{2k}}\big\|P_k\nabla_t\xi_i\sqrt{t'}\big\|_{L^2}\cdot\big\|\nabla P_k\xi_j\sqrt{t'}\big\|_{L^2}+\]\[+\int_{1/2}^t\frac{t'}{2^{4k}}\big\|\nabla P_k\xi_i\sqrt{t'}\big\|_{L^2}\cdot\bigg(\big\|\underline{\widetilde{P}}_k\nabla\xi_i\sqrt{t'}\big\|_{L^2}+2^{-k}\big\|\xi_i\sqrt{t'}\big\|_{H^1}\bigg)+\sum_{j=0}^I\int_{1/2}^t\frac{1}{2^{3k}}\big\|P_k\nabla_t\xi_i\sqrt{t'}\big\|_{L^2}\cdot\big\|\xi_j\sqrt{t'}\big\|_{H^1}+\]\[+\int_{1/2}^t\frac{t'}{2^{2k}}\big\|P_k\nabla_t\xi_i\sqrt{t'}\big\|_{L^2}\Big(\big\|\underline{\widetilde{P}}_k\nabla_t\xi_i\sqrt{t'}\big\|_{L^2}+2^{-k}\big\|\nabla_t\xi_i\sqrt{t'}\big\|_{L^2}\Big)+\]
    \[+\int_{1/2}^t\frac{1}{2^{2k}}\big\|P_k\nabla_t\xi_i\sqrt{t'}\big\|_{L^2}\big\|P_kF_M^i\sqrt{t'}\big\|_{L^2}+\int_{1/2}^t\frac{1}{2^{2k}}\big\|P_k\nabla_t\xi_i\sqrt{t'}\big\|_{L^2}\big\|P_k\big(\psi\nabla\big(\nabla^M\Phi_0\big)_Y\big)\sqrt{t'}\big\|_{L^2}\]
    Once again we use Gronwall for $t\in\big[1/2,2^k\big]$ to get for all $0\leq i\leq I$:
    \[\big\|P_k\nabla_{t}\xi_i\sqrt{t}\big\|^2_{L^2}+\frac{1}{t^2}\big\|P_k\xi_i\sqrt{t}\big\|^2_{L^2}+\frac{1}{2^{2k}}\big\|\nabla P_k\xi_i\sqrt{t}\big\|^2_{L^2}\lesssim\bigg(\big\|P_k\nabla_{t}\xi_i\big\|^2_{L^2}+\big\|P_k\xi_i\big\|^2_{L^2}+\frac{1}{2^{2k}}\big\|\nabla P_k\xi_i\big\|^2_{L^2}\bigg)\bigg|_{t=1/2}\]\[+\int_{1/2}^t\frac{1}{2^{3k}}\big\|\underline{\widetilde{P}}_k\xi_i\sqrt{t'}\big\|_{L^2}^2+\int_{1/2}^t\frac{(t')^2}{2^{5k}}\big\|\underline{\widetilde{P}}_k\nabla\xi_i\sqrt{t'}\big\|_{L^2}^2+\int_{1/2}^t\frac{(t')^2}{2^{3k}}\big\|\underline{\widetilde{P}}_k\nabla_t\xi_i\sqrt{t'}\big\|_{L^2}^2+\sum_{j=0}^I\int_{1/2}^t\frac{1}{2^{3k}}\big\|\nabla P_k\xi_j\sqrt{t'}\big\|_{L^2}^2\]\[+\sum_{j=0}^I\int_{1/2}^t\frac{1}{2^{5k}}\big\|\xi_j\sqrt{t'}\big\|_{H^1}^2+\int_{1/2}^t\frac{(t')^2}{2^{5k}}\big\|\nabla_t\xi_i\sqrt{t'}\big\|_{L^2}^2+\int_{1/2}^t\frac{1}{2^{3k}}\big\|P_kF_M^i\sqrt{t'}\big\|_{L^2}^2+\int_{1/2}^t\frac{1}{2^{3k}}\big\|P_k\big(\psi\nabla\big(\nabla^M\Phi_0\big)_Y\big)\sqrt{t'}\big\|_{L^2}^2\]
    Finally, we sum the above estimates for all $0\leq i\leq I$ to obtain:
    \[\sum_{i=0}^It\big\|P_k\nabla_{t}\xi_i\big\|^2_{L^2}+\sum_{i=0}^I\frac{1}{t}\big\|P_k\xi_i\big\|^2_{L^2}+\sum_{i=0}^I\frac{t}{2^{2k}}\big\|\nabla P_k\xi_i\big\|^2_{L^2}\lesssim\]\[\lesssim\sum_{i=0}^I\bigg(\big\|P_k\nabla_{t}\xi_i\big\|^2_{L^2}+\big\|P_k\xi_i\big\|^2_{L^2}+\frac{1}{2^{2k}}\big\|\nabla P_k\xi_i\big\|^2_{L^2}\bigg)\bigg|_{t=1/2}+\sum_{i=0}^I\int_{1/2}^t\frac{1}{2^{3k}}\big\|\underline{\widetilde{P}}_k\xi_i\sqrt{t'}\big\|_{L^2}^2+\]\[+\sum_{i=0}^I\int_{1/2}^t\frac{(t')^2}{2^{5k}}\big\|\underline{\widetilde{P}}_k\nabla\xi_i\sqrt{t'}\big\|_{L^2}^2+\sum_{i=0}^I\int_{1/2}^t\frac{(t')^2}{2^{3k}}\big\|\underline{\widetilde{P}}_k\nabla_t\xi_i\sqrt{t'}\big\|_{L^2}^2+\sum_{i=0}^I\int_{1/2}^t\frac{1}{2^{5k}}\big\|\xi_i\sqrt{t'}\big\|_{H^1}^2+\]\[+\sum_{i=0}^I\int_{1/2}^t\frac{(t')^2}{2^{5k}}\big\|\nabla_t\xi_i\sqrt{t'}\big\|_{L^2}^2+\sum_{i=0}^I\int_{1/2}^t\frac{t'}{2^{3k}}\big\|P_kF_M^i\big\|_{L^2}^2+\int_{1/2}^t\frac{t'}{2^{3k}}\big\|P_k\big(\psi\nabla\big(\nabla^M\Phi_0\big)_Y\big)\big\|_{L^2}^2.\]
    We change variables to $\tau$ in order to obtain the conclusion.
\end{proof}

\subsubsection{Singular component estimates}\label{singular component outline section}

The singular component decouples from the rest of the system, so it can be studied separately from the regular quantities. Relying on the lower order estimates for the singular component in Section~\ref{lower order estimates section}, the proof of the top order estimates is done in \cite[Theorem~7.2]{Cmain}, where we prove:
\begin{align}
    &\tau^2\big\|\nabla_{\tau}\big(\nabla^M\Phi_0\big)_Y\big\|^2_{H^{1/2}}+\tau^2\big\|\nabla\big(\nabla^M\Phi_0\big)_Y\big\|^2_{H^{1/2}}\lesssim\big\|\mathcal{O}\big\|^2_{H^{M+1}},\label{main estimate PHI Y top order} \\
    &\sum_{m=0}^M\big\|\big(\nabla^m\Phi_0\big)_Y\big\|^2_{H^{1}}\lesssim\big(1+|\log\tau|^2\big) \big\|\mathcal{O}\big\|^2_{H^{M+1}}.\label{practical estimate PHI Y top order}
\end{align}

We outline the strategy of this proof, which follows similar steps to the above argument for the regular components. In the high frequency regime, the estimate is entirely analogous to the one proved in Proposition~\ref{high frequency forward estimate}. Moreover, we point out that the only error terms on the right hand side come from the singular component itself, since there are no inhomogeneous terms in \eqref{equation for Phi Y}. Additionally, once we prove \eqref{main estimate PHI Y top order} and \eqref{practical estimate PHI Y top order}, we also have the high frequency estimate:
\begin{equation}\label{high freq estimate improved singular component}
    \sum_{\tau>2^{-k-1}}2^k\tau\big\|\nabla P_k\big(\nabla^M\Phi_{0}\big)_Y\big\|_{L^2}^2\lesssim\big\|\mathcal{O}\big\|_{H^{M+1}}^2
\end{equation}

In the low frequency regime we take a similar approach to the regular components. However, we need to account for the singular behavior of the solution at $\tau=0,$ which can be seen in the expansions:
\begin{align*}
    &P_k\big(\nabla^M\Phi_0\big)_Y({\tau})=2P_k\nabla^M\mathcal{O}\log({2^k\tau})+R_k\nabla^M\mathcal{O}+O\big({\tau}^2|\log({\tau})|^2\big),\\
    &P_k\nabla_{2^{-k}\partial_\tau}\big(\nabla^M\Phi_0\big)_Y({\tau})=2P_k\nabla^M\mathcal{O}\frac{1}{2^k\tau}+O\big({\tau}|\log({\tau})|^2\big).
\end{align*}
We notice that $R_k\nabla^M\mathcal{O}$ on the right hand side is defined at $\tau=0$ and extended by Lie transport in time, according to Lemma \ref{difference of projections lemma}. The strategy is to prove estimates on the equation for $P_k\big(\nabla^M\Phi_0\big)_Y/\log t$ where $t=2^k\tau.$ It is essential to subtract off the term with $R_k\nabla^M\mathcal{O},$ which is lower order in terms of angular derivatives by Lemma~\ref{R k lemma}. We also use the lower order estimates in Section~\ref{lower order estimates section} in order to deal with the error terms arising from commutation.

Finally, we combine the low frequency and high frequency regime estimates for the singular component in order to obtain \eqref{main estimate PHI Y top order} and \eqref{practical estimate PHI Y top order}, similarly to the approach used in the next section.

\subsubsection{The proof of Theorem~\ref{main theorem first system}}\label{first system combined estimates section}
In this section, we combine the low frequency regime and the high frequency regime estimates for the regular components to establish top order estimates. Together with \eqref{main estimate PHI Y top order}, \eqref{practical estimate PHI Y top order}, and the lower order estimates, these complete the proof of Theorem~\ref{main theorem first system}.

\begin{proof}[Proof of Theorem~\ref{main theorem first system}.]
    Using \eqref{main estimate PHI Y top order}, \eqref{practical estimate PHI Y top order}, and the lower order estimates in Propositions~\ref{practical estimate for Phi Y proposition} and \ref{standard estimates regular components propositionn}, we notice that in order to establish Theorem~\ref{main theorem first system} it suffices to prove the estimate for all $\tau\in(0,1]:$
    \[\tau\big\|\nabla_{\tau}\big(\nabla^M\Phi_{0}\big)_J\big\|^2_{H^{1/2}}+ \tau\big\|\nabla(\nabla^M\Phi_{0}\big)_J\big\|^2_{H^{1/2}}+\sum_{i=1}^I\tau\big\|\nabla_{\tau}\nabla^M\Phi_i\big\|^2_{H^{1/2}}+\sum_{i=1}^I\tau\big\|\nabla^{M+1}\Phi_i\big\|^2_{H^{1/2}}+\big\|\big(\nabla^M\Phi_{0}\big)_J\big\|^2_{H^1}+\]\[+\sum_{i=1}^I\big\|\nabla^M\Phi_i\big\|^2_{H^1}\lesssim\big\|\mathcal{O}\big\|_{H^{M+1}}^2+\big\|\mathfrak{h}\big\|_{H^{M+1}}^2+\sum_{i=1}^I\big\|\Phi_i^0\big\|^2_{H^{M+1}}+\sum_{i=0}^I\int_0^{\tau}\big\|F_M^i\big\|_{L^2}^2+\tau'\big\|F_M^i\big\|_{H^{1/2}}^2d\tau'.\]
    
    For the rest of the proof we show this estimate. The first step is to combine the high frequency regime estimate in Proposition~\ref{high frequency forward estimate} with the low frequency regime estimate in Proposition~\ref{low frequency regular proposition}. We denote $\xi_{0}=\big(\nabla^M\Phi_{0}\big)_J$ and $\xi_i=\nabla^M\Phi_i$ for $1\leq i\leq I$. Using the high frequency regime estimate in Proposition~\ref{high frequency forward estimate}, we get that for $k\geq0$ and $\tau\in\big[2^{-k-1},1\big]$:
    \[\sum_{i=0}^I2^k\tau\big\|P_k\nabla_{\tau}\xi_i\big\|^2_{L^2}+\sum_{i=0}^I\frac{2^k}{\tau}\big\|P_k\xi_i\big\|^2_{L^2}+\sum_{i=0}^I2^k\tau\big\|\nabla P_k\xi_i\big\|^2_{L^2}\lesssim\]
    \[\lesssim\sum_{i=0}^I\bigg(\big\|P_k\nabla_{\tau}\xi_i\big\|^2_{L^2}+2^{2k}\big\|P_k\xi_i\big\|^2_{L^2}+\big\|\nabla P_k\xi_i\big\|^2_{L^2}\bigg)\bigg|_{\tau=2^{-k-1}}+\sum_{i=0}^I\int_{2^{-k-1}}^{\tau}2^k\tau'\big\|\underline{\widetilde{P}}_k\xi_i\big\|_{L^2}^2\]\[+\sum_{i=0}^I\int_{2^{-k-1}}^{\tau}2^k(\tau')^3\big\|\underline{\widetilde{P}}_k\nabla\xi_i\big\|_{L^2}^2+\sum_{i=0}^I\int_{2^{-k-1}}^{\tau}2^k(\tau')^3\big\|\underline{\widetilde{P}}_k\nabla_{\tau}\xi_i\big\|_{L^2}^2+\sum_{i=0}^I\int_{2^{-k-1}}^{\tau}\frac{\tau'}{2^{k}}\big\|\xi_i\big\|_{H^1}^2\]\[+\sum_{i=0}^I\int_{2^{-k-1}}^{\tau}\frac{(\tau')^3}{2^{k}}\big\|\nabla_{\tau}\xi_i\big\|_{L^2}^2+\sum_{i=0}^I\int_{2^{-k-1}}^{\tau}2^k\tau'\big\|P_kF_M^i\big\|_{L^2}^2+\int_{2^{-k-1}}^{\tau}2^k\tau'\big\|P_k\big(\psi\nabla\big(\nabla^M\Phi_{0}\big)_Y\big)\big\|_{L^2}^2.\]
    We use the low frequency regime estimate in Proposition~\ref{low frequency regular proposition} and the bound \eqref{practical estimate PHI Y top order} for the singular component. Thus, we get for $k\geq0$ and $\tau\in\big[2^{-k-1},1\big]$ the following high frequency regime estimate:
    \[\sum_{i=0}^I2^k\tau\big\|P_k\nabla_{\tau}\xi_i\big\|^2_{L^2}+\sum_{i=0}^I\frac{2^k}{\tau}\big\|P_k\xi_i\big\|^2_{L^2}+\sum_{i=0}^I2^k\tau\big\|\nabla P_k\xi_i\big\|^2_{L^2}\lesssim\]
    \[\lesssim\sum_{i=0}^I\big\|\nabla P_k\xi_i^0\big\|^2_{L^2}+\sum_{i=0}^I2^{2k}\big\|P_k\xi_i^0\big\|^2_{L^2}+2^{-k}\big\|\mathcal{O}\big\|^2_{H^{M+1}}+\sum_{i=0}^I\int_0^{2^{-k-1}}\frac{1}{2^{3k}}\big\|\xi_i\big\|_{H^1}^2+\sum_{i=0}^I\int_0^{2^{-k-1}}\frac{(\tau')^2}{2^{k}}\big\|\nabla_{\tau}\xi_i\big\|_{L^2}^2\]\[+\sum_{i=0}^I\int_0^{2^{-k-1}}\frac{1}{2^{k}}\big\|P_kF_M^i\big\|_{L^2}^2+\sum_{i=0}^I\int_{2^{-k-1}}^{\tau}\bigg(2^k\tau'\big\|\underline{\widetilde{P}}_k\xi_i\big\|_{L^2}^2+2^k(\tau')^3\big\|\underline{\widetilde{P}}_k\nabla\xi_i\big\|_{L^2}^2+2^k(\tau')^3\big\|\underline{\widetilde{P}}_k\nabla_{\tau}\xi_i\big\|_{L^2}^2\bigg)d\tau'\]\[+\sum_{i=0}^I\int_{2^{-k-1}}^{\tau}\frac{\tau'}{2^{k}}\big\|\xi_i\big\|_{H^1}^2+\sum_{i=0}^I\int_{2^{-k-1}}^{\tau}\frac{(\tau')^3}{2^{k}}\big\|\nabla_{\tau}\xi_i\big\|_{L^2}^2+\sum_{i=0}^I\int_{2^{-k-1}}^{\tau}2^k\tau'\big\|P_kF_M^i\big\|_{L^2}^2+\int_{2^{-k-1}}^{\tau}2^k\tau'\big\|\nabla P_k\big(\nabla^M\Phi_{0}\big)_Y\big\|_{L^2}^2.\]
    
    The second step is to prove a bound for the sum of the non-negative frequencies. We define the following energy for all $k\geq0$:
    \[2^kE_k^2(\tau)=\sum_{i=0}^I2^k\tau\big\|P_k\nabla_{\tau}\xi_i\big\|^2_{L^2}+\sum_{i=0}^I2^k\big\|P_k\xi_i\big\|^2_{L^2}+\sum_{i=0}^I2^k\tau\big\|\nabla P_k\xi_i\big\|^2_{L^2}+\sum_{i=0}^I\big\|\nabla P_k\xi_i\big\|^2_{L^2}.\]
    Using the singular component high frequency regime estimate \eqref{high freq estimate improved singular component}, we get that for all $k\geq0$:
    \[\sum_{\tau>2^{-k-1}}\int_{2^{-k-1}}^{\tau}2^k\tau'\big\|\nabla P_k\big(\nabla^M\Phi_{0}\big)_Y\big\|_{L^2}^2d\tau'\lesssim\int_{0}^{\tau}\sum_{\tau'>2^{-k-1}}2^k\tau'\big\|\nabla P_k\big(\nabla^M\Phi_{0}\big)_Y\big\|_{L^2}^2d\tau'\lesssim\big\|\mathcal{O}\big\|_{H^{M+1}}^2.\]
    As a result, the above high frequency regime estimate implies the bound:
    \[\sum_{\tau>2^{-k-1}}2^kE_k^2\lesssim\big\|\mathcal{O}\big\|_{H^{M+1}}^2+\sum_{i=0}^I\big\|\xi_i^0\big\|^2_{H^{1}}+\sum_{i=0}^I\int_{0}^{\tau}\big\|\xi_i\big\|_{H^1}^2+\sum_{i=0}^I\int_0^{\tau}(\tau')^2\big\|\nabla_{\tau}\xi_i\big\|_{H^{1/2}}^2+\]\[+\sum_{i=0}^I\int_{0}^{\tau}(\tau')^2\big\|\nabla\xi_i\big\|_{H^{1/2}}^2+\int_{0}^{\tau}\tau'\big\|\xi_i\big\|_{H^{1/2}}^2+\sum_{i=0}^I\sum_{k\geq0}\int_{0}^{\tau}\big(2^k\tau'+2^{-k}\big)\big\|P_kF_M^i\big\|_{L^2}^2.\]
    The low frequency estimates in Proposition~\ref{standard estimates regular components propositionn} and \eqref{practical estimate PHI Y top order} imply the bound:
    \[\sum_{\tau\leq2^{-k-1}}2^kE_k^2\lesssim\big\|\mathcal{O}\big\|_{H^{M+1}}^2+\sum_{i=0}^I\big\|\xi_i^0\big\|^2_{H^{1}}+\sum_{i=0}^I\int_0^{\tau}\big\|\xi_i\big\|_{H^1}^2+\sum_{i=0}^I\int_0^{\tau}(\tau')^2\big\|\nabla_{\tau}\xi_i\big\|_{L^2}^2+\sum_{i=0}^I\sum_{k\geq0}\int_0^{\tau}\frac{1}{2^{k}}\big\|P_kF_M^i\big\|_{L^2}^2.\]
    As a result, we completed the second step of the proof and showed that:
    \begin{align*}
        \sum_{k\geq0}2^kE_k^2&\lesssim\big\|\mathcal{O}\big\|_{H^{M+1}}^2+\sum_{i=0}^I\big\|\xi_i^0\big\|^2_{H^{1}}+\sum_{i=0}^I\int_{0}^{\tau}\big\|\xi_i\big\|_{H^1}^2+\sum_{i=0}^I\int_0^{\tau}(\tau')^2\big\|\nabla_{\tau}\xi_i\big\|_{H^{1/2}}^2+\\
        &+\sum_{i=0}^I\int_{0}^{\tau}(\tau')^2\big\|\nabla\xi_i\big\|_{H^{1/2}}^2+\sum_{i=0}^I\int_{0}^{\tau}\big\|F_M^i\big\|_{L^2}^2+\tau'\big\|F_M^i\big\|_{H^{1/2}}^2d\tau'.
    \end{align*}

    The final step of the proof is dealing with the negative frequencies $k<0.$ We remark that we can repeat the proof of Proposition~\ref{standard estimates regular components propositionn} for $m=M$ and use \eqref{practical estimate PHI Y top order} in order to deal with the singular component. Thus, to get:
    \[\sum_{i=0}^I\big\|\nabla_{\tau}\xi_i\big\|^2_{L^2}+\sum_{i=0}^I\big\|\xi_i\big\|^2_{H^1}\lesssim\big\|\mathcal{O}\big\|^2_{H^{M+1}}+\sum_{i=0}^I\big\| \xi_i^0\big\|^2_{H^{1}}+\sum_{i=0}^I\int_0^{\tau}\big\|F_M^i\big\|_{L^2}^2d\tau'.\]
    We notice that by \cite{geometricLP}, we have that $\|P_k\nabla F\|_{L^2}\lesssim2^k\|F\|_{L^2}$ for any $k<0$. We obtain the following bound for the negative frequencies:
    \[\sum_{i=0}^I\sum_{k<0}\Big(2^k\tau\big\|P_k\nabla_{\tau}\xi_i\big\|^2_{L^2}+2^k\big\|P_k\xi_i\big\|^2_{L^2}+2^k\tau\big\|\nabla P_k\xi_i\big\|^2_{L^2}+\big\|\nabla P_k\xi_i\big\|^2_{L^2}\Big)\lesssim\sum_{i=0}^I\big\|\xi_i\big\|^2_{L^2}+\big\|\nabla_{\tau}\xi_i\big\|^2_{L^2}\]\[\lesssim\big\|\mathcal{O}\big\|^2_{H^{M+1}}+\sum_{i=0}^I\big\| \xi_i^0\big\|^2_{H^{1}}+\sum_{i=0}^I\int_0^{\tau}\big\|F_M^i\big\|_{L^2}^2d\tau'.\]

    To conclude the proof of Theorem~\ref{main theorem first system}, we combine the estimates proved for non-negative frequencies and negative frequencies. We then apply Gronwall to obtain:
    \[\sum_{i=0}^I\tau\big\|\nabla_{\tau}\xi_i\big\|^2_{H^{1/2}}+\sum_{i=0}^I\big\|\xi_i\big\|^2_{H^1}+\sum_{i=0}^I\tau\big\|\nabla\xi_i\big\|^2_{H^{1/2}}\lesssim\big\|\mathcal{O}\big\|_{H^{M+1}}^2+\sum_{i=0}^I\big\|\xi_i^0\big\|^2_{H^{1}}+\sum_{i=0}^I\int_0^{\tau}\big\|F_M^i\big\|_{L^2}^2+\tau'\big\|F_M^i\big\|_{H^{1/2}}^2.\]
\end{proof}

\section{Estimates for the Second Model System}\label{model system back direction section}

In this section we prove Theorem~\ref{main theorem second system}, obtaining estimates for solutions of the second model system at $\tau\in(0,1)$ in terms of the solution at $\tau=1.$ We also prove estimates for the asymptotic data at $\mathcal{I}^-.$ We follow the steps outlined in the Section~\ref{second model system intro section} of the Introduction, and we advise the reader to refer to this section for assistance while reading the proof below.

To complete the statement of Theorem~\ref{main theorem second system}, we first define in detail the energy $\mathcal{E}_{II}$ of the solution on $S_{\tau}:$
\begin{align}\label{definition of E II}
    \mathcal{E}_{II}(\tau)=&\tau\big\|\Phi_0\big\|_{H^{M+1/2}}^2+\tau^2\big\|\Phi_0\big\|_{H^{M+3/2}}^2+\tau^2\big\|\nabla_{\tau}\nabla^M\Phi_0\big\|_{H^{1/2}}^2+\sum_{m=0}^{M-1}\tau^2\big\|\nabla_{\tau}\nabla^m\Phi_0\big\|_{L^2}^2+\int_{\tau}^1\tau'\big\|\Phi_0\big\|_{H^{M+1}}^2d\tau'\notag\\
    &+\sum_{i=1}^I\big\|\Phi_i\big\|_{H^{M+3/2}}^2+\sum_{i=1}^I\sum_{m=0}^{M}\big\|\nabla_{\tau}\nabla^m\Phi_i\big\|_{H^{1/2}}^2+\sum_{i=1}^I\sum_{m=0}^{M}\int_{\tau}^1\frac{1}{\tau'}\big\|\nabla_{\tau}\nabla^m\Phi_i\big\|_{H^{1/2}}^2d\tau'.
\end{align}

In Section~\ref{good quantities estimates subsection}, we prove estimates for the regular quantities $\Phi_1,\ldots,\Phi_I$ which satisfy the equations \eqref{equation for Phi i} with $\sigma=2$ and decouple from the singular quantity $\Phi_0.$ We prove \eqref{good quantities estimate intro} and \eqref{good quantities asymptotic data intro} in Proposition~\ref{good quantities estimates proposition}. In Section~\ref{back singular preliminary section}, we prove the preliminary bound \eqref{back singular preliminary intro} in Proposition~\ref{backward direction basic estimate proposition}. In Section~\ref{low frequency backward section singular}, we prove the low frequency regime estimates in Propositions~\ref{preliminary low low frequency estimate proposition} and \ref{prelim low freq singular proposition}. In Section~\ref{prelim high frequency estimate singular section}, we prove the preliminary high frequency regime estimate in Proposition~\ref{high freq back estimate}. We improve this in Section~\ref{improved high frequency estimates section} to obtain the high frequency regime estimate in Proposition~\ref{improved high frequency estimate proposition}. Combining the previous estimates in Section~\ref{main estimate back direction proof section}, we complete the proof of the main estimate \eqref{main estimate second model system} in Theorem~\ref{main theorem second system}. Finally, we establish the estimates \eqref{main estimate asymptotic data} and \eqref{main estimate asymptotic data h} for the asymptotic data at $\mathcal{I}^-$ in Section~\ref{asympt data estimates section}. We prove \eqref{estimate for O intro} in Proposition~\ref{estimate for O proposition} and \eqref{estimate for h intro} in Proposition~\ref{estimate for PkhM proposition}.

\textbf{Notation.} Unless otherwise noted, in this section we write $A\lesssim B$ for some quantities $A,B>0$ if there exists a constant $C>0$ depending only on the constants $M,C_0,C_2$ defined in the Introduction, such that $A\leq CB.$ 

We point out that the main difficulty in our argument is dealing with the top order quantity $\xi=\nabla^M\Phi_0.$ For this part we use as a guideline the toy problem considered in \cite[Section~9]{Cmain}, where we studied the equation satisfied by $\xi=\nabla^M\Phi_0,$ but we dropped the terms $F'_M=F_M^0+\sum_{i=1}^I\psi\nabla^{M+1}\Phi_i$ for simplicity.

\subsection{Estimates for the regular quantities}\label{good quantities estimates subsection}
In this section, we prove the main estimates for all the regular quantities $\Phi_i$, with $1\leq i\leq I:$
\begin{proposition}\label{good quantities estimates proposition}
For all $0\leq m\leq M$ and $1\leq i\leq I$ we have:
    \[\big\|\nabla_{\tau}\nabla^m\Phi_i\big\|_{H^{1/2}}^2+\big\|\nabla^{m}\Phi_i\big\|_{H^{3/2}}^2+\int_{\tau}^1\frac{1}{\tau'}\big\|\nabla_{\tau}\nabla^m\Phi_i\big\|_{H^{1/2}}^2d\tau'\lesssim\]\[\lesssim \sum_{j=1}^I\bigg(\big\|\Phi_j\big\|_{H^{m+3/2}}^2\big|_{\tau=1}+\big\|\nabla_{\tau}\Phi_j\big\|_{H^{m+1/2}}^2\big|_{\tau=1}+\sum_{k=0}^m\int_{\tau}^1\tau'\big\|F_k^j\big\|_{H^{1/2}}^2d\tau'\bigg)\lesssim\mathcal{D}_{II}+\mathcal{F}_{II}(\tau).\]
\end{proposition}
\begin{proof} For the purpose of this proof we fix $0\leq m\leq M$ and we denote $\xi_i=\nabla^m\Phi_i.$ We can rewrite \eqref{equation for Phi i} with $\sigma=2$ for all $1\leq i\leq I:$
    \begin{equation}\label{eq for xi i}
        \nabla_{\tau}\big(\nabla_{\tau}\xi_i\big)-\frac{1}{\tau}\nabla_{\tau}\xi_i-4\Delta\xi_i=\sum_{j=1}^I\psi\nabla\xi_j+F_m^i.
    \end{equation}
    
    \paragraph{Preliminary estimates.} We first prove the following preliminary estimates for all $1\leq i\leq I:$
    \[\big\|\nabla_{\tau}\nabla^m\Phi_i\big\|_{L^2}^2+\big\|\nabla^m\Phi_i\big\|_{H^1}^2+\int_{\tau}^1\frac{1}{\tau'}\big\|\nabla_{\tau}\nabla^m\Phi_i\big\|_{L^2}^2d\tau'\lesssim\sum_{j=1}^I\bigg(\big\|\Phi_j\big\|_{H^{m+1}}^2\big|_{\tau=1}+\big\|\nabla_{\tau}\Phi_j\big\|_{H^{m}}^2\big|_{\tau=1}+\int_{\tau}^1\tau'\big\|F_m^j\big\|_{L^2}^2\bigg).\]
    To prove this, we contract each equation \eqref{eq for xi i} by $\nabla_{\tau}\xi_i.$ Using Lemma \ref{volume form lemma}, we obtain the energy estimate:
    \[\big\|\nabla_{\tau}\xi_i\big\|_{L^2}^2+\big\|\nabla\xi_i\big\|_{L^2}^2+\int_{\tau}^1\frac{1}{\tau'}\big\|\nabla_{\tau}\xi_i\big\|_{L^2}^2d\tau'\lesssim \Big(\big\|\xi_i\big\|_{H^1}^2+\big\|\nabla_{\tau}\xi_i\big\|_{L^2}^2\bigg)\Big|_{\tau=1}+\]\[+\int_{\tau}^1\tau'\big\|\nabla\xi_i\big\|_{L^2}\big\|[\nabla,\nabla_4]\xi_i\big\|_{L^2}d\tau'+\sum_{j=1}^I\int_{\tau}^1\big\|\nabla\xi_j\big\|_{L^2}\big\|\nabla_{\tau}\xi_i\big\|_{L^2}+\int_{\tau}^1\big\|\nabla_{\tau}\xi_i\big\|_{L^2}\big\|F_m^i\big\|_{L^2}d\tau'.\]
    We use Cauchy-Schwarz, Gronwall, and the bulk term to obtain:
    \[\big\|\nabla_{\tau}\xi_i\big\|_{L^2}^2+\big\|\nabla\xi_i\big\|_{L^2}^2+\int_{\tau}^1\frac{1}{\tau'}\big\|\nabla_{\tau}\xi_i\big\|_{L^2}^2d\tau'\lesssim \Big(\big\|\xi_i\big\|_{H^1}^2+\big\|\nabla_{\tau}\xi_i\big\|_{L^2}^2\bigg)\Big|_{\tau=1}+\]\[+\int_{\tau}^1(\tau')^2\big\|\xi_i\big\|_{L^2}^2d\tau'+\sum_{j=1}^I\int_{\tau}^1\big\|\nabla\xi_j\big\|_{L^2}^2d\tau'+\int_{\tau}^1\tau'\big\|F_m^i\big\|_{L^2}^2d\tau'.\]
    On the other hand, we also have the bound:
    \[\big\|\xi_i\big\|_{L^2}^2\lesssim\big\|\xi_i\big\|_{L^2}^2\Big|_{\tau=1}+\int_{\tau}^1\big\|\xi_i\big\|_{L^2}\big\|\nabla_{\tau}\xi_i\big\|_{L^2}d\tau'\lesssim\big\|\xi_i\big\|_{L^2}^2\Big|_{\tau=1}+\int_{\tau}^1\big\|\xi_i\big\|_{L^2}^2d\tau'+\int_{\tau}^1\big\|\nabla_{\tau}\xi_i\big\|_{L^2}^2d\tau'.\]
    We sum the last two inequalities for all $1\leq i\leq I:$
    \[\sum_{i=1}^I\big\|\nabla_{\tau}\xi_i\big\|_{L^2}^2+\sum_{i=1}^I\big\|\xi_i\big\|_{H^1}^2+\sum_{i=1}^I\int_{\tau}^1\frac{1}{\tau'}\big\|\nabla_{\tau}\xi_i\big\|_{L^2}^2d\tau'\lesssim \sum_{i=1}^I\Big(\big\|\xi_i\big\|_{H^1}^2+\big\|\nabla_{\tau}\xi_i\big\|_{L^2}^2\bigg)\Big|_{\tau=1}+\]\[+\sum_{i=1}^I\int_{\tau}^1\big\|\xi_i\big\|_{H^1}^2d\tau'+\sum_{i=1}^I\int_{\tau}^1\big\|\nabla_{\tau}\xi_i\big\|_{L^2}^2d\tau'+\sum_{i=1}^I\int_{\tau}^1\tau'\big\|F_m^i\big\|_{L^2}^2d\tau'.\]
    We use Gronwall to complete the proof of the above preliminary estimate. For the data terms at $\tau=1$ we also use the commutation formula \eqref{commutation formula 1} and the assumption on the background spacetime $\|\psi\|_{H^{M+1}(S_1)}\leq C_2$.

    \paragraph{The main estimates.} For each $k\geq0$, we apply $P_k$ to equation \eqref{eq for xi i} to obtain:
     \[\nabla_{\tau}\big(P_k\nabla_{\tau}\xi_i\big)-\frac{1}{\tau}P_k\nabla_{\tau}\xi_i-4\Delta P_k\xi_i=\sum_{j=1}^IP_k\big(\psi\nabla\xi_j\big)+P_kF_m^i+[\nabla_{\tau},P_k]\nabla_{\tau}\xi_i.\]
    We contract each equation by $P_k\nabla_{\tau}\xi_i,$ in order to obtain the energy estimate:
    \[\big\|P_k\nabla_{\tau}\xi_i\big\|_{L^2}^2+\big\|\nabla P_k\xi_i\big\|_{L^2}^2+\int_{\tau}^1\frac{1}{\tau'}\big\|P_k\nabla_{\tau}\xi_i\big\|_{L^2}^2d\tau'\lesssim \Big(\big\|\nabla P_k\xi_i\big\|_{L^2}^2+\big\|P_k\nabla_{\tau}\xi_i\big\|_{L^2}^2\bigg)\Big|_{\tau=1}+\]
    \[+\int_{\tau}^1\tau'\big\|\nabla P_k\xi_i\big\|_{L^2}\big\|\nabla[P_k,\nabla_4]\xi_i\big\|_{L^2}d\tau'+\int_{\tau}^1\tau'\big\|\nabla P_k\xi_i\big\|_{L^2}\big\|[\nabla,\nabla_4]P_k\xi_i\big\|_{L^2}d\tau'\]
    \[+\sum_{j=1}^I\int_{\tau}^1\big\|P_k(\psi\nabla\xi_j)\big\|_{L^2}\big\|P_k\nabla_{\tau}\xi_i\big\|_{L^2}+\int_{\tau}^1\big\|P_k\nabla_{\tau}\xi_i\big\|_{L^2}\big\|P_kF_m^i\big\|_{L^2}d\tau'+\int_{\tau}^1\big\|P_k\nabla_{\tau}\xi_i\big\|_{L^2}\big\|[\nabla_{\tau},P_k]\nabla_{\tau}\xi_i\big\|_{L^2}d\tau'.\]
    We use Cauchy-Schwarz and Gronwall to obtain:
    \[\big\|P_k\nabla_{\tau}\xi_i\big\|_{L^2}^2+\big\|\nabla P_k\xi_i\big\|_{L^2}^2+\int_{\tau}^1\frac{1}{\tau'}\big\|P_k\nabla_{\tau}\xi_i\big\|_{L^2}^2d\tau'\lesssim \Big(\big\|\nabla P_k\xi_i\big\|_{L^2}^2+\big\|P_k\nabla_{\tau}\xi_i\big\|_{L^2}^2\bigg)\Big|_{\tau=1}+\]
    \[+\int_{\tau}^1(\tau')^2\big\|\nabla[P_k,\nabla_4]\xi_i\big\|_{L^2}^2d\tau'+\int_{\tau}^1(\tau')^2\big\|P_k\xi_i\big\|_{L^2}^2d\tau'+\sum_{j=1}^I\int_{\tau}^1\big\|[P_k,\psi]\nabla\xi_j\big\|_{L^2}^2d\tau'+\]
    \[+\sum_{j=1}^I\int_{\tau}^1\big\|P_k\nabla\xi_j\big\|_{L^2}^2d\tau'+\int_{\tau}^1\tau'\big\|P_kF_m^i\big\|_{L^2}^2d\tau'+\int_{\tau}^1(\tau')^3\big\|[\nabla_{4},P_k]\nabla_{\tau}\xi_i\big\|_{L^2}^2d\tau'.\]
    We use Lemma \ref{Litt Paley lemma} and Lemma \ref{Litt Paley lemma refined} to get:
    \[\big\|P_k\nabla_{\tau}\xi_i\big\|_{L^2}^2+\big\|\nabla P_k\xi_i\big\|_{L^2}^2+\int_{\tau}^1\frac{1}{\tau'}\big\|P_k\nabla_{\tau}\xi_i\big\|_{L^2}^2d\tau'\lesssim \Big(\big\|\nabla P_k\xi_i\big\|_{L^2}^2+\big\|P_k\nabla_{\tau}\xi_i\big\|_{L^2}^2\bigg)\Big|_{\tau=1}+\]
    \[+\int_{\tau}^1(\tau')^2\big\|\underline{\widetilde{P}}_k\nabla\xi_i\big\|_{L^2}^2d\tau'+\int_{\tau}^1(\tau')^2\big\|P_k\xi_i\big\|_{L^2}^2d\tau'+\sum_{j=1}^I\int_{\tau}^12^{-2k}\big\|\xi_j\big\|_{H^1}^2d\tau'+\sum_{j=1}^I\int_{\tau}^1\big\|P_k\nabla\xi_j\big\|_{L^2}^2d\tau'+\]\[+\int_{\tau}^1\tau'\big\|P_kF_m^i\big\|_{L^2}^2d\tau'+\int_{\tau}^1(\tau')^3\big\|\underline{\widetilde{P}}_k\nabla_{\tau}\xi_i\big\|_{L^2}^2d\tau'+\int_{\tau}^12^{-2k}(\tau')^3\big\|\nabla_{\tau}\xi_i\big\|_{L^2}^2d\tau'.\]
    We multiply each inequality by $2^k$ and sum for all $1\leq i\leq I:$
    \[\sum_{i=1}^I2^k\big\|P_k\nabla_{\tau}\xi_i\big\|_{L^2}^2+\sum_{i=1}^I2^k\big\|\nabla P_k\xi_i\big\|_{L^2}^2+\sum_{i=1}^I\int_{\tau}^1\frac{2^k}{\tau'}\big\|P_k\nabla_{\tau}\xi_i\big\|_{L^2}^2d\tau'\lesssim \sum_{i=1}^I2^k\Big(\big\|\nabla P_k\xi_i\big\|_{L^2}^2+\big\|P_k\nabla_{\tau}\xi_i\big\|_{L^2}^2\bigg)\Big|_{\tau=1}\]
    \[+\sum_{i=1}^I2^k\int_{\tau}^1(\tau')^2\big\|\underline{\widetilde{P}}_k\nabla\xi_i\big\|_{L^2}^2d\tau'+\sum_{i=1}^I2^k\int_{\tau}^1(\tau')^2\big\|P_k\xi_i\big\|_{L^2}^2d\tau'+\sum_{i=1}^I\int_{\tau}^12^{-k}\big\|\xi_i\big\|_{H^1}^2d\tau'+\]\[+\sum_{i=1}^I2^k\int_{\tau}^1\tau'\big\|P_kF_m^i\big\|_{L^2}^2d\tau'+\sum_{i=1}^I2^k\int_{\tau}^1(\tau')^3\big\|\underline{\widetilde{P}}_k\nabla_{\tau}\xi_i\big\|_{L^2}^2d\tau'+\sum_{i=1}^I\int_{\tau}^12^{-k}(\tau')^3\big\|\nabla_{\tau}\xi_i\big\|_{L^2}^2d\tau'.\]
    We then sum for all $k\geq0,$ using the preliminary estimate as well to obtain:
    \[\sum_{i=1}^I\big\|\nabla_{\tau}\xi_i\big\|_{H^{1/2}}^2+\sum_{i=1}^I\big\|\xi_i\big\|_{H^{3/2}}^2+\sum_{i=1}^I\int_{\tau}^1\frac{1}{\tau'}\big\|\nabla_{\tau}\xi_i\big\|_{H^{1/2}}^2d\tau'\lesssim \sum_{i=1}^I\Big(\big\|\xi_i\big\|_{H^{3/2}}^2+\big\|\nabla_{\tau}\xi_i\big\|_{H^{1/2}}^2\bigg)\Big|_{\tau=1}+\]
    \[+\sum_{i=1}^I\int_{\tau}^1\big\|\xi_i\big\|_{H^{3/2}}^2d\tau'+\sum_{i=1}^I\int_{\tau}^1\tau'\big\|F_m^i\big\|_{H^{1/2}}^2d\tau'+\sum_{i=1}^I\int_{\tau}^1(\tau')^3\big\|\nabla_{\tau}\xi_i\big\|_{H^{1/2}}^2d\tau'.\]
    We use Gronwall and we bound the initial data term at $\tau=1$ as before (using $C_2$) to obtain the conclusion.
\end{proof}
\subsection{Preliminary estimates for the singular quantities}\label{back singular preliminary section}
For every $0\leq m\leq M$, we write the equation \eqref{equation for Phi 0} for $\nabla^m\Phi_0$ as:
\begin{equation}\label{backward direction linear wave equation}
    \nabla_{\tau}\big(\nabla_{\tau}\nabla^m\Phi_0\big)+\frac{1}{\tau}\nabla_{\tau}\nabla^m\Phi_0-4\Delta\nabla^m\Phi_0=\psi\nabla^{m+1}\Phi_0+F_{m}',
\end{equation}
where we denote $F_m'=\psi\sum_{i=1}^I\nabla^{m+1}\Phi_i+F_{m}^{0}.$ We prove the preliminary estimate \eqref{back singular preliminary intro} as outlined in Section~\ref{second model system intro section} of the Introduction:
\begin{proposition}\label{backward direction basic estimate proposition}
    For all $0\leq m\leq M$ and $\tau\in[0,1]$ we have the estimates:
    \begin{align*}
        \big\|\tau\nabla_{\tau}\nabla^m\Phi_0\big\|_{L^2}^2+\big\|\tau\nabla^{m+1}\Phi_0\big\|_{L^2}^2+\tau\big\|\nabla^m\Phi_0\big\|_{L^2}^2+\int_{\tau}^1\tau'\big\|\nabla^{m+1}\Phi_0\big\|_{L^2}^2d\tau'\lesssim \\ \lesssim \bigg(\big\|\Phi_0\big\|_{H^{m+3/2}}^2+\big\|\nabla_{\tau}\Phi_0\big\|_{H^{m+1/2}}^2\bigg)\bigg|_{\tau=1}+\int_{\tau}^1(\tau')^2\big\|F_m'\big\|_{L^2}^2d\tau'.
    \end{align*}
\end{proposition}
\begin{proof}
    For the purpose of this proof, we denote $\xi_m=\nabla^m\Phi_0.$ We can rewrite equation (\ref{backward direction linear wave equation}) as:
    \[\nabla_{\tau}\big(\tau\nabla_{\tau}\xi_m\big)-4\tau\Delta\xi_m=\tau\psi\nabla\xi_m+\tau F_m'.\]
    Using Lemma \ref{volume form lemma}, we obtain the energy estimate:
    \[\big\|\tau\nabla_{\tau}\xi_m\big\|_{L^2}^2+\big\|\tau\nabla\xi_m\big\|_{L^2}^2+\int_{\tau}^1\tau'\big\|\nabla\xi_m\big\|_{L^2}^2d\tau'\lesssim \Big(\big\|\xi_m\big\|_{H^1}^2+\big\|\nabla_{\tau}\xi_m\big\|_{L^2}^2\bigg)\Big|_{\tau=1}\]\[+\int_{\tau}^1\tau'\big\|\nabla\xi_m\big\|_{L^2}\Big(\tau'\big\|\nabla_{\tau}\xi_m\big\|_{L^2}+(\tau')^2\big\|[\nabla,\nabla_4]\xi_m\big\|_{L^2}\Big)d\tau'+\int_{\tau}^1(\tau')^2\big\|F_m'\big\|_{L^2}^2d\tau'.\]
    Using Gronwall, we obtain:
    \[\big\|\tau\nabla_{\tau}\xi_m\big\|_{L^2}^2+\big\|\tau\nabla\xi_m\big\|_{L^2}^2+\int_{\tau}^1\tau'\big\|\nabla\xi_m\big\|_{L^2}^2d\tau'\lesssim\Big(\big\|\xi_m\big\|_{H^1}^2+\big\|\nabla_{\tau}\xi_m\big\|_{L^2}^2\bigg)\Big|_{\tau=1}+\int_{\tau}^1(\tau')^4\big\|\xi_m\big\|_{L^2}^2+\int_{\tau}^1(\tau')^2\big\|F_m'\big\|_{L^2}^2.\]
    Similarly, we have the estimate:
    \[\big\|\xi_m\big\|_{L^2}^2\lesssim\|\xi_m\big\|_{L^2}^2\Big|_{\tau=1}+\int_{\tau}^1\big\|\xi_m\big\|_{L^2}\big\|\nabla_{\tau}\xi_m\big\|_{L^2}d\tau'.\]
    In particular, this implies:
    \[\tau^4\big\|\xi_m\big\|_{L^2}^2\lesssim \|\xi_m\big\|_{L^2}^2\Big|_{\tau=1}+\int_{\tau}^1(\tau')^4\big\|\nabla_{\tau}\xi_m\big\|_{L^2}^2d\tau',\]
    which gives in the above inequality:
    \[\big\|\tau\nabla_{\tau}\xi_m\big\|_{L^2}^2+\big\|\tau\nabla\xi_m\big\|_{L^2}^2+\int_{\tau}^1\tau'\big\|\nabla\xi_m\big\|_{L^2}^2d\tau'\lesssim \Big(\big\|\xi_m\big\|_{H^1}^2+\big\|\nabla_{\tau}\xi_m\big\|_{L^2}^2\bigg)\Big|_{\tau=1}+\]\[+\int_{\tau}^1(\tau')^4\big\|\nabla_{\tau}\xi_m\big\|_{L^2}^2d\tau'+\int_{\tau}^1(\tau')^2\big\|F_m'\big\|_{L^2}^2d\tau'.\]
    By Gronwall, we obtain:
    \[\big\|\tau\nabla_{\tau}\xi_m\big\|_{L^2}^2+\big\|\tau\nabla\xi_m\big\|_{L^2}^2+\int_{\tau}^1\tau'\big\|\nabla\xi_m\big\|_{L^2}^2d\tau'\lesssim \Big(\big\|\xi_m\big\|_{H^1}^2+\big\|\nabla_{\tau}\xi_m\big\|_{L^2}^2\bigg)\Big|_{\tau=1}+\int_{\tau}^1(\tau')^2\big\|F_m'\big\|_{L^2}^2d\tau'.\]
    Using this, we also have:
    \[\tau\big\|\xi_m\big\|_{L^2}^2\lesssim \|\xi_m\big\|_{L^2}^2\Big|_{\tau=1}+\int_{\tau}^1\tau'\big\|\xi_m\big\|_{L^2}\cdot\big\|\nabla_{\tau}\xi_m\big\|_{L^2}\lesssim \|\xi_m\big\|_{L^2}^2\Big|_{\tau=1}+ \int_{\tau}^1\sqrt{\tau'}\big\|\xi_m\big\|_{L^2}^2+\int_{\tau}^1(\tau')^{3/2}\big\|\nabla_{\tau}\xi_m\big\|_{L^2}^2\]
    \[\lesssim \Big(\big\|\xi_m\big\|_{H^1}^2+\big\|\nabla_{\tau}\xi_m\big\|_{L^2}^2\bigg)\Big|_{\tau=1}+\int_{\tau}^1(\tau')^2\big\|F_m'\big\|_{L^2}^2d\tau'+\int_{\tau}^1\sqrt{\tau'}\big\|\xi_m\big\|_{L^2}^2d\tau'.\]
    We obtain the conclusion using Gronwall and bounding the initial data term at $\tau=1$.
\end{proof}

\subsection{Low frequency regime estimates}\label{low frequency backward section singular}
For the rest of Section~\ref{model system back direction section}, we denote the top order term $\xi=\nabla^M\Phi_0.$ We write \eqref{equation for Phi 0} as:
\begin{equation}\label{back main lin wave equation}
    \nabla_{\tau}\big(\nabla_{\tau}\xi\big)+\frac{1}{\tau}\nabla_{\tau}\xi-4\Delta\xi=\psi\nabla\xi+F_{M}',
\end{equation}
where we denote $F_M'=\psi\sum_{i=1}^I\nabla^{M+1}\Phi_i+F_{M}^{0}.$ 

As outlined in Section~\ref{second model system intro section} of the Introduction, we consider $X=2^{x+1}$ to be a large constant, to be chosen later depending only on $M,C_0,C_2.$ We consider the following regimes:
\begin{itemize}
    \item Negative frequencies $k<0$;
    \item Low frequency regime $0\leq k<x$ for all $\tau\in[0,1]$;
    \item Low frequency regime $k\geq x$ for $\tau\in[0,X2^{-k-1}]$;
    \item High frequency regime $k\geq x$ for $\tau\in[X2^{-k-1},1]$.
\end{itemize}
\textbf{Notation.} In addition to our previous notation convention, we write $A\lesssim_X B$ for some quantities $A,B>0$ if there exists a constant $C>0$ depending only on the constants $M,C_0,C_2$ and $X$, such that $A\leq CB.$ Otherwise, if we write $A\lesssim B$ then the implicit constant $C$ is independent of $X$. 

In order to deal with the negative frequencies, we can simply use the preliminary estimate \eqref{back singular preliminary intro}. Similarly, we have the following bound in the low frequency regime $0\leq k<x$:
\begin{proposition}\label{preliminary low low frequency estimate proposition}
    For $0\leq k<x$, we have the low frequency regime estimate for all $\tau\in[0,1]$:
    \[\big\|\tau P_k\nabla_{\tau}\xi\big\|_{L^2}^2+\big\|\tau\nabla P_k\xi\big\|_{L^2}^2\lesssim_X2^{-3k}\mathcal{D}_{II}+2^{-3k}\int_{\tau}^{1}(\tau')^2\big\|F_M'\big\|_{L^2}^2d\tau'.\]
\end{proposition}
\begin{proof}
    Since $0\leq k<x,$ we have using the preliminary estimate in Proposition~\ref{backward direction basic estimate proposition}:
    \[\big\|\tau\nabla P_k\xi\big\|_{L^2}^2+\big\|\tau P_k\nabla_{\tau}\xi\big\|_{L^2}^2\lesssim_X2^{-3k}\Big(\big\|\tau\xi\big\|_{L^2}^2+\big\|\tau\nabla_{\tau}\xi\big\|_{L^2}^2\Big)\lesssim_X2^{-3k}\mathcal{D}_{II}+2^{-3k}\int_{\tau}^{1}(\tau')^2\big\|F_M'\big\|_{L^2}^2d\tau'.\]
\end{proof}

We prove the main estimate in the low frequency regime $k\geq x$ for $\tau\in[0,X2^{-k-1}]$. The idea is to prove a similar estimate to \eqref{back singular preliminary intro} for $P_k\xi.$ We note that we follow the argument for the toy problem in \cite[Section~9]{Cmain}, while keeping track of the inhomogeneous terms.
\begin{proposition}\label{prelim low freq singular proposition}
    For any $0\leq\tau<X2^{-k-1}\leq1$ we have the low frequency regime estimate:
    \begin{align*}
        \big\|\tau P_k\nabla_{\tau}\xi\big\|_{L^2}^2+\big\|\tau\nabla P_k\xi\big\|_{L^2}^2\lesssim &X^22^{-2k}\Big(\big\|P_k\nabla_{\tau}\xi\big\|^2_{L^2}+\big\|\nabla P_k\xi\big\|^2_{L^2}\Big)\Big|_{\tau=X2^{-k-1}}+C_X2^{-3k}\mathcal{D}_{II}+\\
        &+C_X2^{-k}\int_{\tau}^{X2^{-k-1}}(\tau')^2\big\|P_kF_M'\big\|_{L^2}^2d\tau'+C_X2^{-3k}\int_{\tau}^{1}(\tau')^2\big\|F_M'\big\|_{L^2}^2d\tau'.
    \end{align*}
\end{proposition}
\begin{proof}
    We apply $P_k$ to \eqref{back main lin wave equation} to obtain for any $k\geq x$:
    \[\nabla_{\tau}\big(\tau P_k\nabla_{\tau}\xi\big)-4\tau\Delta P_k\xi=\tau P_k\big(\psi\nabla\xi\big)+\tau P_kF_M'+[\nabla_{\tau},P_k]\tau\nabla_{\tau}\xi.\]
    We contract the equation with $\tau P_k\nabla_{\tau}\xi$ and integrate by parts to obtain the following energy estimate:
    \[\big\|\tau P_k\nabla_{\tau}\xi\big\|_{L^2}^2+\big\|\tau\nabla P_k\xi\big\|_{L^2}^2\lesssim\]\[\lesssim X^22^{-2k}\Big(\big\|P_k\nabla_{\tau}\xi\big\|^2_{L^2}+\big\|\nabla P_k\xi\big\|^2_{L^2}\Big)\Big|_{X2^{-k-1}}+\int_{\tau}^{X2^{-k-1}}(\tau')^3\big\|\nabla P_k\xi\big\|_{L^2}\cdot\big\|\nabla [P_k,\nabla_4]\xi\big\|_{L^2}d\tau'+\]\[+\int_{\tau}^{X2^{-k-1}}(\tau')^3\big\|\nabla P_k\xi\big\|_{L^2}\cdot\big\|[\nabla,\nabla_4]P_k\xi\big\|_{L^2}d\tau'+\int_{\tau}^{X2^{-k-1}}(\tau')^2\big\|P_k\nabla_{\tau}\xi\big\|_{L^2}\cdot\big\|P_k\big(\psi\nabla\xi\big)\big\|_{L^2}d\tau'+\]\[+\int_{\tau}^{X2^{-k-1}}(\tau')^2\big\|P_k\nabla_{\tau}\xi\big\|_{L^2}\cdot\big\|[P_k,\nabla_4]\tau'\nabla_{\tau}\xi\big\|_{L^2}d\tau'+\int_{\tau}^{X2^{-k-1}}(\tau')^2\big\|P_k\nabla_{\tau}\xi\big\|_{L^2}\cdot\big\|P_kF_M'\big\|_{L^2}d\tau'.\]
    We point out that similarly to the proof of Proposition~\ref{backward direction basic estimate proposition}, we obtain a bulk term with a favorable sign, which we drop. We use Lemma \ref{Litt Paley lemma} to bound the commutation terms:
    \begin{align*}
        \big\|\tau P_k\nabla_{\tau}\xi\big\|_{L^2}^2+\big\|\tau\nabla P_k\xi\big\|_{L^2}^2\lesssim X^22^{-2k}\Big(\big\|P_k\nabla_{\tau}\xi\big\|^2_{L^2}+\big\|\nabla P_k\xi\big\|^2_{L^2}\Big)\Big|_{X2^{-k-1}}+\int_{\tau}^{X2^{-k-1}}(\tau')^3\big\|\nabla P_k\xi\big\|_{L^2}\big\|\xi\big\|_{H^1}d\tau'\\
        +\int_{\tau}^{X2^{-k-1}}(\tau')^2\big\|P_k\nabla_{\tau}\xi\big\|_{L^2}\big\|\nabla P_k\xi\big\|_{L^2}d\tau'+\int_{\tau}^{X2^{-k-1}}(\tau')^22^{-k}\big\|P_k\nabla_{\tau}\xi\big\|_{L^2}\big\|\xi\big\|_{H^1}d\tau'+\\
        +\int_{\tau}^{X2^{-k-1}}(\tau')^2\big\|P_k\nabla_{\tau}\xi\big\|_{L^2}\big\|\tau'\nabla_{\tau}\xi\big\|_{L^2}d\tau'+\int_{\tau}^{X2^{-k-1}}(\tau')^2\big\|P_k\nabla_{\tau}\xi\big\|_{L^2}\big\|P_kF_M'\big\|_{L^2}d\tau'.
    \end{align*}
    Using Cauchy-Schwarz, we obtain the bound:
    \[\big\|\tau P_k\nabla_{\tau}\xi\big\|_{L^2}^2+\big\|\tau\nabla P_k\xi\big\|_{L^2}^2\lesssim X^22^{-2k}\Big(\big\|P_k\nabla_{\tau}\xi\big\|^2_{L^2}+\big\|\nabla P_k\xi\big\|^2_{L^2}\Big)\Big|_{X2^{-k-1}}+\]
    \[+\int_{\tau}^{X2^{-k-1}}\frac{2^k}{X}\big\|\tau'\nabla P_k\xi\big\|_{L^2}^2d\tau'+\int_{\tau}^{X2^{-k-1}}\frac{2^k}{X}\big\|\tau'P_k\nabla_{\tau}\xi\big\|_{L^2}^2d\tau'+C_X\int_{\tau}^{X2^{-k-1}}2^{-3k}(\tau')^2\big\|\xi\big\|_{H^1}^2d\tau'+\]\[+C_X\int_{\tau}^{X2^{-k-1}}2^{-k}(\tau')^4\big\|\nabla_{\tau}\xi\big\|_{L^2}^2d\tau'+2^{-k}C_X\int_{\tau}^{X2^{-k-1}}(\tau')^2\big\|P_kF_M'\big\|_{L^2}^2d\tau'.\]
    Using the Gronwall inequality for $\tau\leq X2^{-k-1}\leq1$ we get:
    \[\big\|\tau\nabla P_k\xi\big\|_{L^2}^2+\big\|\tau P_k\nabla_{\tau}\xi\big\|_{L^2}^2\lesssim X^22^{-2k}\Big(\big\|P_k\nabla_{\tau}\xi\big\|^2_{L^2}+\big\|\nabla P_k\xi\big\|^2_{L^2}\Big)\Big|_{X2^{-k-1}}+\]\[+C_X\int_{\tau}^{X2^{-k-1}}2^{-3k}(\tau')^2\big\|\xi\big\|_{H^1}^2+C_X\int_{\tau}^{X2^{-k-1}}2^{-k}(\tau')^4\big\|\nabla_{\tau}\xi\big\|_{L^2}^2+2^{-k}C_X\int_{\tau}^{X2^{-k-1}}(\tau')^2\big\|P_kF_M'\big\|_{L^2}^2.\]
    Finally, we use the above estimate \eqref{back singular preliminary intro} to complete the proof.
\end{proof}

\subsection{High frequency regime estimates}\label{prelim high frequency estimate singular section}

In this section, we prove a high frequency regime estimate for the top order term $\xi=\nabla^M\Phi_0.$ The proof is similar to that of Proposition \ref{high frequency forward estimate}, but in the current case the bulk term has an unfavorable sign, which will create several complications in Section~\ref{improved high frequency estimates section}. Moreover, it is essential that the implicit constant in the estimate obtained is independent of the parameter $X.$

\begin{proposition}\label{high freq back estimate} $\xi$ satisfies the high frequency regime estimate for any $\tau\in\big[X2^{-k-1},1\big]:$
    \[2^k\tau\big\|P_k\nabla_{\tau}\xi\big\|^2_{L^2}+\frac{2^k}{\tau}\big\|P_k\xi\big\|^2_{L^2}+2^k\tau\big\|\nabla P_k\xi\big\|^2_{L^2}\lesssim\bigg(2^k\big\|P_k\nabla_{\tau}\xi\big\|^2_{L^2}+2^k\big\|P_k\xi\big\|^2_{L^2}+2^k\big\|\nabla P_k\xi\big\|^2_{L^2}\bigg)\bigg|_{\tau=1}\]\[+\int_{\tau}^1\frac{2^k}{(\tau')^2}\big\|P_k\xi\big\|_{L^2}^2d\tau'+\int_{\tau}^12^k\tau'\big\|\underline{\widetilde{P}}_k\xi\big\|_{L^2}^2d\tau'+\int_{\tau}^12^k(\tau')^3\big\|\underline{\widetilde{P}}_k\nabla\xi\big\|_{L^2}^2d\tau'+\int_{\tau}^12^k(\tau')^3\big\|\underline{\widetilde{P}}_k\nabla_{\tau}\xi\big\|_{L^2}^2d\tau'+\]\[+\int_{\tau}^1\frac{\tau'}{2^{k}}\big\|\xi\big\|_{H^1}^2d\tau'+\int_{\tau}^1\frac{(\tau')^3}{2^{k}}\big\|\nabla_{\tau}\xi\big\|_{L^2}^2d\tau'+\int_{\tau}^12^k\tau'\big\|P_kF'_M\big\|_{L^2}^2d\tau'.\]
\end{proposition}
\begin{proof} We introduce the new time variable $t=X^{-1}2^k\tau\in[1/2,X^{-1}2^k]$. The equation \eqref{back main lin wave equation} becomes:
    \[\nabla_{t}\big(\nabla_{t}\xi\big)+\frac{1}{t}\nabla_{t}\xi-4\frac{X^2}{2^{2k}}\cdot\Delta\xi=\frac{X^2}{2^{2k}}\cdot\psi\nabla\xi+\frac{X^2}{2^{2k}}\cdot F_M'.\]
    We multiply by $\sqrt{t}$ to get:
    \[\nabla_{t}\big(\nabla_{t}(\xi\sqrt{t})\big)+\frac{1}{4t^2}\xi\sqrt{t}-4\frac{X^2}{2^{2k}}\cdot\Delta\xi\sqrt{t}=\frac{X^2}{2^{2k}}\cdot\psi\nabla\xi\sqrt{t}+\frac{X^2}{2^{2k}}\cdot F_M'\sqrt{t}.\]
    We apply $P_k$ to the equation:
    \begin{align*}
        \nabla_{t}\big(P_k\nabla_{t}(\xi\sqrt{t})\big)+\frac{1}{4t^2}P_k\xi\sqrt{t}-4\frac{X^2}{2^{2k}}\Delta P_k\xi\sqrt{t}=&\frac{X^2}{2^{2k}}\psi\nabla P_k\xi\sqrt{t}+\frac{X^2}{2^{2k}}\psi[P_k,\nabla]\xi\sqrt{t}+\\
        &+\frac{X^2}{2^{2k}}[P_k,\psi]\nabla \xi\sqrt{t}+\frac{X^2P_kF_M'\sqrt{t}}{2^{2k}}+[\nabla_t,P_k]\nabla_{t}\xi\sqrt{t}.
    \end{align*}
    We contract the equation with $P_k\nabla_t(\xi\sqrt{t})$ and integrate by parts. We notice that using the analogue of Lemma \ref{volume form lemma t} for the new time variable $t$ does not introduce any constants that depend on $X.$ We obtain the energy estimate:
    \[\big\|P_k\nabla_{t}\xi\sqrt{t}\big\|^2_{L^2}+\frac{1}{t^2}\big\|P_k\xi\sqrt{t}\big\|^2_{L^2}+\frac{X^2}{2^{2k}}\big\|\nabla P_k\xi\sqrt{t}\big\|^2_{L^2}\lesssim\]\[\lesssim\bigg(X^{-1}2^k\big\|P_k\nabla_{t}\xi\big\|^2_{L^2}+\frac{X}{2^{k}}\big\|P_k\xi\big\|^2_{L^2}+\frac{X}{2^{k}}\big\|\nabla P_k\xi\big\|^2_{L^2}\bigg)\bigg|_{t=X^{-1}2^k}+\int_t^{X^{-1}2^k}\frac{1}{(t')^2}\big\|P_k\xi\big\|_{L^2}^2dt'+\]\[+\int_t^{X^{-1}2^k}\int_{S^n}\frac{1}{(t')^2}\big|P_k\xi\sqrt{t}\big|\cdot\big|[P_k,\nabla_t]\xi\sqrt{t'}\big|dt'+\int_t^{X^{-1}2^k}\int_{S^n}\frac{X^2}{2^{2k}}\big|\nabla P_k\xi\sqrt{t'}\big|\cdot\big|\nabla[P_k,\nabla_t]\xi\sqrt{t'}\big|dt'+\]\[+\int_t^{X^{-1}2^k}\int_{S^n}\frac{X^2}{2^{2k}}\big|\nabla P_k\xi\sqrt{t'}\big|\cdot\big|[\nabla,\nabla_t]P_k\xi\sqrt{t'}\big|dt'+\int_t^{X^{-1}2^k}\int_{S^n}\frac{X^2}{2^{2k}}\big|P_k\nabla_t\xi\sqrt{t'}\big|\cdot\big|\nabla P_k\xi\sqrt{t'}\big|dt'+\]\[+\int_t^{X^{-1}2^k}\int_{S^n}\frac{X^2}{2^{2k}}\big|P_k\nabla_t\xi\sqrt{t'}\big|\cdot\big|[P_k,\nabla]\xi\sqrt{t'}\big|dt'+\int_t^{X^{-1}2^k}\int_{S^n}\frac{X^2}{2^{2k}}\big|P_k\nabla_t\xi\sqrt{t'}\big|\cdot\big|[P_k,\psi]\nabla\xi\sqrt{t'}\big|dt'+\]\[+\int_t^{X^{-1}2^k}\int_{S^n}\frac{X^2}{2^{2k}}\big|P_k\nabla_t\xi\sqrt{t'}\big|\cdot\big|P_kF_M'\sqrt{t'}\big|dt'+\int_t^{X^{-1}2^k}\int_{S^n}\big|P_k\nabla_t\xi\sqrt{t'}\big|\cdot\big|[\nabla_t,P_k]\nabla_{t}\xi\sqrt{t'}\big|dt'.\]
    We use Lemma \ref{Litt Paley lemma refined}, and Gronwall for $t\in\big[1/2,X^{-1}2^k\big]$ (to deal with the fifth and sixth error terms):
    \[\big\|P_k\nabla_{t}\xi\sqrt{t}\big\|^2_{L^2}+\frac{1}{t^2}\big\|P_k\xi\sqrt{t}\big\|^2_{L^2}+\frac{X^2}{2^{2k}}\big\|\nabla P_k\xi\sqrt{t}\big\|^2_{L^2}\lesssim\]\[\lesssim\bigg(X^{-1}2^k\big\|P_k\nabla_{t}\xi\big\|^2_{L^2}+\frac{X}{2^{k}}\big\|P_k\xi\big\|^2_{L^2}+\frac{X}{2^{k}}\big\|\nabla P_k\xi\big\|^2_{L^2}\bigg)\bigg|_{t=X^{-1}2^k}+\int_t^{X^{-1}2^k}\frac{1}{(t')^2}\big\|P_k\xi\big\|_{L^2}^2dt'+\]\[+\int_t^{X^{-1}2^k}\frac{X^2}{2^{2k}t'}\big\|P_k\xi\sqrt{t'}\big\|_{L^2}\cdot\Big(\big\|\underline{\widetilde{P}}_k\xi\sqrt{t'}\big\|_{L^2}+2^{-k}\big\|\xi\sqrt{t'}\big\|_{L^2}\Big)dt'+\]\[+\int_t^{X^{-1}2^k}\frac{X^4t'}{2^{4k}}\big\|\nabla P_k\xi\sqrt{t'}\big\|_{L^2}\cdot\bigg(\big\|\underline{\widetilde{P}}_k\nabla\xi\sqrt{t'}\big\|_{L^2}+2^{-k}\big\|\xi\sqrt{t'}\big\|_{H^1}\bigg)+\int_t^{X^{-1}2^k}\frac{X^2}{2^{3k}}\big\|P_k\nabla_t\xi\sqrt{t'}\big\|_{L^2}\cdot\big\|\xi\sqrt{t'}\big\|_{L^2}+\]\[+\int_t^{X^{-1}2^k}\frac{X^2}{2^{3k}}\big\|P_k\nabla_t\xi\sqrt{t'}\big\|_{L^2}\cdot\big\|\nabla\xi\sqrt{t'}\big\|_{L^2}dt'+\int_t^{X^{-1}2^k}\frac{X^2}{2^{2k}}\big\|P_k\nabla_t\xi\sqrt{t'}\big\|_{L^2}\big\|P_kF_M'\sqrt{t'}\big\|_{L^2}dt'+\]\[+\int_t^{X^{-1}2^k}\frac{X^2t'}{2^{2k}}\big\|P_k\nabla_t\xi\sqrt{t'}\big\|_{L^2}\Big(\big\|\underline{\widetilde{P}}_k\nabla_t\xi\sqrt{t'}\big\|_{L^2}+2^{-k}\big\|\nabla_t\xi\sqrt{t'}\big\|_{L^2}\Big)dt'.\]
    We use Gronwall for $t\in\big[1/2,X^{-1}2^k\big]$ to get:
    \[\big\|P_k\nabla_{t}\xi\sqrt{t}\big\|^2_{L^2}+\frac{1}{t^2}\big\|P_k\xi\sqrt{t}\big\|^2_{L^2}+\frac{X^2}{2^{2k}}\big\|\nabla P_k\xi\sqrt{t}\big\|^2_{L^2}\lesssim\]\[\lesssim\bigg(X^{-1}2^k\big\|P_k\nabla_{t}\xi\big\|^2_{L^2}+\frac{X}{2^{k}}\big\|P_k\xi\big\|^2_{L^2}+\frac{X}{2^{k}}\big\|\nabla P_k\xi\big\|^2_{L^2}\bigg)\bigg|_{t=X^{-1}2^k}+\int_t^{X^{-1}2^k}\frac{1}{(t')^2}\big\|P_k\xi\big\|_{L^2}^2dt'+\]\[+\int_t^{X^{-1}2^k}\frac{X^3}{2^{3k}}\big\|\underline{\widetilde{P}}_k\xi\sqrt{t'}\big\|_{L^2}^2dt'+\int_t^{X^{-1}2^k}\frac{X^5(t')^2}{2^{5k}}\big\|\underline{\widetilde{P}}_k\nabla\xi\sqrt{t'}\big\|_{L^2}^2dt'+\int_t^{X^{-1}2^k}\frac{X^3(t')^2}{2^{3k}}\big\|\underline{\widetilde{P}}_k\nabla_t\xi\sqrt{t'}\big\|_{L^2}^2dt'+\]\[+\int_t^{X^{-1}2^k}\frac{X^3t'}{2^{5k}}\big\|\xi\big\|_{H^1}^2dt'+\int_t^{X^{-1}2^k}\frac{X^3(t')^3}{2^{5k}}\big\|\nabla_t\xi\big\|_{L^2}^2dt'+\int_t^{X^{-1}2^k}\frac{X^3t'}{2^{3k}}\big\|P_kF_M'\big\|_{L^2}^2dt'.\]
    We change variables to $\tau$ and we obtain the conclusion.
\end{proof}

\subsection{Improved high frequency estimates}\label{improved high frequency estimates section}

The goal of this section is to improve the high frequency regime estimate of Proposition \ref{high freq back estimate} in order to prove estimates in the high frequency regime only in terms of data and the inhomogeneous terms.
\begin{proposition}\label{improved high frequency estimate proposition}
    We denote by $\mathbf{1}_{k,\tau}$ the characteristic function of $\{1\geq\tau\geq X2^{-k-1}\},$ and we denote:
    \[a_k(\tau):=\tau\big\|P_k\nabla_{\tau}\xi\big\|^2_{L^2}+\frac{1}{\tau}\big\|P_k\xi\big\|^2_{L^2}+\tau\big\|\nabla P_k\xi\big\|^2_{L^2}.\]
    We have the improved high frequency estimate for any $\tau\in(0,1]$:
    \begin{equation}\label{optimal high frequency estimate}
        \sum_{\tau\geq X2^{-k-1}}2^ka_k(\tau)\mathbf{1}_{k,\tau}\lesssim \mathcal{D}_{II}+\int_{\tau}^1\tau'\big\|F_M'\big\|_{H^{1/2}}^2d\tau'.
    \end{equation}
\end{proposition}

The rest of this section is dedicated to the proof of this proposition. We follow the detailed outline of the proof from Section~\ref{second model system intro section} of the Introduction, and we divide the proof into multiple steps. We note that the proof is very similar to the toy problem considered in \cite[Section~9]{Cmain}, but in the current case we also need to keep track of the inhomogeneous terms. Additionally, in the present proof we fill in several of the details that we omitted in \cite{Cmain}, such as the proof of the Gronwall-like inequality in Lemma~\ref{grwonwall type lemma}.

Throughout the proof we use the schematic notation $\{d_k\}_{k\geq0}$ for data terms at $\tau=1$ which satisfy: \[\sum_{k\geq0} d_k\lesssim\mathcal{D}_{II}.\]
\paragraph{Consequences of Proposition~\ref{high freq back estimate}.} The starting point is the preliminary high frequency regime estimate in Proposition~\ref{high freq back estimate}. Using the above notation, we have for all $\tau\in[X2^{-k-1},1]$:
\[2^ka_k(\tau)\mathbf{1}_{k,\tau}\lesssim2^ka_k(1)+ \mathbf{1}_{k,\tau}\int_{\tau}^1\frac{2^k}{(\tau')^2}\big\|P_k\xi\big\|_{L^2}^2d\tau'+\mathbf{1}_{k,\tau}\int_{\tau}^1e_k(\tau')+\mathbf{1}_{k,\tau}\int_{\tau}^1\overline{e_k}(\tau')+\int_{\tau}^12^k\tau'\big\|P_kF_M'\big\|_{L^2}^2,\]
where we define the energies:
\begin{align*}
e_k&=2^k\tau^3\big\|\underline{\widetilde{P}}_k\nabla\xi\big\|_{L^2}^2+2^k\tau^3\big\|\underline{\widetilde{P}}_k\nabla_{\tau}\xi\big\|_{L^2}^2,\\
    \overline{e_k}&=2^k\tau\big\|\underline{\widetilde{P}}_k\xi\big\|_{L^2}^2+\frac{\tau}{2^{k}}\big\|\xi\big\|_{H^1}^2+\frac{\tau^3}{2^{k}}\big\|\nabla_{\tau}\xi\big\|_{L^2}^2.
\end{align*}

The error term containing $e_k$ is similar to the error terms with different projection operators of Section~\ref{model system for direction section}, as explained in Section~\ref{second model system intro section}. We deal with these terms towards the end of our argument, when summing the estimates obtained for all $k\geq x.$ In the meantime we simply keep track of these terms, similarly to the inhomogeneous terms.

On the other hand, we can already bound the error terms containing $\overline{e_k}$ using the preliminary estimate in Proposition~\ref{backward direction basic estimate proposition}. We can bound the first term in $\overline{e_k}$ by the second term, using the finite band property for LP projections. We then apply (\ref{back singular preliminary intro}) to get:
\[\int_{\tau}^1\overline{e_k}(\tau')d\tau'\lesssim_X d_k+2^{-k}\int_{\tau}^1(\tau')^2\big\|F_M'\big\|_{L^2}^2d\tau'.\]
As a consequence, we proved that for all $\tau\in[X2^{-k-1},1]$ we have the high frequency regime estimate:
\[2^ka_k(\tau)\mathbf{1}_{k,\tau}\lesssim C_Xd_k+ \mathbf{1}_{k,\tau}\int_{\tau}^1\frac{2^k}{(\tau')^2}\big\|P_k\xi\big\|_{L^2}^2+\mathbf{1}_{k,\tau}\int_{\tau}^1e_k(\tau')+C_X2^{-k}\int_{\tau}^1(\tau')^2\big\|F_M'\big\|_{L^2}^2+\int_{\tau}^12^k\tau'\big\|P_kF_M'\big\|_{L^2}^2.\]

\paragraph{Applying the refined Poincaré inequality.} According to Section~\ref{second model system intro section} of the Introduction, the second error term in the above estimate causes significant challenges. We also explained in Section~\ref{LP theory intro} that we could not bound this term directly using the finite band property for LP projections. Instead, we use the refined Poincaré inequality for LP projections \eqref{ref Poincarey}. As a result, we have for all $\tau\in[X2^{-k-1},1]$ and $\delta>0$:
\[\int_{\tau}^1\frac{2^k}{(\tau')^2}\big\|P_k\xi\big\|_{L^2}^2d\tau'\lesssim \frac{1}{\delta}\int_{\tau}^1\frac{2^{-k}}{(\tau')^2}\big\|\nabla P_k\xi\big\|_{L^2}^2d\tau'+\delta\int_{\tau}^1\frac{1}{(\tau')^2}\sum_{l=0}^{k-1}2^{-8k+7l}\big\|\nabla P_l\xi\big\|_{L^2}^2d\tau'+\frac{1}{\delta}\int_{\tau}^1\frac{2^{-3k}}{(\tau')^2}\big\|\xi\big\|_{L^2}^2d\tau'.\]
The last term in this inequality is bounded using the preliminary estimate (\ref{back singular preliminary intro}):
\[\int_{\tau}^1\frac{2^{-3k}}{(\tau')^2}\big\|\xi\big\|_{L^2}^2d\tau'\lesssim_X 2^{-k}\mathcal{D}_{II}+2^{-k}\int_{\tau}^1(\tau')^2\big\|F_M'\big\|_{L^2}^2d\tau'.\]
As a result, there exist constants $C',C_{X},C_{X,\delta}>0$ such that for all $k\geq x$ and $\tau\in[X2^{-k-1},1]$:
\[2^ka_k(\tau)\leq C_{X,\delta}d_k+\frac{C'}{\delta}\int_{\tau}^1\frac{2^{-2k}}{(\tau')^3}\cdot2^ka_k(\tau')d\tau'+C'\delta\int_{\tau}^1\frac{1}{(\tau')^2}\sum_{l=0}^{k-1}2^{-8k+7l}\big\|\nabla P_l\xi\big\|_{L^2}^2d\tau'+\]\[+C_{X,\delta}\int_{\tau}^1e_k(\tau')d\tau'+C_{X,\delta}2^{-k}\int_{\tau}^1(\tau')^2\big\|F_M'\big\|_{L^2}^2d\tau'+C_{X,\delta}\int_{\tau}^12^k\tau'\big\|P_kF_M'\big\|_{L^2}^2d\tau'.\]

We can deal with the second error term by using Gronwall for $\tau\in[X2^{-k-1},1]$. We compute:
\[\exp\bigg(\frac{C'}{\delta}\int_{\tau}^1\frac{2^{-2k}}{(\tau')^3}d\tau'\bigg)\leq\exp\bigg(\frac{C'}{\delta}\cdot\frac{1}{X^2}\bigg)\leq2,\]
where we fix $X>0$ large enough, depending on $C'$ and $\delta$, such that:
\begin{equation}\label{fix constant X}
    \frac{C'}{\delta}\cdot\frac{1}{X^2}\leq\log(2).
\end{equation}
Therefore, we have for all $\tau\in[X2^{-k-1},1]$:
\[2^ka_k(\tau)\leq C_{\delta}d_k+2C'\delta\int_{\tau}^1\frac{1}{(\tau')^2}\sum_{l=0}^{k-1}2^{-8k+7l}\big\|\nabla P_l\xi\big\|_{L^2}^2+C_{\delta}\int_{\tau}^1e_k(\tau')+C_{\delta}\int_{\tau}^12^{-k}(\tau')^2\big\|F_M'\big\|_{L^2}^2+2^k\tau'\big\|P_kF_M'\big\|_{L^2}^2.\]

\paragraph{Bounding the low frequency regime error terms.} As a consequence of the refined Poincaré inequality, the second term in the above has both a low frequency regime and a high frequency regime part. We want to separate these two, and apply the low frequency regime estimates in Section~\ref{low frequency backward section singular}. This process creates a sum of discrete error terms which pose additional difficulties. We notice that we can write:
\[\mathbf{1}_{k,\tau}\int_{\tau}^1\frac{1}{(\tau')^3}\sum_{l=0}^{k-1}2^{-8k+7l}\tau'\big\|\nabla P_l\xi\big\|_{L^2}^2d\tau'\leq\mathbf{1}_{k,\tau}\int_{\tau}^1\frac{1}{(\tau')^3}\sum_{l=x}^{k-1}2^{-8k+6l}\cdot 2^la_l(\tau')\mathbf{1}_{l,\tau'}d\tau'+\]\[+\mathbf{1}_{k,\tau}\int_{\tau}^1\frac{1}{(\tau')^4}\sum_{l=0}^{x-1}2^{-8k+7l}\big\|\tau'\nabla P_l\xi\big\|_{L^2}^2d\tau'+\mathbf{1}_{k,\tau}\int_{\tau}^1\frac{1}{(\tau')^4}\sum_{\tau'<X2^{-l-1}\leq1}2^{-8k+7l}\big\|\tau'\nabla P_l\xi\big\|_{L^2}^2d\tau',\]
where the first term is in the high frequency regime and the last two terms are in the low frequency regime. We use Proposition~\ref{preliminary low low frequency estimate proposition} to obtain:
\[\mathbf{1}_{k,\tau}\int_{\tau}^1\frac{1}{(\tau')^4}\sum_{l=0}^{x-1}2^{-8k+7l}\big\|\tau'\nabla P_l\xi\big\|_{L^2}^2\lesssim_{X}\mathcal{D}_{II}\cdot\mathbf{1}_{k,\tau}\frac{2^{-8k}}{\tau^7}+\mathbf{1}_{k,\tau}\frac{2^{-8k}}{\tau^7}\int_{\tau}^1(\tau')^2\big\|F_M'\big\|_{L^2}^2\lesssim_X d_k+2^{-k}\int_{\tau}^1(\tau')^2\big\|F_M'\big\|_{L^2}^2.\]
Similarly, we use Proposition \ref{prelim low freq singular proposition} to obtain:
\[\mathbf{1}_{k,\tau}\int_{\tau}^1\frac{1}{(\tau')^4}\sum_{\tau'<X2^{-l-1}\leq1}2^{-8k+7l}\big\|\tau'\nabla P_l\xi\big\|_{L^2}^2d\tau'=\mathbf{1}_{k,\tau}\sum_{\tau<X2^{-l-1}\leq1}\int_{\tau}^{X2^{-l-1}}\frac{1}{(\tau')^4}2^{-8k+7l}\big\|\tau'\nabla P_l\xi\big\|_{L^2}^2d\tau'\]
\[\lesssim \mathbf{1}_{k,\tau}X\sum_{\tau<X2^{-l-1}\leq1}\frac{2^{-8k+6l}}{\tau^3}a_l(X2^{-l-1})+C_X\mathcal{D}_{II}\cdot\mathbf{1}_{k,\tau}\frac{2^{-8k}}{\tau^7}+\]\[+C_X\mathbf{1}_{k,\tau}\frac{2^{-8k}}{\tau^7}\int_{\tau}^1(\tau')^2\big\|F_M'\big\|_{L^2}^2d\tau'+C_X\mathbf{1}_{k,\tau}\frac{2^{-8k}}{\tau^7}\int_{\tau}^1\tau'\big\|F_M'\big\|_{H^{1/2}}^2d\tau'.\]
We notice that in the above estimate the last term is obtained by bounding:
\[\sum_{\tau<X2^{-l-1}\leq1}\int_{\tau}^{X2^{-l-1}}\frac{1}{(\tau')^4}2^{-8k+6l}\int_{\tau'}^{X2^{-l-1}}\Tilde{\tau}^2\big\|P_lF_M'\big\|_{L^2}^2d\Tilde{\tau}d\tau'\lesssim_X\sum_{\tau<X2^{-l-1}\leq1}\frac{1}{\tau^3}2^{-8k+4l}\int_{\tau}^{1}\Tilde{\tau}2^l\big\|P_lF_M'\big\|_{L^2}^2d\Tilde{\tau}.\]

As a result, we proved that there exist constants $C'',C_{\delta}>0$ such that for all $\tau\in[X2^{-k-1},1]$:
\[2^ka_k(\tau)\leq C_{\delta}d_k+C''\delta\cdot\mathbf{1}_{k,\tau}\int_{\tau}^1\frac{1}{(\tau')^3}\sum_{l=x}^{k-1}2^{-8k+6l}\cdot 2^la_l(\tau')\mathbf{1}_{l,\tau'}d\tau'+C''\delta X\mathbf{1}_{k,\tau}\sum_{\tau<X2^{-l-1}\leq1}\frac{2^{-8k+6l}}{\tau^3}a_l(X2^{-l-1})\]\[+C_{\delta}\int_{\tau}^1e_k(\tau')d\tau'+C_{\delta}\int_{\tau}^12^k\tau'\big\|P_kF_M'\big\|_{L^2}^2d\tau'+C_{\delta}2^{-k}\int_{\tau}^1(\tau')^2\big\|F_M'\big\|_{L^2}^2d\tau'+C_{\delta}2^{-k}\int_{\tau}^1\tau'\big\|F_M'\big\|_{H^{1/2}}^2d\tau'.\]

\paragraph{Gronwall-like inequality.} Our next goal is to deal with the second term in the above estimate using a suitable Gronwall-like inequality. We set-up the problem in order to isolate this error term. We consider $\delta>0$ to be small enough, so that:
\begin{equation}\label{fix constant delta 1}
    C''\delta<\frac{1}{10}.
\end{equation}
We introduce some notation for the error terms on the right hand side of the above estimate:
\begin{align*}
    S_k(\tau)&=\mathbf{1}_{k,\tau}C''\delta X\sum_{\tau<X2^{-l-1}\leq1}\frac{2^{-8k+6l}}{\tau^3}a_l(X2^{-l-1}),\\
    E_k(\tau)&=\mathbf{1}_{k,\tau}C_{\delta}\int_{\tau}^1e_k(\tau')d\tau',\\
    I_k(\tau)&=\mathbf{1}_{k,\tau}C_{\delta}2^{-k}\int_{\tau}^1\tau'\big\|F_M'\big\|_{H^{1/2}}^2d\tau'+\mathbf{1}_{k,\tau}C_{\delta}\int_{\tau}^12^k\tau'\big\|P_kF_M'\big\|_{L^2}^2d\tau',\\
    A_k(\tau)&=C_{\delta}d_k+S_k(\tau)+E_k(\tau)+I_k(\tau).
\end{align*}
Using this notation, we have that for all $k\geq x$ and $\tau\in[X2^{-k-1},1]$:
\begin{equation}\label{second schematic eq}
    2^ka_k(\tau)\mathbf{1}_{k,\tau}\leq A_k(\tau)+\frac{1}{10}\cdot\mathbf{1}_{k,\tau}\int_{\tau}^1\frac{1}{(\tau')^3}\sum_{l=x}^{k-1}2^{-8k+6l}\cdot 2^la_l(\tau')\mathbf{1}_{l,\tau'}d\tau'.
\end{equation}
This motivates us to prove the following Gronwall-like lemma:
\begin{lemma}\label{grwonwall type lemma}
    We consider the functions $u,A,b,c:\mathbb{N}\times[0,1]\rightarrow[0,\infty),$ which for all $k\geq x,\ \tau\in(0,1]$ satisfy the inequality:
    \begin{equation}\label{inequality assumption}
        u(k,\tau)\leq A(k,\tau)+b(k)\int_{\tau}^1\sum_{l=x}^{k-1}c(l,\tau')u(l,\tau')d\tau'
    \end{equation}
    Then, we have that for all $k\geq x,\ \tau\in(0,1]$:
    \begin{equation}\label{inequality conclusion}
        u(k,\tau)\leq A(k,\tau)+b(k)\int_{\tau}^1\sum_{l=x}^{k-1}c(l,\tau')A(l,\tau')\prod_{j=l+1}^{k-1}\bigg(1+\int_{\tau}^{\tau'}b(j)c(j,\tau'')d\tau''\bigg)d\tau'.
    \end{equation}
\end{lemma}
\begin{proof}
    We prove this by induction on $k.$ In the case $k=x,$ we have from (\ref{inequality assumption}) that $u(x,\tau)\leq A(x,\tau)$ as desired. Next, we assume that (\ref{inequality conclusion}) holds for $x,x+1,\ldots,k$ and we prove it for $k+1.$ We have:
    \[u(k+1,\tau)\leq A(k+1,\tau)+b(k+1)\int_{\tau}^1\sum_{l=x}^{k}c(l,\tau')u(l,\tau')d\tau'\]
    \[\leq A(k+1,\tau)+b(k+1)\int_{\tau}^1\sum_{l=x}^{k}c(l,\tau')\bigg\{A(l,\tau')+b(l)\int_{\tau'}^1\sum_{i=x}^{l-1}c(i,\tau'')A(i,\tau'')\prod_{j=i+1}^{l-1}\bigg(1+\int_{\tau'}^{\tau''}b(j)c(j,\Tilde{\tau})\bigg)d\tau''\bigg\}d\tau'\]
    \[\leq A(k+1,\tau)+b(k+1)\int_{\tau}^1\sum_{l=x}^{k}c(l,\tau')A(l,\tau')d\tau'+\]
    \[+b(k+1)\int_{\tau}^1\int_{\tau'}^1\sum_{l=x}^{k}\sum_{i=x}^{l-1}\bigg\{c(l,\tau')b(l)c(i,\tau'')A(i,\tau'')\prod_{j=i+1}^{l-1}\bigg(1+\int_{\tau'}^{\tau''}b(j)c(j,\Tilde{\tau})d\Tilde{\tau}\bigg)\bigg\}d\tau''d\tau'\]
    \[\leq A(k+1,\tau)+b(k+1)\int_{\tau}^1\sum_{l=x}^{k}c(l,\tau')A(l,\tau')d\tau'+\]
    \[+b(k+1)\int_{\tau}^1\int_{\tau}^{\tau''}\sum_{i=x}^{k-1}\sum_{l=i+1}^{k}\bigg\{c(l,\tau')b(l)c(i,\tau'')A(i,\tau'')\prod_{j=i+1}^{l-1}\bigg(1+\int_{\tau'}^{\tau''}b(j)c(j,\Tilde{\tau})d\Tilde{\tau}\bigg)\bigg\}d\tau'd\tau''\]
    \[\leq A(k+1,\tau)+b(k+1)\int_{\tau}^1\sum_{l=x}^{k}c(l,\tau')A(l,\tau')d\tau'+\]
    \[+b(k+1)\int_{\tau}^1\sum_{i=x}^{k-1}c(i,\tau'')A(i,\tau'')\bigg\{\int_{\tau}^{\tau''}\sum_{l=i+1}^{k}c(l,\tau')b(l)\prod_{j=i+1}^{l-1}\bigg(1+\int_{\tau'}^{\tau''}b(j)c(j,\Tilde{\tau})d\Tilde{\tau}\bigg)d\tau'\bigg\}d\tau''\]
    \[\leq A(k+1,\tau)+b(k+1)\int_{\tau}^1\sum_{i=x}^{k}c(i,\tau'')A(i,\tau'')\bigg\{1+\int_{\tau}^{\tau''}\sum_{l=i+1}^{k}c(l,\tau')b(l)\prod_{j=i+1}^{l-1}\bigg(1+\int_{\tau}^{\tau''}b(j)c(j,\Tilde{\tau})d\Tilde{\tau}\bigg)d\tau'\bigg\}d\tau''\]
    We conclude, since we have:
    \[1+\sum_{l=i+1}^{k}\int_{\tau}^{\tau''}c(l,\tau')b(l)d\tau'\cdot\prod_{j=i+1}^{l-1}\bigg(1+\int_{\tau}^{\tau''}b(j)c(j,\Tilde{\tau})d\Tilde{\tau}\bigg)=\prod_{j=i+1}^{k}\bigg(1+\int_{\tau}^{\tau''}b(j)c(j,\Tilde{\tau})d\Tilde{\tau}\bigg).\]
\end{proof}

\paragraph{Consequences of Lemma~\ref{grwonwall type lemma}.} We use the Gronwall lemma established above for the inequality \eqref{second schematic eq}.
\begin{corollary}
    For all $k\geq x$ and $\tau\in[X2^{-k-1},1]$, we have the high frequency regime estimate:
    \begin{equation}\label{third schematic eq}
        2^ka_k(\tau)\mathbf{1}_{k,\tau}\lesssim A_k(\tau)+2^{-7k}\int_{\tau}^1\sum_{l=0}^{k-1}\frac{2^{5l}}{(\tau')^3}\cdot A_l(\tau')\mathbf{1}_{l,\tau'}d\tau'
    \end{equation}
\end{corollary}
\begin{proof}
    We apply Lemma~\ref{grwonwall type lemma} for the inequality (\ref{second schematic eq}), where:\[u(k,\tau)=2^ka_k(\tau)\mathbf{1}_{k,\tau},\ b(k)=\frac{1}{10}\cdot2^{-8k},\ c(k,\tau)=\tau^{-3}2^{6k}\mathbf{1}_{k,\tau}.\]
    From \eqref{inequality conclusion} we obtain for $\tau\in[X2^{-k-1},1]$:
    \[2^ka_k(\tau)\mathbf{1}_{k,\tau}\leq A_k(\tau)+2^{-8k}\int_{\tau}^1\sum_{l=x}^{k-1}A_l(\tau')\cdot\frac{1}{(\tau')^{3}}2^{6l}\mathbf{1}_{l,\tau'}\prod_{j=l+1}^{k-1}\bigg(1+\int_{\tau}^{\tau'}b(j)c(j,\tau'')d\tau''\bigg)d\tau'.\]
    In order to bound the above, we first note that:
    \[\int_{\tau}^{\tau'}b(j)c(j,\tau'')d\tau''=\frac{1}{10}\cdot2^{-2j}\int_{\tau}^{\tau'}\frac{\mathbf{1}_{j,\tau''}}{(\tau'')^3}d\tau''\leq\min\big(X^{-2},2^{-2j}\tau^{-2}\big)\leq\min\big(1,2^{-2j}\tau^{-2}\big).\]
    In this bound we used the good control of the constant in the definition of $b(k),$ obtained using the smallness of $\delta$ in \eqref{fix constant delta 1}. As a result, we obtain for all $\tau\in[X2^{-k-1},1]$:
    \[\prod_{j=l+1}^{k-1}\bigg(1+\int_{\tau}^{\tau'}b(j)c(j,\tau'')d\tau''\bigg)\leq\prod_{j=l+1}^{k-1}\Big(\min\big(2,1+2^{-2j}\tau^{-2}\big)\Big)\lesssim1+2^{-l}\tau^{-1},\]
    where we bounded the terms with $2^{-j}<\tau$ using the inequality $x+1\leq e^x.$
    Finally, we conclude by noticing that:
    \[2^{-8k}2^{6l}\cdot2^{-l}\tau^{-1}\leq2^{-7k}2^{5l}\cdot\frac{2^{-k}}{\tau}\leq2^{-7k}2^{5l}\cdot\frac{2}{X}\leq2^{-7k}2^{5l}.\]
\end{proof}

We use the definition of $A_k(\tau)$ to compute the new error terms obtained on the RHS of (\ref{third schematic eq}). For the data terms we compute:
\[d_k+2^{-7k}\int_{\tau}^1\sum_{l=x}^{k-1}\frac{2^{5l}}{(\tau')^3}d_l\mathbf{1}_{l,\tau'}d\tau'\lesssim_{\delta} d_k+\sum_{l=x}^{k-1}2^{5(l-k)}d_l\lesssim_{\delta} d_k.\]
In this inequality, we used the fact that $\sum_k\sum_l2^{-5|k-l|}d_l\lesssim\mathcal{D}_{II},$ so the term $\sum_l2^{-5|k-l|}d_l$ can be written schematically as $d_k.$ We will use similar bounds implicitly for the rest of the proof.

Additionally, we have the following bound for the discrete error terms $S_k:$
\[S_k(\tau)+2^{-7k}\int_{\tau}^1\sum_{l=x}^{k-1}\frac{2^{5l}}{(\tau')^3}S_l(\tau')\mathbf{1}_{l,\tau'}d\tau'\lesssim S_k(\tau)+\delta X2^{-7k}\int_{\tau}^1\sum_{l=x}^{k-1}\sum_{\tau'<X2^{-m-1}\leq1}\frac{2^{-3l+6m}}{(\tau')^6}a_m(X2^{-m-1})\mathbf{1}_{l,\tau'}d\tau'\]
\[\lesssim S_k(\tau)+\delta X2^{-7k}\sum_{l=x}^{k-1}\sum_{\substack{\tau<X2^{-m-1} \\ x\leq m\leq l}}\int_{X2^{-l-1}}^{X2^{-m-1}}\frac{2^{-3l+6m}}{(\tau')^6}a_m(X2^{-m-1})d\tau'\]\[\lesssim S_k(\tau)+\frac{\delta}{X^4}2^{-7k}\sum_{l=x}^{k-1}\sum_{\substack{\tau<X2^{-m-1} \\ x\leq m\leq l}}2^{2l+6m}a_m(X2^{-m-1})\lesssim\frac{\delta}{X^2}2^{-5k}\sum_{\tau<X2^{-m-1}\leq1}2^{6m}a_m(X2^{-m-1}).\]

\paragraph{Discrete Gronwall inequality.} We deal with the sum of discrete error terms obtained in \eqref{third schematic eq} using the discrete Gronwall inequality. We first write the estimate in a convenient form by defining:
\begin{align*}
    \widetilde{S}_k(\tau)&:=\frac{\delta}{X^2}2^{-5k}\sum_{\tau<X2^{-m-1}\leq1}2^{6m}a_m(X2^{-m-1}),\\
    \widetilde{E}_k(\tau)&:= E_k(\tau)+2^{-7k}\int_{\tau}^1\sum_{l=x}^{k-1}\frac{2^{5l}}{(\tau')^3}E_l(\tau')\mathbf{1}_{l,\tau'}d\tau',\\
    \widetilde{I}_k(\tau)&:=I_k(\tau)+2^{-7k}\int_{\tau}^1\sum_{l=x}^{k-1}\frac{2^{5l}}{(\tau')^3}I_l(\tau')\mathbf{1}_{l,\tau'}d\tau'.
\end{align*}
Thus, (\ref{third schematic eq}) implies that for all $\tau\in[X2^{-k-1},1]$:
\begin{equation}\label{some est needed later}
    2^ka_k(\tau)\mathbf{1}_{k,\tau}\lesssim C_{\delta}d_k+\widetilde{S}_k(\tau)+C_{\delta}\widetilde{E}_k(\tau)+C_{\delta}\widetilde{I}_k(\tau).
\end{equation}
In particular, there exist constants $C''',C_{\delta}>0$ such that for any $k\geq x:$
\[2^ka_k(X2^{-k-1})\leq C_{\delta}d_k+C_{\delta}\widetilde{E}_k(X2^{-k-1})+C_{\delta}\widetilde{I}_k(X2^{-k-1})+C'''\frac{\delta}{X^2}2^{-5k}\sum_{m=x}^{k-1}2^{6m}a_m(X2^{-m-1}).\]
We fix the parameters $X,\delta>0$ satisfying (\ref{fix constant X}), (\ref{fix constant delta 1}), and :
\begin{equation}\label{fix constant delta 2}
    C'''\frac{\delta}{X^2}<\frac{1}{10}.
\end{equation}
We also introduce the notation:
\[b_k=C_{\delta}d_k+C_{\delta}\widetilde{E}_k(X2^{-k-1})+C_{\delta}\widetilde{I}_k(X2^{-k-1}).\]
Using this notation, we have for any $k\geq x:$
\[2^ka_k(X2^{-k-1})\leq b_k+\frac{1}{10}2^{-5k}\sum_{m=x}^{k-1}2^{5m}\cdot 2^ma_m(X2^{-m-1}).\]
We apply the discrete Gronwall inequality according to \cite{gronwall}:
\[2^ka_k(X2^{-k-1})\leq b_k+\frac{1}{10}2^{-5k}\sum_{m=x}^{k-1}\bigg(2^{5m}b_m\cdot\prod_{j=m+1}^{k-1}\big(1+1/10\cdot2^{-5j}\cdot2^{5j}\big)\bigg)\lesssim b_k+2^{-4k}\sum_{m=x}^{k-1}2^{4m}b_m.\]
As a consequence, we have the following bound for all $k\geq x$:
\begin{equation}\label{bound for discrete terms high}
    2^ka_k(X2^{-k-1})\lesssim_{\delta} d_k+\widetilde{E}_k(X2^{-k-1})+\widetilde{I}_k(X2^{-k-1})+\sum_{m=x}^{k-1}2^{-4k+4m}\Big(\widetilde{E}_m(X2^{-m-1})+\widetilde{I}_m(X2^{-m-1})\Big).
\end{equation}

\textbf{Notation.} We used the parameters $X>0$ and $\delta>0$ to apply Gronwall-like inequalities in the above proof. We now fixed these parameters (depending on $M,C_0,C_2$), so we return to our usual notation convention that we write $A\lesssim B$ if there exists a constant $C>0$ depending only on the constants $M,C_0,C_2$ such that $A\leq CB.$

\paragraph{Consequences of the discrete Gronwall inequality.} We use the estimate (\ref{bound for discrete terms high}) in the high frequency regime estimate \eqref{some est needed later}. From (\ref{bound for discrete terms high}) we obtain for all $\tau\in[X2^{-k-1},1]$:
\[\widetilde{S}_k(\tau)\lesssim\sum_{\tau<X2^{-m-1}\leq1}2^{-5k+5m}d_m+\sum_{\tau<X2^{-m-1}\leq1}2^{-4k+4m}\Big(\widetilde{E}_m(X2^{-m-1})+\widetilde{I}_m(X2^{-m-1})\Big)\]
\[\lesssim\sum_{\tau<X2^{-m-1}\leq1}2^{-5k+5m}d_m+\sum_{\tau<X2^{-m-1}\leq1}2^{-4k+4m}\Big(E_m(X2^{-m-1})+I_m(X2^{-m-1})\Big),\]
where the second bound follows by:
\[\sum_{\tau<X2^{-m-1}\leq1}2^{-4k+4m}\cdot2^{-7m}\sum_{l=x}^{m-1}\int_{X2^{-l-1}}^1\frac{2^{5l}}{(\tau')^3}E_l(\tau')d\tau'\lesssim\sum_{\tau<X2^{-l-1}\leq1}2^{-4k+4l}E_l(X2^{-l-1}),\]
and we have a similar inequality for the inhomogeneous terms $I_l.$ We also bound the remaining terms of \eqref{some est needed later}:
\[\widetilde{E}_k(\tau)=E_k(\tau)+2^{-7k}\int_{\tau}^1\sum_{l=x}^{k-1}\frac{2^{5l}}{(\tau')^3}E_l(\tau')\mathbf{1}_{l,\tau'}d\tau'\lesssim E_k(\tau)+2^{-2k}\tau^{-2}\sum_{l\geq x}\int_{\tau}^1e_l(\tau')d\tau'\]
\[\lesssim E_k(\tau)+2^{-2k}\tau^{-2}\int_{\tau}^1(\tau')^3\Big(\big\|\xi\big\|_{H^{3/2}}^2+\big\|\nabla_{\tau}\xi\big\|_{H^{1/2}}^2\Big)d\tau',\]
\[\widetilde{I}_k(\tau)=I_k(\tau)+2^{-7k}\int_{\tau}^1\sum_{l=x}^{k-1}\frac{2^{5l}}{(\tau')^3}I_l(\tau')\mathbf{1}_{l,\tau'}d\tau'\lesssim I_k(\tau)+2^{-2k}\tau^{-2}\int_{\tau}^1\tau'\big\|F_M'\big\|_{H^{1/2}}^2d\tau'.\]
We combine these estimates and we obtain that (\ref{some est needed later}) implies for all $\tau\in[X2^{-k-1},1]$:
\[2^ka_k(\tau)\mathbf{1}_{k,\tau}\lesssim d_k+\sum_{\tau<X2^{-m-1}\leq1}2^{-5k+5m}d_m+\sum_{\tau<X2^{-m-1}\leq1}2^{-4k+4m}\Big(E_m(X2^{-m-1})+I_m(X2^{-m-1})\Big)\]
\[+E_k(\tau)+I_k(\tau)+2^{-2k}\tau^{-2}\int_{\tau}^1(\tau')^3\Big(\big\|\xi\big\|_{H^{3/2}}^2+\big\|\nabla_{\tau}\xi\big\|_{H^{1/2}}^2\Big)d\tau'+2^{-2k}\tau^{-2}\int_{\tau}^1\tau'\big\|F_M'\big\|_{H^{1/2}}^2d\tau'.\]

\paragraph{Summing the high frequency estimates.} For each $\tau\in(0,1],$ we sum the above high frequency regime estimates for all $k\geq x$ such that $X2^{-k-1}\leq\tau:$
\[\sum_{\tau\geq X2^{-k-1}}2^ka_k(\tau)\lesssim\mathcal{D}_{II}+\sum_{\tau\geq X2^{-k-1}}\Big(E_k(\tau)+I_k(\tau)\Big)+\sum_{\tau<X2^{-m-1}\leq1}\Big(E_m(X2^{-m-1})+I_m(X2^{-m-1})\Big)+\]\[+\int_{\tau}^1(\tau')^3\Big(\big\|\xi\big\|_{H^{3/2}}^2+\big\|\nabla_{\tau}\xi\big\|_{H^{1/2}}^2\Big)d\tau'+\int_{\tau}^1\tau'\big\|F_M'\big\|_{H^{1/2}}^2d\tau'.\]
Additionally, we have the estimates:
\begin{align*}
    \sum_{\tau\geq X2^{-k-1}}E_k(\tau)+\sum_{\tau<X2^{-m-1}\leq1}E_m(X2^{-m-1})&\lesssim\sum_{k\geq x}\int_{\tau}^1e_k(\tau')\lesssim\int_{\tau}^1(\tau')^3\Big(\big\|\xi\big\|_{H^{3/2}}^2+\big\|\nabla_{\tau}\xi\big\|_{H^{1/2}}^2\Big)d\tau',\\
    \sum_{\tau\geq X2^{-k-1}}I_k(\tau)+\sum_{\tau<X2^{-m-1}\leq1}I_m(X2^{-m-1})&\lesssim\int_{\tau}^1\tau'\big\|F_M'\big\|_{H^{1/2}}^2d\tau'.
\end{align*}
Therefore, we obtain the following high frequency regime estimate:
\begin{equation}\label{high frequency prelim imp est}
    \sum_{\tau\geq X2^{-k-1}}2^ka_k(\tau)\lesssim\mathcal{D}_{II}+\int_{\tau}^1(\tau')^3\Big(\big\|\xi\big\|_{H^{3/2}}^2+\big\|\nabla_{\tau}\xi\big\|_{H^{1/2}}^2\Big)d\tau'+\int_{\tau}^1\tau'\big\|F_M'\big\|_{H^{1/2}}^2d\tau'.
\end{equation}

\paragraph{Proof of Proposition~\ref{improved high frequency estimate proposition}.}
The final step in the proof of the optimal high frequency regime estimate \eqref{optimal high frequency estimate} consists of bounding the second error term on the RHS of \eqref{high frequency prelim imp est}.

We first notice that we have the estimate for all $\tau\in(0,1]$:
\[\tau^3\Big(\big\|\xi\big\|_{H^{3/2}}^2+\big\|\nabla_{\tau}\xi\big\|_{H^{1/2}}^2\Big)\lesssim\tau^3\Big(\big\|\xi\big\|_{H^{1}}^2+\big\|\nabla_{\tau}\xi\big\|_{L^2}^2\Big)+\sum_{k\geq x}\tau^32^k\Big(\big\|\nabla P_k\xi\big\|_{L^2}^2+\big\|P_k\nabla_{\tau}\xi\big\|_{L^2}^2\Big)\]\[\lesssim\tau^3\Big(\big\|\xi\big\|_{H^{1}}^2+\big\|\nabla_{\tau}\xi\big\|_{L^2}^2\Big)+\sum_{\tau<X2^{-k-1}\leq1}\tau^32^k\Big(\big\|\nabla P_k\xi\big\|_{L^2}^2+\big\|P_k\nabla_{\tau}\xi\big\|_{L^2}^2\Big)+\tau^2\sum_{\tau\geq X2^{-k-1}}2^ka_k(\tau)\]\[\lesssim\tau^3\Big(\big\|\xi\big\|_{H^{1}}^2+\big\|\nabla_{\tau}\xi\big\|_{L^2}^2\Big)+\sum_{\tau<X2^{-k-1}\leq1}\tau^2\Big(\big\|\nabla P_k\xi\big\|_{L^2}^2+\big\|P_k\nabla_{\tau}\xi\big\|_{L^2}^2\Big)+\tau^2\sum_{\tau\geq X2^{-k-1}}2^ka_k(\tau)\]\[\lesssim\tau^2\Big(\big\|\xi\big\|_{H^{1}}^2+\big\|\nabla_{\tau}\xi\big\|_{L^2}^2\Big)+\tau^2\sum_{\tau\geq X2^{-k-1}}2^ka_k(\tau).\]
We use the preliminary estimates in Proposition~\ref{backward direction basic estimate proposition} and (\ref{high frequency prelim imp est}) to get:
\[\tau^3\Big(\big\|\xi\big\|_{H^{3/2}}^2+\big\|\nabla_{\tau}\xi\big\|_{H^{1/2}}^2\Big)\lesssim \mathcal{D}_{II}+\int_{\tau}^1(\tau')^3\Big(\big\|\xi\big\|_{H^{3/2}}^2+\big\|\nabla_{\tau}\xi\big\|_{H^{1/2}}^2\Big)d\tau'+\int_{\tau}^1\tau'\big\|F_M'\big\|_{H^{1/2}}^2d\tau'.\]
Using Gronwall, we proved that for all $\tau\in(0,1]$:
\begin{equation}\label{nonoptimal estimate}
    \tau^3\Big(\big\|\xi\big\|_{H^{3/2}}^2+\big\|\nabla_{\tau}\xi\big\|_{H^{1/2}}^2\Big)\lesssim \mathcal{D}_{II}+\int_{\tau}^1\tau'\big\|F_M'\big\|_{H^{1/2}}^2d\tau'.
\end{equation}
Finally, we use this estimate in (\ref{high frequency prelim imp est}) to obtain the optimal high frequency estimate \eqref{optimal high frequency estimate}.

\subsection{The main estimate in Theorem \ref{main theorem second system}}\label{main estimate back direction proof section}

In this section, we prove the main estimate \eqref{main estimate second model system} in Theorem \ref{main theorem second system}. For this, we first establish the following top order estimate for the singular component $\Phi_0:$
\begin{equation}\label{claim for top order singular}
    \mathcal{E}^0_{II}(\tau)\lesssim \mathcal{D}_{II}+\sum_{m=0}^M\int_{\tau}^1\tau'\big\|F_m'\big\|_{H^{1/2}}^2d\tau',
\end{equation}
where we define the top order energy for the singular component:
\[\mathcal{E}^0_{II}(\tau):=\tau\big\|\Phi_0\big\|_{H^{M+1/2}}^2+\tau^2\big\|\Phi_0\big\|_{H^{M+3/2}}^2+\tau^2\big\|\nabla_{\tau}\nabla^M\Phi_0\big\|_{H^{1/2}}^2+\sum_{m=0}^{M-1}\tau^2\big\|\nabla_{\tau}\nabla^m\Phi_0\big\|_{L^2}^2+\int_{\tau}^1\tau'\big\|\Phi_0\big\|_{H^{M+1}}^2d\tau'.\]

Using Propositions~\ref{backward direction basic estimate proposition}, \ref{prelim low freq singular proposition} and the estimate \eqref{optimal high frequency estimate}, we have the bound for all $\tau\in(0,1]$:
\[\tau^2\Big(\big\|\xi\big\|_{H^{3/2}}^2+\big\|\nabla_{\tau}\xi\big\|_{H^{1/2}}^2\Big)\lesssim\tau^2\Big(\big\|\xi\big\|_{H^{1}}^2+\big\|\nabla_{\tau}\xi\big\|_{L^2}^2\Big)+\sum_{k\geq x}\tau^22^k\Big(\big\|\nabla P_k\xi\big\|_{L^2}^2+\big\|P_k\nabla_{\tau}\xi\big\|_{L^2}^2\Big)\]\[\lesssim\tau^2\Big(\big\|\xi\big\|_{H^{1}}^2+\big\|\nabla_{\tau}\xi\big\|_{L^2}^2\Big)+\sum_{\tau<X2^{-k-1}\leq1}\tau^22^{k}\Big(\big\|\nabla P_k\xi\big\|_{L^2}^2+\big\|P_k\nabla_{\tau}\xi\big\|_{L^2}^2\Big)+\tau\sum_{\tau\geq X2^{-k-1}}2^ka_k(\tau)\]\[\lesssim\mathcal{D}_{II}+\int_{\tau}^1\tau'\big\|F_M'\big\|_{H^{1/2}}^2d\tau'+\sum_{\tau<X2^{-k-1}\leq1}2^ka_k(X2^{-k-1}).\]
We use the bound (\ref{bound for discrete terms high}) for $2^ka_k(X2^{-k-1})$ to get:
\[\sum_{\tau<X2^{-k-1}\leq1}2^ka_k(X2^{-k-1})\lesssim\mathcal{D}_{II}+\sum_{\tau<X2^{-k-1}\leq1}\widetilde{E}_k(X2^{-k-1})+\sum_{\tau<X2^{-k-1}\leq1}\widetilde{I}_k(X2^{-k-1})\]\[\lesssim\mathcal{D}_{II}+\sum_{\tau<X2^{-k-1}\leq1}E_k(X2^{-k-1})+\sum_{\tau<X2^{-k-1}\leq1}I_k(X2^{-k-1})\lesssim \mathcal{D}_{II}+\int_{\tau}^1\tau'\big\|F_M'\big\|_{H^{1/2}}^2d\tau'.\]
We note that we used (\ref{nonoptimal estimate}) in the last inequality. As a result, we proved that:
\[\tau^2\Big(\big\|\xi\big\|_{H^{3/2}}^2+\big\|\nabla_{\tau}\xi\big\|_{H^{1/2}}^2\Big)\lesssim \mathcal{D}_{II}+\int_{\tau}^1\tau'\big\|F_M'\big\|_{H^{1/2}}^2d\tau'.\]
As usual, we also get for all $\tau\in(0,1]$:
\begin{align*}
    \tau\big\|\xi\big\|_{H^{1/2}}^2&\lesssim \mathcal{D}_{II}+\int_{\tau}^1\tau'\big\|\xi\big\|_{H^{1/2}}\cdot\big\|\nabla_{\tau}\xi\big\|_{H^{1/2}}d\tau'\lesssim \mathcal{D}_{II}+\int_{\tau}^1\sqrt{\tau'}\big\|\xi\big\|_{H^{1/2}}^2d\tau'+\int_{\tau}^1(\tau')^{3/2}\big\|\nabla_{\tau}\xi\big\|_{H^{1/2}}^2d\tau'\\
    &\lesssim \mathcal{D}_{II}+\int_{\tau}^1\tau'\big\|F_M'\big\|_{H^{1/2}}^2d\tau'+\int_{\tau}^1\sqrt{\tau'}\big\|\xi\big\|_{H^{1/2}}^2d\tau'.
\end{align*}
Applying Gronwall, we proved that:
\[\tau^2\big\|\xi\big\|_{H^{3/2}}^2+\tau^2\big\|\nabla_{\tau}\xi\big\|_{H^{1/2}}^2+\tau\big\|\xi\big\|_{H^{1/2}}^2\lesssim \mathcal{D}_{II}+\int_{\tau}^1\tau'\big\|F_M'\big\|_{H^{1/2}}^2d\tau'.\]
We recall that $\xi=\nabla^M\Phi_0.$ Combining the above estimate with the preliminary estimates in Proposition~\ref{backward direction basic estimate proposition}, we proved \eqref{claim for top order singular}.

\begin{proof}[Proof of \eqref{main estimate second model system}]
    We use the estimates for the regular quantities in Section~\ref{good quantities estimates subsection} and the top order estimate for the singular quantities \eqref{claim for top order singular} to get:
    \[\mathcal{E}_{II}(\tau)\lesssim\mathcal{D}_{II}+\mathcal{F}_{II}(\tau)+\sum_{m=0}^M\int_{\tau}^1\tau'\big\|F_m'\big\|_{H^{1/2}}^2d\tau'.\]
    We also recall our notation $F_m'=\psi\sum_{i=1}^I\nabla^{m+1}\Phi_i+F_{m}^{0}.$ Using again the estimates for the regular quantities proved in Section~\ref{good quantities estimates subsection}, we conclude.
\end{proof}

\subsection{Estimates for the asymptotic quantities}\label{asympt data estimates section}
In this section we prove the estimates \eqref{main estimate asymptotic data} and \eqref{main estimate asymptotic data h} for the asymptotic quantities $\mathcal{O}$, $\mathfrak{h},$ and $\Phi_i^0$ with $1\leq i\leq I,$ in order to complete the proof of Theorem~\ref{main theorem second system}. We notice that the estimates for the regular quantities in Proposition~\ref{good quantities estimates proposition} already imply the bound \eqref{good quantities asymptotic data intro} for $\Phi_i^0$ with $1\leq i\leq I.$ As a result, we only need to prove the estimates \eqref{estimate for O intro} and \eqref{main estimate asymptotic data h} for $\mathcal{O}$ and $\mathfrak{h}$.
\begin{proposition}\label{estimate for O proposition}
    The obstruction tensor $\mathcal{O}$ satisfies the estimate \eqref{estimate for O intro}:
    \[\big\|\mathcal{O}\big\|_{H^{M+1}}^2\lesssim \mathcal{D}_{II}+\mathcal{F}_{II}(0).\]
\end{proposition}
\begin{proof}
    As a consequence of the preliminary estimates in Proposition \ref{backward direction basic estimate proposition}, we obtain for all $0\leq m\leq M$:
    \[\big\|\nabla^m\mathcal{O}\big\|_{L^2}^2\lesssim \bigg(\big\|\Phi_0\big\|_{H^{m+3/2}}^2+\big\|\nabla_{\tau}\Phi_0\big\|_{H^{m+1/2}}^2\bigg)\bigg|_{\tau=1}+\int_{0}^1(\tau')^2\big\|F_m'\big\|_{L^2}^2d\tau'\lesssim \mathcal{D}_{II}+\mathcal{F}_{II}(0),\]
    so it remains to prove the above bound for $\big\|\nabla^{M+1}\mathcal{O}\big\|_{L^2}^2$. 
    
    For any $k\geq x$ and $0\leq\tau<X2^{-k-1}\leq1$, the low frequency regime estimate in Proposition~\ref{prelim low freq singular proposition} implies:
    \[2^{2k}\big\|\tau P_k\nabla_{\tau}\nabla^M\Phi_0\big\|_{L^2}^2\lesssim2^ka_k(X2^{-k-1})+2^{-k}\mathcal{D}_{II}+2^{k}\int_{\tau}^{1}(\tau')^2\big\|P_kF_M'\big\|_{L^2}^2d\tau'+2^{-k}\int_{\tau}^{1}(\tau')^2\big\|F_M'\big\|_{L^2}^2d\tau'.\]
    Using the expansion of $\nabla_{\tau}\nabla^M\Phi_0$ at $\mathcal{I}^-$, we get that for all $k\geq x:$
    \[2^{2k}\big\|P_k\nabla^M\mathcal{O}\big\|_{L^2}^2\lesssim2^ka_k(X2^{-k-1})+2^{-k}\mathcal{D}_{II}+2^{k}\int_{0}^{1}(\tau')^2\big\|P_kF_M'\big\|_{L^2}^2d\tau'+2^{-k}\int_{0}^{1}(\tau')^2\big\|F_M'\big\|_{L^2}^2d\tau'.\]
    Together with our previous estimate for $\|\mathcal{O}\|_{H^M},$ we obtain that:
    \[\big\|\nabla^{M}\mathcal{O}\big\|_{H^1}^2\lesssim\mathcal{D}_{II}+\mathcal{F}_{II}(0)+\sum_{k\geq x}2^ka_k(X2^{-k-1}).\]
    We complete the proof since we already proved in Section~\ref{main estimate back direction proof section}:
    \[\sum_{k\geq x}2^ka_k(X2^{-k-1})\lesssim \mathcal{D}_{II}+\int_{0}^1\tau'\big\|F_M'\big\|_{H^{1/2}}^2d\tau'\lesssim\mathcal{D}_{II}+\mathcal{F}_{II}(0).\]
\end{proof}

In order to prove the estimate \eqref{main estimate asymptotic data h} for $\mathfrak{h}$, we first prove that it can be reduced to the proof of \eqref{estimate for h intro}. We recall the notation $\mathfrak{h}_M=\nabla^Mh-2(\log\nabla)\nabla^M\mathcal{O},$ and notice that we have the bounds:
\begin{align*}
    \big\|\mathfrak{h}\big\|_{H^{M+1}}^2&\lesssim\big\|\mathfrak{h}\big\|_{L^2}^2+\big\|\big[(\log\nabla),\nabla^M\big]\mathcal{O}\big\|_{H^{1}}^2+\big\|\nabla\mathfrak{h}_M\big\|_{L^2}^2\\
    &\lesssim\big\|\mathfrak{h}\big\|_{L^2}^2+C\Big(\big\|\slashed{Riem}_0\big\|_{H^{M}}^2\Big)\big\|\mathcal{O}\big\|_{H^{M}}^2+\sum_{k\geq x}2^{2k}\big\|P_k\mathfrak{h}_M\big\|_{L^2}^2+\big\|\mathfrak{h}_M\big\|_{L^2}^2\\
    &\lesssim\big\|\mathfrak{h}\big\|_{H^M}^2+C\Big(\big\|\slashed{Riem}_0\big\|_{H^{M}}^2\Big)\big\|\mathcal{O}\big\|_{H^{M}}^2+\sum_{k\geq x}2^{2k}\big\|P_k\mathfrak{h}_M\big\|_{L^2}^2,
\end{align*}
where we used  the proof of Lemma~\ref{commute with log lemma} to bound the commutator term. Using the interpolation inequalities of \cite{geometricLP} and the bound in Lemma~\ref{bound on log nabla lemma}, we obtain:
\[\big\|\mathfrak{h}\big\|_{H^{M+1}}^2\lesssim\big\|h\big\|_{L^2}^2+C\Big(\big\|\slashed{Riem}_0\big\|_{H^{M}}^2\Big)\big\|\mathcal{O}\big\|_{H^{M}}^2+\sum_{k\geq x}2^{2k}\big\|P_k\mathfrak{h}_M\big\|_{L^2}^2.\]
Using the already established bound for $\mathcal{O},$ we showed that it suffices to prove \eqref{estimate for h intro}:
\[\sum_{k\geq x}2^{2k}\big\|P_k\mathfrak{h}_M\big\|_{L^2}^2\lesssim\mathcal{D}_{II}+\mathcal{F}_{II}(0).\]

We recall our notation $\xi=\nabla^M\Phi_0$ and we consider the renormalized quantities for all $k\geq x$:
\[\overline{\xi}_k=P_k\xi-\tau\log(2^k\tau)P_k\nabla_{\tau}\xi.\]
Using the expansion for $\xi$ at $\mathcal{I}^-,$ we get:
\[\lim_{\tau\rightarrow0}\overline{\xi}_k=\lim_{\tau\rightarrow0}\Big(P_k\nabla^M\Phi_0-\tau\log(2^k\tau)P_k\nabla_{\tau}\nabla^M\Phi_0\Big)=P_k\nabla^Mh-2\log(2^k)P_k\nabla^M\mathcal{O}.\]
Recalling the definition of the operator $R_k$ in \eqref{definition of R k}, we can write:
\[P_k\nabla^Mh-2\log(2^k)P_k\nabla^M\mathcal{O}=P_k\mathfrak{h}_M+2P_k(\log\nabla)\nabla^M\mathcal{O}-2\log(2^k)P_k\nabla^M\mathcal{O}=P_k\mathfrak{h}_M+R_k\nabla^M\mathcal{O}.\]
Thus, we obtain using (\ref{Rk est 3}):
\[2^{2k}\big\|P_k\mathfrak{h}_M\big\|_{L^2}^2\lesssim2^{2k}\lim_{\tau\rightarrow0}\big\|\overline{\xi}_k\big\|_{L^2}^2+2^{2k}\big\|R_k\nabla^M\mathcal{O}\big\|_{L^2}^2\lesssim2^{2k}\lim_{\tau\rightarrow0}\big\|\overline{\xi}_k\big\|_{L^2}^2+\big\|\underline{P}_k\nabla^M\mathcal{O}\big\|_{H^1}^2.\]
Summing for all $k\geq x$, we get that:
\[\sum_{k\geq x}2^{2k}\big\|P_k\mathfrak{h}_M\big\|_{L^2}^2\lesssim\big\|\mathcal{O}\big\|_{H^{M+1}}^2+\sum_{k\geq x}2^{2k}\lim_{\tau\rightarrow0}\big\|\overline{\xi}_k\big\|_{L^2}^2.\]

According to the above estimate and \eqref{estimate for O intro}, in order to prove \eqref{estimate for h intro} we need to establish the following result:

\begin{proposition}\label{estimate for PkhM proposition}
    We have the following estimate:
    \[\sum_{k\geq x}2^{2k}\lim_{\tau\rightarrow0}\big\|\overline{\xi}_k\big\|_{L^2}^2\lesssim\mathcal{D}_{II}+\mathcal{F}_{II}(0).\]
    As a result, \eqref{main estimate asymptotic data h} holds, completing the proof of Theorem~\ref{main theorem second system}.
\end{proposition}
\begin{proof}
Using \eqref{back main lin wave equation} as in the proof of Proposition~\ref{prelim low freq singular proposition}, we get that $\overline{\xi}_k$ satisfies the equation on $\tau\in(0,X2^{-k-1}]$:
\[\nabla_{\tau}\overline{\xi}_k=[\nabla_{\tau},P_k]\xi-\log(2^k\tau)\nabla_{\tau}\big(\tau P_k\nabla_{\tau}\xi\big)=\]\[=\tau[\nabla_4,P_k]\xi-\log(2^k\tau)\Big(4\tau\Delta P_k\xi+\tau P_k\big(\psi\nabla\xi\big)+\tau[\nabla_{4},P_k]\tau\nabla_{\tau}\xi+\tau P_kF_M'\Big).\]

We contract by $\overline{\xi}_k$ to obtain the energy estimate:
\[2^{2k}\big\|\overline{\xi}_k\big\|_{L^2}^2\lesssim\Big(2^{2k}\big\|P_k\xi\big\|_{L^2}^2+\big\|P_k\nabla_{\tau}\xi\big\|_{L^2}^2\Big)\Big|_{\tau={X2^{-k-1}}}+\int_{\tau}^{X2^{-k-1}}2^k\big\|\overline{\xi}_k\big\|_{L^2}\cdot2^k\tau'\big\|[\nabla_4,P_k]\xi\big\|_{L^2}+\]\[+\int_{\tau}^{X2^{-k-1}}2^k\big\|\overline{\xi}_k\big\|_{L^2}\cdot|2^k\tau'\log(2^k\tau')|\cdot\big\|\Delta P_k\xi\big\|_{L^2}+\int_{\tau}^{X2^{-k-1}}2^k\big\|\overline{\xi}_k\big\|_{L^2}\cdot|2^k\tau'\log(2^k\tau')|\cdot\big\|P_k\big(\psi\nabla\xi\big)\big\|_{L^2}+\]\[+\int_{\tau}^{X2^{-k-1}}2^k\big\|\overline{\xi}_k\big\|_{L^2}\cdot|2^k\tau'\log(2^k\tau')|\cdot\big\|[\nabla_{4},P_k]\tau'\nabla_{\tau}\xi\big\|_{L^2}+\int_{\tau}^{X2^{-k-1}}2^k\big\|\overline{\xi}_k\big\|_{L^2}\cdot|2^k\tau'\log(2^k\tau')|\cdot\big\|P_kF_M'\big\|_{L^2}.\]
We use Lemma~\ref{Litt Paley lemma}, Gronwall, and the notation in Section~\ref{improved high frequency estimates section} to obtain:
\begin{align*}
    2^{2k}\big\|\overline{\xi}_k\big\|_{L^2}^2\lesssim &2^ka_{k}(X2^{-k-1})+\int_{\tau}^{X2^{-k-1}}2^k(\tau')^2\big\|\xi\big\|_{L^2}^2+\int_{\tau}^{X2^{-k-1}}2^k(\tau')^2|\log(2^k\tau')|^2\big\|\Delta P_k\xi\big\|_{L^2}^2\\
    &+\int_{\tau}^{X2^{-k-1}}2^{-k}(\tau')^2|\log(2^k\tau')|^2\big\|\xi\big\|_{H^1}^2+\int_{\tau}^{X2^{-k-1}}2^k(\tau')^2|\log(2^k\tau')|^2\big\|\nabla P_k\xi\big\|_{L^2}^2\\
    &+\int_{\tau}^{X2^{-k-1}}2^k(\tau')^2|\log(2^k\tau')|^2\big\|\tau'\nabla_{\tau}\xi\big\|_{L^2}^2+\int_{\tau}^{X2^{-k-1}}2^k(\tau')^2|\log(2^k\tau')|^2\big\|P_kF_M'\big\|_{L^2}^2.
\end{align*}
We use the preliminary estimates in Proposition \ref{backward direction basic estimate proposition} to get:
\[2^{2k}\big\|\overline{\xi}_k\big\|_{L^2}^2\lesssim 2^{-k}\mathcal{D}_{II}+2^ka_{k}(X2^{-k-1})+\]\[+\int_{\tau}^{X2^{-k-1}}2^k(\tau')^2|\log(2^k\tau')|^2\big\|\Delta P_k\xi\big\|_{L^2}^2d\tau'+2^{-k}\int_{\tau}^{1}(\tau')^2\big\|F_M'\big\|_{L^2}^2d\tau'+\int_{\tau}^{1}\tau'\big\|P_kF_M'\big\|_{L^2}^2d\tau'.\]
Using the finite band property of the LP projections for $P_k=\underline{P}_k^2$, we have the bound:
\begin{equation}\label{est that gives 4 error terms}
    \int_{\tau}^{X2^{-k-1}}2^k(\tau')^2|\log(2^k\tau')|^2\big\|\Delta P_k\xi\big\|_{L^2}^2d\tau'\lesssim\int_{\tau}^{X2^{-k-1}}2^{3k}(2^k\tau')^2|\log(2^k\tau')|^2\big\|\underline{P}_k\xi\big\|_{L^2}^2d\tau'.
\end{equation}

Once we established the above estimate for $2^{k}\|\overline{\xi}_k\|_{L^2},$ the idea of the proof is to decompose the error term in \eqref{est that gives 4 error terms} into its low frequency and high frequency regime parts, in order to use our previous estimates. For each $\tau'\in[\tau,X2^{-k-1}]$ we have:
\begin{equation}\label{decomposition into error terms}
    \big\|\underline{P}_k\xi\big\|_{L^2}(\tau')\lesssim\sum_{l<x}\big\|P_l^2\underline{P}_k\xi\big\|_{L^2}+\sum_{l=x}^{k-1}\big\|P_l^2\underline{P}_k\xi\big\|_{L^2}+\sum_{\substack{l\geq k \\ \tau'<X2^{-l-1}}}\big\|P_l^2\underline{P}_k\xi\big\|_{L^2}+\sum_{l\geq k}\mathbf{1}_{l,\tau'}\big\|P_l^2\underline{P}_k\xi\big\|_{L^2}.
\end{equation}
The first three terms in \eqref{decomposition into error terms} are in the low frequency regime, and the last term is in the high frequency regime. 

We bound the \textit{first term} in \eqref{decomposition into error terms} using the $L^2$ almost orthogonality of the LP projections:
\[\sum_{l<x}\big\|P_l^2\underline{P}_k\xi\big\|_{L^2}(\tau')\lesssim\sum_{l<x}2^{-5k+5l}\big\|P_l\xi\big\|_{L^2}(\tau')\lesssim2^{-5k}\big\|\xi\big\|_{L^2}(\tau').\]
The corresponding term in (\ref{est that gives 4 error terms}) is bounded using Proposition \ref{backward direction basic estimate proposition}:
\[\int_{\tau}^{X2^{-k-1}}2^{-6k}(2^k\tau')|\log(2^k\tau')|^2\cdot\tau'\big\|\xi\big\|_{L^2}^2d\tau'\lesssim2^{-k}\mathcal{D}_{II}+2^{-k}\mathcal{F}_{II}(\tau).\]

For the \textit{second term} in \eqref{decomposition into error terms} we use Cauchy-Schwarz and we consider a projection operator with $\underline{P}_k=\underline{\widetilde{P}}_k^2$:
\begin{align*}
    \bigg(\sum_{l=x}^{k-1}\big\|P_l^2\underline{\widetilde{P}}_k^2\xi\big\|_{L^2}(\tau')\bigg)^2&\lesssim\sum_{l=x}^{k-1}2^{-5(k-l)}\big\|\underline{\widetilde{P}}_kP_l\xi\big\|_{L^2}^2(\tau')\lesssim\sum_{l=x}^{k-1}2^{-2k}2^{-5(k-l)}\big\|\nabla P_l\xi\big\|_{L^2}^2(\tau')\\
    &\lesssim\frac{1}{(\tau')^{2}}\sum_{l=x}^{k-1}2^{-2k}2^{-5(k-l)}\big(2^{-2l}\cdot2^la_l(X2^{-l-1})+2^{-3l}\mathcal{D}_{II}+2^{-3l}\mathcal{F}_{II}(\tau')\big),
\end{align*}
where we used Proposition~\ref{prelim low freq singular proposition} in the second line. The corresponding term in (\ref{est that gives 4 error terms}) is bounded by:
\[\int_{\tau}^{X2^{-k-1}}\sum_{l=x}^{k-1}2^{k}|\log(2^k\tau')|^22^{-3(k-l)}\bigg(2^la_l(X2^{-l-1})+2^{-l}\mathcal{D}_{II}+2^{-l}\mathcal{F}_{II}(\tau')\bigg)d\tau'\lesssim\]
\[\lesssim2^{-k}\mathcal{D}_{II}+2^{-k}\mathcal{F}_{II}(\tau)+\sum_{l=x}^{k-1}2^{-3(k-l)}\cdot2^la_l(X2^{-l-1}).\]

For the \textit{third term} in \eqref{decomposition into error terms}, we have similarly from Proposition \ref{prelim low freq singular proposition}:
\[\bigg(\sum_{\substack{l\geq k \\ \tau'<X2^{-l-1}}}\big\|P_l^2\underline{\widetilde{P}}_k^2\xi\big\|_{L^2}(\tau')\bigg)^2\lesssim\sum_{\substack{l\geq k \\ \tau'<X2^{-l-1}}}2^{-5(l-k)}\big\|\underline{\widetilde{P}}_kP_l\xi\big\|_{L^2}^2(\tau')\lesssim\sum_{\substack{l\geq k \\ \tau'<X2^{-l-1}}}2^{-2k}2^{-5(l-k)}\big\|\nabla P_l\xi\big\|_{L^2}^2(\tau')\]
\[\lesssim\frac{1}{(\tau')^{2}}\sum_{\substack{l\geq k \\ \tau'<X2^{-l-1}}}2^{-2k}2^{-5(l-k)}\big(2^{-2l}\cdot2^la_l(X2^{-l-1})+2^{-3l}\mathcal{D}_{II}+2^{-3l}\mathcal{F}_{II}(\tau')\big).\]
As a result, the corresponding term in (\ref{est that gives 4 error terms}) is bounded using Proposition \ref{prelim low freq singular proposition}:
\[\sum_{\substack{l\geq k \\ \tau<X2^{-l-1}}}\int_{\tau}^{X2^{-l-1}}2^{k}|\log(2^k\tau')|^22^{-7(l-k)}\big(2^la_l(X2^{-l-1})+2^{-l}\mathcal{D}_{II}+2^{-l}\mathcal{F}_{II}(\tau')\big)d\tau'\lesssim\]
\[\lesssim\sum_{\substack{l\geq k \\ \tau<X2^{-l-1}}}\big(2^la_l(X2^{-l-1})+2^{-l}\mathcal{D}_{II}+2^{-l}\mathcal{F}_{II}(\tau)\big)\cdot\int_{\tau}^{X2^{-l-1}}2^{k}|\log(2^k\tau')|^22^{-7(l-k)}d\tau'\lesssim\]
\[\lesssim2^{-k}\mathcal{D}_{II}+2^{-k}\mathcal{F}_{II}(\tau)+\sum_{l\geq k}2^{-7(l-k)}\cdot2^la_l(X2^{-l-1}).\]

For the high frequency regime \textit{fourth term} in \eqref{decomposition into error terms}, we have using our notation in Section \ref{improved high frequency estimates section}:
\[\bigg(\sum_{l\geq k}\mathbf{1}_{l,\tau'}\big\|P_l^2\underline{P}_k\xi\big\|_{L^2}(\tau')\bigg)^2\lesssim\sum_{l\geq k}\mathbf{1}_{l,\tau'}2^{-5(l-k)}\big\|P_l\xi\big\|_{L^2}^2(\tau')\lesssim\tau'\sum_{l\geq k}\mathbf{1}_{l,\tau'}2^{-5(l-k)}a_l(\tau').\]
The corresponding term in (\ref{est that gives 4 error terms}) is bounded in Lemma \ref{lemma for the fourth intermediate term} at the end of the section. 

We combine the previous bounds and we get from (\ref{est that gives 4 error terms}) that for all $\tau\in(0,X2^{-k-1}]$:
\[2^{2k}\big\|\overline{\xi}_k\big\|_{L^2}^2\lesssim d_k+\tau^62^{6k}\mathcal{D}_{II}+2^ka_{k}(X2^{-k-1})+\sum_{l\geq x}2^{-3|k-l|}\cdot2^la_l(X2^{-l-1})\]\[+\big(2^{-k}+\tau^62^{6k}\big)\mathcal{F}_{II}(\tau)+\sum_{l\geq k}2^{-6(l-k)}\int_{\tau}^12^l\tau'\big\|P_lF_M'\big\|_{L^2}^2d\tau'+\widetilde{E}_k(X2^{-k-1})+\widetilde{I}_k(X2^{-k-1})+\]
\[+\sum_{\tau<X2^{-m-1}\leq1}2^{-4|k-m|}\Big(E_m(X2^{-m-1})+I_m(X2^{-m-1})\Big)+\sum_{l\geq k}2^{-6(l-k)}\int_{\tau}^1e_l(\tau')d\tau'.\]
We take the limit $\tau\rightarrow0$ and sum for all $k\geq x:$
\begin{align*}
    \sum_{k\geq x}\Big(2^{2k}\lim_{\tau\rightarrow0}\big\|\overline{\xi}_k\big\|_{L^2}^2\Big)&\lesssim\mathcal{D}_{II}+\mathcal{F}_{II}(0)+\sum_{k\geq x}2^ka_{k}(X2^{-k-1})\\
    &+\sum_{k\geq x}\big(\widetilde{E}_k(X2^{-k-1})+\widetilde{I}_k(X2^{-k-1})\big)+\int_0^1(\tau')^3\Big(\big\|\xi\big\|_{H^{3/2}}^2+\big\|\nabla_{\tau}\xi\big\|_{H^{1/2}}^2\Big)d\tau'.
\end{align*}
For the last three terms, we use the estimates in Section~\ref{improved high frequency estimates section} and we obtain the conclusion.
\end{proof}

We conclude the section by proving the additional estimate used in the previous proof:
\begin{lemma}\label{lemma for the fourth intermediate term}
    We have the bound for all $\tau\in(0,X2^{-k-1}]$:
    \[\int_{\tau}^{X2^{-k-1}}2^{k}(2^k\tau')^3|\log(2^k\tau')|^2\sum_{l\geq k}\mathbf{1}_{l,\tau'}2^{-6(l-k)}\cdot2^la_l(\tau')d\tau'\lesssim\]\[\lesssim d_k+\tau^62^{6k}\cdot\big(\mathcal{D}_{II}+\mathcal{F}_{II}(\tau)\big)+\sum_{\tau<X2^{-m-1}\leq1}2^{-4|k-m|}\Big(E_m(X2^{-m-1})+I_m(X2^{-m-1})\Big)+\]\[+\widetilde{E}_k(X2^{-k-1})+\widetilde{I}_k(X2^{-k-1})+\sum_{l\geq k}2^{-6(l-k)}\bigg(\int_{\tau}^1e_l(\tau')d\tau'+\int_{\tau}^12^l\tau'\big\|P_lF_M'\big\|_{L^2}^2d\tau'\bigg)+2^{-k}\mathcal{F}_{II}(\tau).\]
\end{lemma}
\begin{proof}
    We use the estimate (\ref{some est needed later}) as a starting point. We also use the bound on $\widetilde{S}_l$ in Section~\ref{improved high frequency estimates section}, and the definitions of $\widetilde{E}_l$ and $\widetilde{I}_l$. Thus, we have for all $\tau\in[X2^{-l-1},1]$:
    \[2^la_l(\tau)\mathbf{1}_{l,\tau}\lesssim d_l+\widetilde{S}_l(\tau)+\widetilde{E}_l(\tau)+\widetilde{I}_l(\tau)\]
    \[\lesssim d_l+\sum_{\tau<X2^{-m-1}\leq1}2^{-5l+5m}d_m+\sum_{\tau<X2^{-m-1}\leq1}2^{-4l+4m}\Big(E_m(X2^{-m-1})+I_m(X2^{-m-1})\Big)+\]
    \[+E_l(\tau)+I_l(\tau)+2^{-7l}\int_{\tau}^1\sum_{m=x}^{l-1}\frac{2^{5m}}{(\tau')^3}\Big(E_m(\tau')+I_m(\tau')\Big)\mathbf{1}_{m,\tau'}d\tau'.\]
    We use this bound to obtain a total of 8 error terms that control our main integral. We notice that the first two terms can be dealt with in a straightforward way. For the third and fourth terms we compute that:
    \[\int_{\tau}^{X2^{-k-1}}2^{k}\sum_{l\geq k}\sum_{\tau'<X2^{-m-1}\leq1}\mathbf{1}_{l,\tau'}2^{-6(l-k)}2^{-4l+4m}\Big(E_m(X2^{-m-1})+I_m(X2^{-m-1})\Big)d\tau'\lesssim\]
    \[\lesssim\sum_{\tau<X2^{-m-1}\leq1}2^{-4|k-m|}\Big(E_m(X2^{-m-1})+I_m(X2^{-m-1})\Big),\]
    where we used the fact that $l\geq m$ and $l\geq k$ in the first line. For the fifth and sixth terms we have:
    \[\int_{\tau}^{X2^{-k-1}}2^{k}\sum_{l\geq k}\mathbf{1}_{l,\tau'}2^{-6(l-k)}\Big(E_l(\tau')+I_l(\tau')\Big)d\tau'\lesssim\]\[\lesssim\sum_{l\geq k}2^{-6(l-k)}\bigg(\int_{\tau}^1e_l(\tau')d\tau'+\int_{\tau}^12^l\tau'\big\|P_lF_M'\big\|_{L^2}^2d\tau'\bigg)+2^{-k}\int_{\tau}^1\tau'\big\|F_M'\big\|_{H^{1/2}}^2d\tau'.\]
    For the seventh and eight terms we can write:
    \[\int_{\tau}^{X2^{-k-1}}2^{k}\sum_{l\geq k}\mathbf{1}_{l,\tau'}2^{-6(l-k)}\cdot\bigg(2^{-7l}\int_{\tau'}^1\sum_{m=x}^{l-1}\frac{2^{5m}}{(\tau'')^3}\Big(E_m(\tau'')+I_m(\tau'')\Big)\mathbf{1}_{m,\tau''}d\tau''\bigg)d\tau'\lesssim I+II+III,\]
    where we introduce the notation:
    \begin{align*}
        I&=\int_{\tau}^{X2^{-k-1}}2^{k}\sum_{l\geq k}\mathbf{1}_{l,\tau'}2^{-6(l-k)}\cdot2^{-7l}\sum_{m=x}^{k-1}\bigg(\int_{\tau'}^1\frac{2^{5m}}{(\tau'')^3}\Big(E_m(\tau'')+I_m(\tau'')\Big)\mathbf{1}_{m,\tau''}d\tau''\bigg)d\tau',\\
        II&=\int_{\tau}^{X2^{-k-1}}2^{k}\sum_{l\geq k}\mathbf{1}_{l,\tau'}2^{-6(l-k)}\cdot2^{-7l}\sum_{\substack{\tau<X2^{-m-1}\leq1 \\ k\leq m< l}}\bigg(\int_{\tau'}^1\frac{2^{5m}}{(\tau'')^3}\Big(E_m(\tau'')+I_m(\tau'')\Big)\mathbf{1}_{m,\tau''}d\tau''\bigg)d\tau',\\
        III&=\int_{\tau}^{X2^{-k-1}}2^{k}\sum_{l\geq k}\mathbf{1}_{l,\tau'}2^{-6(l-k)}\cdot2^{-7l}\sum_{\substack{\tau>X2^{-m-1} \\ m<l}}\bigg(\int_{\tau'}^1\frac{2^{5m}}{(\tau'')^3}\Big(E_m(\tau'')+I_m(\tau'')\Big)d\tau''\bigg)d\tau'.
    \end{align*}
    We first notice that in the inner integral of $I$ we have $\tau''\geq X2^{-m-1}>X2^{-k-1}.$ Thus, we have the bound:
    \[I\lesssim2^{-7k}\sum_{m=x}^{k-1}\bigg(\int_{X2^{-k-1}}^1\frac{2^{5m}}{(\tau'')^3}\Big(E_m(\tau'')+I_m(\tau'')\Big)\mathbf{1}_{m,\tau''}d\tau''\bigg)\lesssim\widetilde{E}_k(X2^{-k-1})+\widetilde{I}_k(X2^{-k-1}),\]
    where in the second inequality we used the definitions of $\widetilde{E}_k$ and $\widetilde{I}_k.$ Next, we have the bound:
    \begin{align*}
        II&\lesssim\int_{\tau}^{X2^{-k-1}}2^{k}\sum_{l\geq k}\mathbf{1}_{l,\tau'}2^{6k-13l}\sum_{\substack{\tau<X2^{-m-1}\leq1 \\ k\leq m< l}}\bigg(\Big(E_m(X2^{-m-1})+I_m(X2^{-m-1})\Big)\int_{\tau'}^1\frac{2^{5m}}{(\tau'')^3}\mathbf{1}_{m,\tau''}d\tau''\bigg)d\tau'\\
        &\lesssim\int_{\tau}^{X2^{-k-1}}2^{k}\sum_{l\geq k}\mathbf{1}_{l,\tau'}2^{6k-13l}\sum_{\substack{\tau<X2^{-m-1}\leq1 \\ k\leq m< l}}2^{7m}\Big(E_m(X2^{-m-1})+I_m(X2^{-m-1})\Big)d\tau'\\
        &\lesssim\sum_{\tau<X2^{-m-1}\leq X2^{-k-1}}2^{7m}\Big(E_m(X2^{-m-1})+I_m(X2^{-m-1})\Big)\cdot\sum_{l>m}2^{6k-13l}\\
        &\lesssim\sum_{\tau<X2^{-m-1}\leq X2^{-k-1}}2^{-6|m-k|}\Big(E_m(X2^{-m-1})+I_m(X2^{-m-1})\Big),
    \end{align*}
    where we used again $\tau''\geq X2^{-m-1}$ to compute the inner integral in the first line.

    Finally, we notice that the last error term $III$ must vanish as $\tau\rightarrow0,$ since the inner sum would be empty. We prove a bound consistent with this expectation:
    \begin{align*}
        III&\lesssim\sum_{\tau>X2^{-m-1}}\bigg(\int_{\tau}^1\frac{2^{5m}}{(\tau'')^3}\Big(E_m(\tau'')+I_m(\tau'')\Big)d\tau''\bigg)\cdot\int_{\tau}^{X2^{-k-1}}2^{k}\sum_{m<l}2^{-6(l-k)}\cdot2^{-7l}d\tau'\\
        &\lesssim2^{6k}\sum_{\tau>X2^{-m-1}}2^{-8m}\bigg(\int_{\tau}^1\frac{1}{(\tau')^3}\Big(E_m(\tau')+I_m(\tau')\Big)d\tau'\bigg)\\
        &\lesssim\int_{\tau}^1\Big(\tau'\big\|F_M'\big\|_{H^{1/2}}^2+(\tau')^3\big\|\xi\big\|_{H^{3/2}}^2+(\tau')^3\big\|\nabla_{\tau}\xi\big\|_{H^{1/2}}^2\Big)d\tau'\cdot\tau^{-2}2^{6k}\sum_{\tau>X2^{-m-1}}2^{-8m}\\ 
        &\lesssim\tau^62^{6k}\bigg(\mathcal{D}_{II}+\int_{\tau}^1\tau'\big\|F_M'\big\|_{H^{1/2}}^2d\tau'\bigg)\lesssim\tau^62^{6k}\big(\mathcal{D}_{II}+\mathcal{F}_{II}(\tau)\big).
    \end{align*}
    Collecting all the bounds proved for the error terms, we conclude the proof.
\end{proof}

\bibliographystyle{amsalpha}
\bibliography{refs}

\end{document}